\documentclass[a4paper]{amsart}

\calclayout

\usepackage{microtype}
\usepackage{amsmath,amssymb, accents}
\usepackage{xstring}
\usepackage{xcolor}

\usepackage[mathscr]{euscript}

\usepackage{url}
\usepackage{latexsym}
\usepackage{amsthm}
\usepackage[latin1]{inputenc}
\usepackage{enumitem}
\usepackage[colorlinks=true, linkcolor={blue!50!black}, pdfhighlight=/O, 
ocgcolorlinks=true]{hyperref}
\hypersetup{bookmarksdepth=2}
\usepackage{graphicx}
\usepackage{color}
\usepackage{tikz-cd}

\numberwithin{equation}{section}
\theoremstyle{plain}
\newtheorem{theorem}[equation]{Theorem}

\newtheorem{lemma}[equation]{Lemma}

\newtheorem{proposition}[equation]{Proposition}

\newtheorem{corollary}[equation]{Corollary}

\newtheorem{introtheorem}{Theorem}

\theoremstyle{definition}

\newtheorem{definition}[equation]{Definition}
\newtheorem{variant}[equation]{Variant}
\newtheorem{construction}[equation]{Construction}
\newtheorem{example}[equation]{Example}

\newtheorem{remark}[equation]{Remark}
\newtheorem{warning}[equation]{Warning}
\newtheorem*{nwarning}{Warning}
\newtheorem{notation}[equation]{Notation}

\setlist[itemize]{itemsep=3pt, topsep=3pt, labelindent=1ex, itemindent=0mm, labelsep=1.2ex, leftmargin=*}
\setlist[enumerate]{itemsep=3pt, topsep=3pt, labelindent=1ex, itemindent=0mm, labelsep=1.2ex, leftmargin=*, label=(\arabic*)}

\newcommand{\mf}[1]{\mathfrak{#1}}

\newcommand{\ul}[1]{\underline{#1}}

\newcommand{\mb}[1]{\mathbf{#1}}
\newcommand{\mm}[1]{\mathrm{#1}}

\newcommand{\cat}[1]{
\StrLen{#1}[\mystrlen]
\ifnum\mystrlen=1 \mathscr{#1}
\else \mathrm{#1}
\fi} 
\newcommand{\scat}[1]{\mathsf{#1}} 
\newcommand{\modcat}[1]{\mathsf{#1}} 
\newcommand{\hcat}[1]{\mb{#1}} 
\newcommand{\dcat}[1]{\mathbb{#1}}

\newcommand{\sSet}[0]{\cat{sSet}}
\newcommand{\sS}[0]{\cat{S}}
\newcommand{\Span}[0]{\cat{Span}}

\newcommand{\colim}{\operatornamewithlimits{\mathrm{colim}}}

\newcommand{\rto}[1]{\stackrel{#1}{\rt}}

\newcommand{\rt}[0]{\longrightarrow}
\newcommand{\lt}[0]{\longleftarrow} 

\newcommand{\hooklongrightarrow}{\lhook\joinrel\longrightarrow}

\newcommand{\ho}[0]{\mathrm{ho}}
\newcommand{\core}[0]{\mf{c}}

\newcommand{\Map}[0]{\mm{Map}}
\newcommand{\Fun}[0]{\cat{Fun}}
\newcommand{\Aut}[0]{\mm{Aut}}
\newcommand{\rN}[0]{\mm{N}}

\newcommand{\Del}[0]{\mb{\Delta}}
\newcommand{\op}[0]{\mm{op}}
\newcommand{\cart}[0]{\mm{cart}}
\newcommand{\Seg}[0]{\mm{Seg}}
\newcommand{\enr}[0]{\mm{enr}}

\newcommand{\Ar}[0]{\mathbb{A}\mm{r}}
\newcommand{\Tw}[0]{\mathbb{T}\mm{w}}

\newcommand{\Cat}[0]{\cat{Cat}}
\newcommand{\dCat}[0]{\cat{Cat}_d}
\newcommand{\dfCat}[0]{\cat{Cat}^{\otimes d}}
\newcommand{\DCat}[0]{\cat{Cat}^{\otimes 2}}
\newcommand{\bisAn}[0]{\cat{bis}\sS}
\newcommand{\CAT}[0]{\cat{CAT}}
\newcommand{\SegCoPSh}[0]{\cat{coPSh}^{\mm{Seg}}}
\newcommand{\Cocart}[0]{\cat{Cocart}}

\newcommand{\SegS}[0]{\mathsf{CatAlg}}
\newcommand{\CompSegS}[0]{\mathsf{Cat}}
\newcommand{\SegCoPShproj}[0]{\mathsf{coPSh}^{\mm{Seg, proj}}}
\newcommand{\SegCoPShinj}[0]{\mathsf{coPSh}^{\mm{Seg, inj}}}
\newcommand{\EnrFun}[0]{\mathsf{EnrFun}}

\newcommand{\vertcolor}{red!70!black}
\newcommand{\vto}[0]{\begin{tikzcd}[cramped, ampersand replacement=\&, column sep=1.3pc] {}\arrow[r, \vertcolor] \& {}\end{tikzcd}}

\title{On straightening for Segal spaces}
\author{Joost Nuiten}
\email{joost.nuiten@math.univ-toulouse.fr}
\address{IMT, Universit\'e Toulouse 3 Paul Sabatier, France.}

\begin{document}

\begin{abstract}
The straightening-unstraightening correspondence of Grothendieck and Lurie \hbox{provides an equivalence} between cocartesian fibrations between $(\infty, 1)$-categories and diagrams of $(\infty, 1)$-categories. We provide an alternative proof of this correspondence, as well as an extension of straightening-unstraightening to all higher categorical dimensions. This is based on an explicit combinatorial result relating two types of fibrations between double categories, which can be applied inductively to construct the straightening of a cocartesian fibration between higher categories.
\end{abstract}

\maketitle

\setcounter{tocdepth}{1}
\tableofcontents

\section{Introduction}
The Grothendieck construction, or \emph{unstraightening}, is a standard procedure in category theory that associates to a diagram of categories $F\colon \cat{C}\rt \Cat$ a map of categories $\pi\colon \mm{Un}(F)\rt \cat{C}$ whose fibers are precisely the values of the diagram $F$. This maneuver does not result in a loss of information: the unstraightening functor
$$\begin{tikzcd}
\mm{Un}\colon \Fun(\cat{C}, \Cat)\arrow[r] & \Cat/\cat{C}
\end{tikzcd}$$
is an embedding, whose image is the subcategory of \emph{cocartesian fibrations} over $\cat{C}$ and maps between them that preserve cocartesian arrows \cite{sga1}. 

The practical consequences of this result are twofold. On the one hand, the Grothendieck construction also provides an equivalence between cocartesian fibrations over $\cat{C}$ and pseudofunctors $F\colon \cat{C}\rt \Cat$. From this point of view, cocartesian fibrations provide a convenient way to encode coherent diagrams of categories, which can be rectified to strict diagrams via the inverse of unstraightening, aptly called \emph{straightening}. Consequently, the language of fibrations starts to play an essential role in the homotopy coherent setting of $(\infty, 1)$-categories, where a version of straightening and unstraightening is available as well \cite{lur09} (see also \cite{rie18}). Indeed, many homotopy coherent diagrams of $(\infty, 1)$-categories are most naturally described in terms of cocartesian fibrations.

On the other hand, the unstraightening of a diagram $F\colon \cat{C}\rt \Cat$ is convenient to describe ``lax constructions'' with categories, stemming from the fact that $\mm{Un}(F)$ is the oplax colimit of $F$ \cite{gep17}. For example, a lax natural transformation between two diagrams $F$ and $G$ can be described by a map $\mm{Un}(F)\rt \mm{Un}(G)$ over $\cat{C}$ that need not preserve cocartesian arrows (see e.g.\ \cite{gag20, hau21} for a discussion in the $\infty$-categorical setting).

\medskip

The purpose of this text is to establish a version of straightening and unstraightening for $(\infty, d)$-categories. Recall that at the moment, the theory of $(\infty, d)$-categories is only in the first stages of its development, in particular when compared to $(\infty, 1)$-category theory; it is already not straightforward to compare the various different homotopy-theoretic models for $(\infty, d)$-categories (see e.g.\ \cite{sim12} for a textbook account and \cite{bar21} for a more recent state of affairs). Nonetheless, $(\infty, d)$-categories have been of interest in various areas of mathematics, also outside of the realm of category theory itself. For example, higher categories play a role in (derived) algebraic geometry via the six functor formalism ($d=2$) \cite{gai16, mac21} and other constructions involving correspondences \cite{cal21}, while in topology they arise in factorization homology \cite{aya18, sch14} and notably in topological field theories \cite{cal19, lur10}. 

Given the central place of the (un)straightening correspondence in $(\infty, 1)$-category theory, it seems fair to say that the lack of a higher-categorical version of this correspondence has been one of the main problems in the theory of $(\infty, d)$-categories (henceforth simply referred to as \emph{$d$-categories}). For instance, the unstraightening procedure could be useful in the study of lax natural transformations between diagrams of $d$-categories and in the study of (lax) limits and colimits. Our main results establish such a procedure and show that it is essentially unique.

To this end, we define for each $(d+1)$-category $\cat{C}$ two certain (non-full) subcategories
$$
\Cocart_d(\cat{C})\subseteq \Cat_{d+1}/\cat{C}\qquad \qquad \cat{Cart}_d(\cat{C})\subseteq \Cat_{d+1}/\cat{C}
$$
whose objects we refer to as \emph{$d$-cocartesian fibrations} and \emph{$d$-cartesian fibrations}; the two notions are equivalent under reversing the directions of morphisms in every categorical dimension. 
The fibers of a $d$-cocartesian fibration are $d$-categories; in particular, a $0$-cocartesian fibration between $1$-categories is simply a left fibration between $1$-categories \cite{lur09}. Furthermore, a $1$-cocartesian fibration between $1$-categories, viewed as $2$-categories with only invertible $2$-morphisms, is simply a cocartesian fibration in the usual sense. 

More generally, a map of $(d+1)$-categories $\pi\colon \cat{D}\rt \cat{C}$ is roughly said to be a $d$-cocartesian fibration if it induces $(d-1)$-cartesian fibrations on mapping $d$-categories and if every $1$-morphism in $\cat{C}$ admits enough cocartesian lifts (in an enriched sense, see Section \ref{sec:cocart fib} for more details). For $d=2$, the notion of a $2$-cocartesian fibration between $(2, 2)$-categories has also been considered (in a different variance) by Hermida \cite{her99} (see also \cite{bak11, buc14}). More recently, $2$-cocartesian fibrations between $(\infty, 2)$-categories have been described in terms of inner cocartesian $2$-fibrations between scaled simplicial sets \cite{gag21}.

Our main result is then the following:
\begin{introtheorem}[Theorem \ref{thm:unstraightening}]\label{introthm:A}
For each $(d+1)$-category $\cat{C}$, there exist equivalences of $1$-categories (here $\core_1$ takes the $1$-category underlying a $(d+1)$-category)
$$\begin{tikzcd}[row sep=0.1pc]
\mm{Un}\colon \core_1\Fun_{d+1}(\cat{C}, \hcat{Cat}_d)\arrow[r, yshift=1ex, "\sim"{swap}]& \Cocart_d(\cat{C})\colon \mm{St}\arrow[l, yshift=-1ex] \\
\mm{Un}\colon \core_1\Fun_{d+1}(\cat{C}^{\op}, \hcat{Cat}_d)\arrow[r, yshift=1ex, "\sim"{swap}]& \cat{Cart}_d(\cat{C})\colon \mm{St}\arrow[l, yshift=-1ex]
\end{tikzcd}$$
between $d$-cocartesian fibrations over $\cat{C}$ and $(d+1)$-functors $\cat{C}\rt \hcat{Cat}_d$ to the $(d+1)$-category of $d$-categories, resp.\ between $d$-cartesian fibrations and $(d+1)$-functors $\cat{C}^{\op}\rt \hcat{Cat}_d$. 
\end{introtheorem}
Our proof of Theorem \ref{introthm:A} \emph{only relies on the straightening equivalence for left} (i.e.\ $0$-cocartesian) \emph{fibrations}, which we recall in Section \ref{sec:copresheaves}. In particular, it provides an independent proof of the straightening-unstraightening equivalence for cocartesian fibrations of Lurie \cite{lur09}. In fact, we will use an inductive argument to deduce Theorem \ref{introthm:A} from a slight generalization of the $1$-categorical straightening equivalence, Theorem \ref{introthm:B} below.

To facilitate this inductive argument, we make use of the model for $(d+1)$-categories given by iterated complete Segal spaces \cite{bar21}. Recall that these are certain types of $(d+1)$-fold simplicial spaces 
$$
\cat{C}\colon \Del^{\times d+1, \op}\rt \sS.
$$
Here the $d$-fold simplicial space $\cat{C}(0)=\cat{C}(0, -)$ is constant and corresponds to the space of objects of $\cat{C}$, while $\cat{C}(1)$ is itself an iterated complete Segal space, modeling the $d$-category of arrows in $\cat{C}$.

This Segal space model has the benefit that a diagram of $d$-categories indexed by a $(d+1)$-category $\cat{C}$ can be encoded explicitly by a map of $(d+1)$-fold simplicial spaces $X\rt \cat{C}$ satisfying the following two conditions \cite{boa18, ras21}:
\begin{enumerate}
\item $X(0)$ is a $d$-category.
\item The Segal maps $X(n)\rt \cat{C}(n)\times_{\cat{C}(\{0\})} X(\{0\})
$ are equivalences of $d$-categories.
\end{enumerate}
We will refer to such maps as \emph{Segal copresheaves}, because they should be viewed as encoding a homotopy coherent action of the $(d+1)$-category $\cat{C}$ on the $d$-category $X(0)$. In particular, fixing the last $d$ simplicial coordinates yields a left fibration between $1$-categories. Work of Boavida \cite{boa18} gives a way to straighten each of these left fibrations, resulting in an equivalence between the $1$-category of Segal copresheaves over $\cat{C}$ and the $1$-category of $(d+1)$-functors $\cat{C}\rt \hcat{Cat}_d$ (see also \cite{ras21} and Section \ref{sec:copresheaves}). This equivalence can be thought of as a rectification procedure relating an external (Segal-style) definition of copresheaves to their internal definition as functors $\cat{C}\rt \hcat{Cat}_d$ (cf.\ in particular Theorem \ref{thm:rectification}).

\begin{nwarning}
Let us emphasize that one should \emph{not} think of a Segal copresheaf $X\rt \cat{C}$ as some sort of cocartesian fibration: its domain is typically not a $(d+1)$-category (unless $d=0$) and the notion of a Segal copresheaf is therefore not defined internal to the theory of $(d+1)$-categories. In particular, note that the fibers of a Segal copresheaf are constant in the \emph{first} simplicial direction, while the fibers of a $d$-cocartesian fibration between $(d+1)$-categories are constant in the \emph{last} simplicial direction. Unfortunately, the terminology employed in \cite{boa18, ras21} seems to suggest otherwise. 
\end{nwarning}

The straightening-unstraightening equivalence from Theorem \ref{introthm:A} is then given by a $d$-step combinatorial procedure that passes from Segal copresheaves to $d$-cocartesian fibrations. In each step, we modify the behaviour of the map of $(d+1)$-simplicial spaces $X\rt \cat{C}$ only in two consecutive simplicial directions, using the following result:
\begin{introtheorem}[Theorem \ref{thm:main theorem}]\label{introthm:B}
Let $\dcat{C}$ be a double category, i.e.\ a bisimplicial space satisfying the complete Segal conditions in the two simplicial directions (referred to as horizontal and vertical). Then there is an equivalence of categories
$$\begin{tikzcd}
\Psi^{\perp}\colon \cat{Fib}^{\mm{left, cart}}(\dcat{C})\arrow[r, yshift=1ex] & \cat{Fib}^{\mm{cocart, right}}(\dcat{C})\colon \Psi^{\top}\arrow[l, yshift=-1ex, "\sim"{swap}]
\end{tikzcd}$$
between maps of double categories $p\colon \dcat{D}\rt \dcat{C}$ of the following two types:
\begin{enumerate}
\item $p$ is a left fibration in the horizontal and a cartesian fibration in the vertical direction.
\item $p$ is a cocartesian fibration in the horizontal and a right fibration in the vertical direction.
\end{enumerate}
Furthermore, the fibers of $\dcat{D}\rt \dcat{C}$ and $\Psi^\perp(\dcat{D})\rt \dcat{C}$ differ simply by exchanging the horizontal and vertical directions.
\end{introtheorem}
The definition of the equivalence $\Psi^\perp$ uses rather simple combinatorics, inspired by the classical Grothendieck construction. Indeed, recall that for $F\colon \cat{C}\rt \Cat$, an object in the Grothendieck construction is a tuple $(c, x)$ with $c\in \cat{C}$ and $x\in F(c)$, and a morphism is given by the composite of a cocartesian (``horizontal'') arrow $f\colon (c, x)\rt (c', f_!x)$, followed by a ``vertical" arrow in the fiber $F(c')$. Likewise, $\Psi^\perp(\dcat{D})$ is a double category with the same objects as $\dcat{D}$, but with a horizontal arrow consisting of a horizontal arrow in $\dcat{D}$ followed by a fiberwise vertical arrow in $\dcat{D}$. 

Theorem \ref{introthm:A} now follows from an inductive application of Theorem \ref{introthm:B}. The lowest dimensional case, the straightening-unstraightening correspondence between $2$-functors $\cat{C}\rt \Cat_1$ and $1$-cocartesian fibrations over a $2$-category $\cat{C}$, is simply a special case of Theorem \ref{introthm:B}: when the double category $\dcat{C}$ is a $2$-category, the domain and codomain of $\Psi^\perp$ consist precisely of Segal copresheaves and $1$-cocartesian fibrations, respectively. In particular, when applied to a diagram of $(1, 1)$-categories $F\colon \cat{C}\rt \Cat_{(1, 1)}$ indexed by a $(1, 1)$-category $\cat{C}$, one obtains the following two-step process. We first associate to $F$ a Segal copresheaf over $\cat{C}$, which corresponds informally to passing from strict to pseudo-functors, after which the functor $\Psi^\perp$ produces the classical Grothendieck construction.

\medskip

Note that Theorem \ref{introthm:A} is not complete: both $(d+1)$-functors $\cat{C}\rt \hcat{Cat}_d$ and $d$-cocartesian fibrations over $\cat{C}$ can naturally be organized into $(d+1)$-categories, rather than $1$-categories. In the latter case, the $(d+1)$-category $\hcat{Cocart}_d(\cat{C})$ of $d$-cocartesian fibrations can be realized as a subcategory of the $(d+2)$-category $\hcat{Cat}_{d+1}/\cat{C}$. Both of these (large) $(d+1)$-categories furthermore depend functorially on $\cat{C}$ by restriction and base change respectively; more precisely, they can each be organized into functors of $(d+2)$-categories $\hcat{Cat}_{d+1}^{\op}\rt \hcat{CAT}_{d+1}$. We then prove the following refinement of Theorem \ref{introthm:A}:
\begin{introtheorem}[Theorem \ref{thm:unstraightening very functorial}]\label{introthm:C}
There is a unique natural equivalence of $(d+2)$-functors $\hcat{Cat}_{d+1}^{\op}\rt \hcat{CAT}_{d+1}$ given at a $(d+1)$-category $\cat{C}$ by an equivalence of $(d+1)$-categories
$$\begin{tikzcd}
\mm{Un}\colon \Fun_{d+1}(\cat{C}, \hcat{Cat}_d)\arrow[r, yshift=1ex, "\sim"{swap}] & \hcat{Cocart}_d(\cat{C})\colon \mm{St}.\arrow[l, yshift=-1ex]
\end{tikzcd}$$
\end{introtheorem}
The statement of Theorem \ref{introthm:C} hides the following subtlety: it is in fact not so obvious how to write down the $(d+2)$-functor 
$$\begin{tikzcd}
\hcat{Cocart}_d(-)\colon \hcat{Cat}_{d+1}^{\op}\arrow[r] & \hcat{CAT}_{d+1}; \quad \cat{C}\arrow[r, mapsto] & \hcat{Cocart}_d(\cat{C}).
\end{tikzcd}$$
For example, notice that the slice $(d+2)$-categories $\hcat{Cat}_{d+1}/\cat{C}$ do not depend on $\cat{C}$ in a $2$-functorial way in general. As one may expect, it will instead be more convenient to organize all $(d+1)$-categories of $d$-cocartesian fibrations into a $(d+1)$-cartesian fibration $\pi\colon \hcat{Cocart}_d\rt \hcat{Cat}_{d+1}$. The above functor then arises as its straightening, using Theorem \ref{introthm:A} (one categorical dimension higher).

\medskip

Variants of Theorem \ref{introthm:A} and Theorem \ref{introthm:C} have appeared in the literature before. The straightening of $0$-cocartesian fibrations nowadays has various proofs \cite{boa18, heu15, lur09}. Theorem \ref{introthm:A} for $1$-cocartesian fibrations over $2$-categories is the content of Lurie's scaled straightening and unstraightening theorem \cite{gag20, lur09g}. Work of Buckley \cite{buc14} describes the Grothendieck construction for $2$-cocartesian fibrations between $(2, 2)$-categories as an equivalence of $3$-categories. Gaitsgory and Rozenblyum prove a strengthening of Theorem \ref{introthm:C} for $2$-cocartesian fibrations between $(\infty, 2)$-categories, which also allows for lax natural transformations \cite[Chapter 11, Theorem 1.1.8]{gai16}; this seems to rely on some unproven statements about $(\infty, 2)$-categories appearing in Chapter 10 of loc.\ cit., notably about the (conjectured) model for $(\infty, 2)$-categories in terms of lax squares. Ayala, Mazel-Gee and Rozenblyum have recently deduced this lax refinement of straightening for $(\infty, 2)$-categories from our results and have also established a version for locally $2$-cocartesian fibrations \cite{aya19}. Finally, let us mention that a version of Theorem \ref{introthm:C} has been hypothesized in \cite{mac21}, where it is used to prove a universal property for the $(\infty, 2)$-category of correspondences in an $(\infty, 1)$-category with fiber products.

\subsection*{Outline}
Let us give a brief overview of the contents of this paper.
\emph{Sections} \ref{sec:dcat} \emph{and} \ref{sec:reflecting} treat Theorem \ref{introthm:B} and contain our main combinatorial results: Section \ref{sec:dcat} discusses the fibrations of double categories appearing in Theorem \ref{introthm:B}, and defines the functor $\Psi^\perp$ on objects (Section \ref{sec:reflection pointwise}). In Section \ref{sec:reflecting}, we then show that $\Psi^\perp$ extends to an equivalence of categories (Theorem \ref{thm:main theorem}). This essentially already gives a proof of straightening and unstraightening for cocartesian fibrations of $1$-categories (see Corollary \ref{cor:straightening low degrees}).

\emph{Section} \ref{sec:higher cats} gives some background material on higher categories, mainly using the model of iterated complete Segal spaces. To produce examples of higher categories (e.g.\ of the $(d+1)$-category $\hcat{Cat}_d$ of $d$-categories), we also recall the comparison to the point-set model of (strict) enriched categories (Section \ref{sec:point set}). Section \ref{sec:copresheaves} reviews the results of Boavida \cite{boa18}, relating Segal copresheaves over $\cat{C}$ to functors $\cat{C}\rt \hcat{Cat}_d$. For our purposes, we will need to slightly rephrase the results from loc.\ cit.\ to make them ($(\infty, 1)$-)functorial in the base $\cat{C}$; this requires some technical results about cartesian fibrations in the setting of categories of fibrant objects, discussed in \emph{Appendix} \ref{sec:fib obj}.

In \emph{Section} \ref{sec:cocart fib}, we give the definitions of $d$-cocartesian fibrations and $d$-cartesian fibrations. The main result here (Proposition \ref{prop:cocart forms cart fibration}) asserts that $d$-cocartesian fibrations can themselves be organized into a $(d+1)$-cartesian fibration $\pi\colon \hcat{Cocart}_d\rt \hcat{Cat}_{d+1}$.

Finally, \emph{Section} \ref{sec:straightening} contains the proofs of our main results, Theorems \ref{introthm:A} and \ref{introthm:C}. We first prove a version of straightening-unstraightening at the level of $1$-categories (Theorem \ref{thm:unstraightening}). Using this, it follows that there is a $(d+2)$-functor $\hcat{Cocart}_d\colon \hcat{Cat}_{d+1}^{\op}\rt \hcat{CAT}_{d+1}$ sending each $(d+1)$-category $\cat{C}$ to the large $(d+1)$-category of $d$-cocartesian fibrations over $\cat{C}$. We then show that this functor is representable by $\hcat{Cat}_d$ (Theorem \ref{thm:unstraightening very functorial}), providing the existence of the natural equivalence of Theorem \ref{introthm:C}. The uniqueness of this natural equivalence follows from the results in Section \ref{sec:rigidity}.

\subsection*{Conventions and notation}
We phrase our results as much as possible in the language of $(\infty, 1)$-categories (with the notable exception of Sections \ref{sec:point set} and \ref{sec:copresheaves}, where we need to recourse to point-set models). This is mostly a matter of terminology: effectively, we only start from the $(\infty, 1)$-category $\sS$ of spaces instead of the Kan--Quillen model structure on simplicial sets to avoid having to impose fibrancy conditions at every step. In particular, we avoid any $\infty$-categorical machinery that (tacitly) relies on the straightening results of Lurie.

Throughout the text, we will omit all prefixes ``$\infty$-''. For instance, we abbreviate $(\infty, 1)$-categories simply to categories and $(\infty, d)$-categories to $d$-categories, referring to ordinary categories instead as \emph{$(1, 1)$-categories}. We write $\core_k\cat{C}$ for the \emph{$k$-core} of a higher category, i.e.\ the $k$-category obtained by removing non-invertible morphisms in dimensions $>k$.

A \emph{subspace} of a space is a $(-1)$-truncated map $S\hooklongrightarrow T$, i.e.\ an inclusion of path components. Likewise, by a \emph{subcategory} of a $d$-category we will mean a $(-1)$-truncated map $i\colon \cat{C}\hooklongrightarrow \cat{D}$; inductively, this means that $i$ induces a subspace inclusion on spaces of objects and a subcategory inclusion on mapping $(d-1)$-categories. (This is in fact closer to the classical notion of a \emph{replete} subcategory.) In particular, this means that specifying a subcategory of a $d$-category comes down to specifying a subcategory of its homotopy $(d, d)$-category.

Categories and $d$-categories will typically be denoted by $\cat{C}$, while $d$-fold categories (see Section \ref{sec:dcat}) will be denoted $\dcat{C}$. We have furthermore tried to differentiate these notationally from specific point-set models by enriched categories or relative categories: for example, a $(d+1)$-category $\cat{C}$ can be presented by an ordinary category $\scat{C}$ strictly enriched in the $(1, 1)$-category $\CompSegS_{d}$ of fibrant objects in the $d$-categorical model structure (cf.\ Section \ref{sec:point set}). In particular, the $(d+1)$-category $\hcat{Cat}_d$ of $d$-categories is modeled by $\CompSegS_{d}$ with its canonical self-enrichment (Definition \ref{def:d+1cat of dcats}); we write $\Cat_d=\core_1\hcat{Cat}_d$ for its $1$-core. Likewise, we will denote by $\hcat{CAT}_d$ and $\cat{CAT}_d$ the (higher) categories of large $d$-categories.

\subsection*{Acknowledgments}
This text finds its origins (via Proposition \ref{prop:orthocart}) in joint work with Rune Haugseng, Fabian Hebestreit and Sil Linskens; I thank them for several helpful discussions, especially about constructions with spans. I would also like to thank Damien Calaque, Harry Gindi, Yonatan Harpaz, Corina Keller, Claudia Scheimbauer and Pelle Steffens for various useful conversations. Finally, I am grateful to the anonymous referee for their useful comments that helped improve the paper. 
This project has received funding from the European Research Council (ERC) under the European Union's Horizon 2020 research and innovation programme (grant agreement No 768679).

\section{Fibrations of double categories}\label{sec:dcat}
In this section, we will introduce a certain type of fibration between double categories, behaving like a left fibration in one direction and like a cartesian fibration in the other. We will show that such a fibration has a `\emph{reflection}', which is a cocartesian fibration in one direction and a right fibration in the other.

\subsection{Recollections on double categories}
Recall that a simplicial space $X\colon \Del^{\op}\rt \sS$ is said to be a \emph{complete Segal space} if it satisfies the following two conditions \cite{rez01}:
\begin{enumerate}
\item \emph{Segal condition}: for every $n\geq 2$, the natural map
$$\begin{tikzcd}
X\big(n\big)\arrow[r] & X\big(\{0\leq 1\}\big)\times_{X(\{1\})} X\big(\{1\leq \dots \leq n\}\big)
\end{tikzcd}$$
is an equivalence. Equivalently, $X(n)\simeq X(1)\times_{X(0)} \dots \times_{X(0)} X(1)$ via the inclusions of the intervals $\{i\leq i+1\}$ into $[n]$.

\item\label{it:completeness} \emph{completeness}: consider the simplicial set $H=\Delta[3]\big/\sim$ obtained by collapsing the $1$-simplex $\{0\leq 2\}$ to a point and collapsing the $1$-simplex $\{1\leq 3\}$ to a (different) point. Then the map $X(0)\rt \Map_{s\sS}(H, X)$ is an equivalence.
\end{enumerate}
Recall that complete Segal spaces are a model for the homotopy theory of ($(\infty, 1)$-)categories; we will therefore also refer to complete Segal spaces simply as \emph{(1-)categories}. We will be interested in analogues of complete Segal spaces for diagrams of spaces indexed by multiple copies of the simplex category.
\begin{notation}\label{not:objects in product of Delta}
For any $k\geq 0$, we will write $\vec{n}_k=([n_1], \dots, [n_k])\in \Del^{\times k}$ for a generic (but fixed) $k$-tuple of objects in $\Del$. Furthermore, we will abbreviate $\vec{0}_k=([0], \dots, [0])$. 
For example, if $X\colon \Del^{\times d}\rt \sS$ is a $d$-fold simplicial space, then 
$$\begin{tikzcd}
X(\vec{n}_k, -, \vec{0}_{d-k-1})\colon \Del^{\op}\arrow[r] & \sS
\end{tikzcd}$$
denotes the simplicial space obtained by fixing the first $k$ inputs to be $\vec{n}_k$ and the last $d-k-1$ inputs to be $0$, varying only the $(k+1)$-st input.
Finally, we will write $X(n)\colon \Del^{\times d-1}\rt \sS$ for the $(d-1)$-fold simplicial space given by $X(n, -, \dots, -)$.
\end{notation}
\begin{definition}\label{def:dfcat}
A \emph{$d$-fold category} $\dcat{C}$ is a diagram $\dcat{C}\colon (\Del^{\op})^{\times d}\rt \sS$ such that for each $1\leq k\leq d$, all the simplicial spaces
$$\begin{tikzcd}
\dcat{C}\big(\vec{n}_{k-1}, -, \vec{n}_{d-k}\big)\colon \Del^{\op}\arrow[r] & \sS
\end{tikzcd}$$
are complete Segal spaces. When $d=2$, we will refer to $\dcat{C}$ as a \emph{double category}. We will denote by $\dfCat\subseteq \Fun\big((\Del^{\op})^{\times d}, \sS\big)$ the full subcategory spanned by the $d$-fold categories.
\end{definition}
\begin{warning}
Because we impose completeness conditions in all directions, this notion of a $d$-fold category is \emph{not quite} the $\infty$-categorical analogue of that of a \emph{$d$-uple category} in the sense of Ehresmann \cite{ehr63}, i.e.\ a (strict) category object in $(d-1)$-uple categories. The latter type of objects (which are sometimes also called $d$-fold categories) corresponds instead to the more general notion of a \emph{$d$-uple Segal space} \cite{hau14b}.

For instance, for a double category in the sense of Definition \ref{def:dfcat}, the completeness conditions in the horizontal and vertical direction imply that the horizontal equivalences and vertical equivalences coincide. Many examples of double (and more generally, $d$-uple) Segal spaces do not satisfy this condition: as an example, one can take the double Segal space of associative algebras, with horizontal morphisms given by bimodules and vertical morphisms given by algebra homomorphisms \cite{hau14b}.
\end{warning}
\begin{notation}[Opposites]\label{not:opposite}
Let $\dcat{C}$ be a $d$-fold category and $1\leq k\leq d$. We will write $\dcat{C}^{k-\op}$ for the $d$-fold category given by taking opposites in the $k$-th variable:
$$\begin{tikzcd}[column sep=3.2pc]
\dcat{C}^{k-\op}\colon \Del^{\times d}\arrow[rr, "\mm{id}^{\times k-1}\times \op\times \mm{id}^{\times d-k}"] & & \Del^{\times d}\arrow[r, "\dcat{C}"] & \sS.
\end{tikzcd}$$
When $k=1$, we will simply write $\dcat{C}^{\op}$.
If $\dcat{C}$ is a double category, we will furthermore write $\dcat{C}^{\mm{rev}}$ for the double category with its coordinates reflected, i.e.\ $\dcat{C}^\mm{rev}(m, n)=\dcat{C}(n, m)$.
\end{notation}

In the case of double categories, we will employ the following conventions: we will refer to the direction of the first copy of $\Del^{\op}$ as the \emph{horizontal} direction and to the second copy of $\Del^{\op}$ as the \emph{vertical} direction. Furthermore, we will use the notation $\rt$ for horizontal arrows and and $\vto$ for vertical arrows.
\begin{example}\label{ex:generating dcat}
Let $\cat{C}$ and $\cat{D}$ be categories. We denote by $\cat{C}\boxtimes \cat{D}$ the double category given by 
$$
\cat{C}\boxtimes \cat{D}\big([p], [q]\big)=\Map_{\cat{Cat}}\big([p], \cat{C}\big)\times \Map_{\cat{Cat}}\big([q], \cat{D}\big).
$$
In particular, we write $[m, n]=[m]\boxtimes [n]$ for the (generating) double categories that look like
$$\begin{tikzcd}[column sep=1.2pc, row sep=1.2pc]
00\arrow[r]\arrow[d, \vertcolor] & 10\arrow[r]\arrow[d, \vertcolor] & \dots\arrow[r]\arrow[d, \vertcolor] & m0\arrow[d, \vertcolor]\\
\vphantom{00}\dots \arrow[d, \vertcolor]\arrow[r]  & \dots\vphantom{00} \arrow[r]\arrow[d, \vertcolor] & \vphantom{00}\dots\arrow[r]\arrow[d, \vertcolor] & \vphantom{00}\dots\arrow[d, \vertcolor]\\
0n\arrow[r] & 1n\arrow[r] & \dots \arrow[r] & mn.
\end{tikzcd}$$
The squares in this picture should not be viewed as commuting squares, but rather as being filled with some type of non-invertible $2$-arrow; no confusion should arise since one cannot compose $\rt$ and $\vto$. 
\end{example}
\begin{example}\label{ex:arrow dcat}
Let $n\geq 0$ and consider the \emph{arrow double category} $\Ar[n]$ defined by
$$
\Map_{\DCat}\Big([p, q], \Ar[n]\Big)\simeq \Map_{\cat{Cat}}\Big([q]\star [p], [n]\Big).
$$
where $\star$ denotes the join. In other words, it is the double category that looks like
$$\begin{tikzcd}[column sep=1.2pc, row sep=1.2pc]
00\arrow[r] & 01\arrow[r]\arrow[d, \vertcolor] & \dots\vphantom{00} \arrow[r]\arrow[d, \vertcolor] & 0n\arrow[d, \vertcolor]\\
& 11\arrow[r] & \dots\vphantom{00}\arrow[d, \vertcolor] \arrow[r] & \dots\vphantom{00}\arrow[d, \vertcolor]\\
& & \dots \vphantom{00}\arrow[r] & \dots\vphantom{00}\arrow[d, \vertcolor] \\
& & & nn.
\end{tikzcd}$$
There is an evident map $\Ar[n]\rt [n, n]$, and we denote by $\pi_{\mm{hor}}\colon \Ar[n]\rt [n, 0]$ and $\pi_{\mm{vert}}\colon \Ar[n]\rt [0, n]$ the obvious projections.
\end{example}
\begin{example}\label{ex:twisted arrow dcat}
Let $n\geq 0$ and consider the double category $\Tw[n]$ defined by
$$
\Map_{\DCat}\Big([p, q], \Tw[n]\Big)\simeq \Map_{\cat{Cat}}\Big([q]^{\op}\star [p], [n]\Big).
$$
In other words, $\Tw[n]\simeq \Ar[n]^{2-\mm{op}}$ and can be depicted as
$$\begin{tikzcd}[column sep=1.2pc, row sep=1.2pc]
00\arrow[r] & 01\arrow[r]\arrow[d, \vertcolor, leftarrow] & \dots \arrow[r]\arrow[d, \vertcolor, leftarrow] & 0n\arrow[d, \vertcolor, leftarrow]\\
& 11\arrow[r] & \dots\vphantom{00}\arrow[d, \vertcolor, leftarrow] \arrow[r] & \dots\vphantom{00}\arrow[d, \vertcolor, leftarrow]\\
& & \dots \arrow[r] & \dots\vphantom{00}\arrow[d, \vertcolor, leftarrow] \\
& & & nn.
\end{tikzcd}$$
\end{example}
\begin{proposition}\label{prop:dcat cart closed}
For any $d\geq 1$, the category $\dfCat$ is cartesian closed. In fact, the localization 
$$\begin{tikzcd}
\Fun\big((\Del^{\op})^{\times d}, \sS\big) \arrow[r, yshift=0.7ex]\arrow[r, yshift=-0.7ex, hookleftarrow] & {\dfCat}
\end{tikzcd}$$ 
of $d$-fold simplicial spaces is monoidal for the cartesian product.
\end{proposition}
We will write $Y^X$ for the internal mapping object in  $d$-fold simplicial spaces; the proposition implies that this is also the internal mapping object for $d$-fold categories.
\begin{proof}
It suffices to verify that for each $d$-fold category $\dcat{C}$ and $m_1, \dots, m_d\geq 0$, the $d$-fold simplicial space given by
$$
\dcat{C}^{[m_1, \dots, m_d]}(n_1, \dots, n_d)=\Map\big([m_1, \dots, m_d]\times [n_1, \dots, n_d], \dcat{C}\big)
$$
remains a $d$-fold category. Using that 
$[m_1, \dots, m_d]=[m_1, 0, \dots, 0]\times\dots\times [0,\dots, 0, m_d]
$
together with symmetry, it suffices to verify that $\dcat{C}^{[m, 0, \dots, 0]}$ is a $d$-fold category, i.e.\ satisfies the complete Segal conditions in all simplicial directions.

For the first direction, note that each $\dcat{C}^{[m, 0, \dots, 0]}(-, \vec{n}_{d-1})$ agrees with $\dcat{C}(-, \vec{n}_{d-1})^{[m]}$, which is a $1$-category by \cite[Theorem 7.2]{rez01}. For the remaining simplicial directions, note that the $(d-1)$-fold simplicial space
$$\begin{tikzcd}
\dcat{C}^{[m, 0, \dots, 0]}(n, -)\colon (\Del^{\op})^{\times d-1}\arrow[r] & \sS
\end{tikzcd}$$
sends each $(n_2, \dots, n_d)$ to space of maps between simplicial spaces
\begin{align*}
\Map_{s\sS}\Big([m]\times [n], \dcat{C}(-, n_2, \dots, n_d)\Big) &\simeq \Map_{s\sS}\Big(\colim_{[k]\in \Del/[m]\times [n]} \ [k], \ \dcat{C}(-, n_2, \dots, n_d)\Big)\\
&\simeq \lim_{[k]\in (\Del/[m]\times [n])^{\op}} \dcat{C}(k, n_2, \dots, n_d).
\end{align*}
In other words, $\dcat{C}^{[m, 0, \dots, 0]}(n, -)$ is naturally equivalent to the limit of a certain diagram consisting of the $(d-1)$-fold categories $\dcat{C}(k, -)$. Since $\Cat^{\otimes d-1}\subseteq \Fun\big((\Del^{\op})^{\times d-1}, \sS\big)$ is closed under limits, each $\dcat{C}^{[m, 0, \dots, 0]}(n, -)$ is a $(d-1)$-fold category, as desired.
\end{proof}

\subsection{Fibrations of double categories}
Let $p\colon \cat{D}\rt \cat{C}$ be a map of $1$-categories. Recall that an arrow $\alpha\colon d_1\rt d_2$ in $\cat{D}$ is called \emph{p-cartesian} (or simply cartesian, if $p$ is understood) if for any diagram
$$\begin{tikzcd}
\Lambda^2[2]\arrow[r, "\sigma"]\arrow[d] & \cat{D}\arrow[d]\\
{[2]}\arrow[r]\arrow[ru, dotted] & \cat{C}
\end{tikzcd}$$
where the restriction of $\sigma$ to $[1<2]$ is given by $\alpha$, there is a contractible space of lifts, as indicated. Let $\cat{D}(1)^{\cart}\subseteq \cat{D}(1)$ denote the subspace consisting of cartesian arrows. Then $p$ is said to be a \emph{cartesian fibration} if the map $\cat{D}(1)^{\cart}\rt \cat{D}(\{1\})\times_{\cat{C}(\{1\})} \cat{C}(1)$ induces a surjection on path components \cite{aya17, lur09}. It is called a \emph{right fibration} if furthermore every arrow in $\cat{D}$ is cartesian. Equivalently, it is a right fibration if the map $\cat{D}(1)\rt \cat{D}(\{1\})\times_{\cat{C}(\{1\})} \cat{C}(1)$ is an equivalence, see e.g.\ \cite[Proposition 1.7]{boa18}.
Taking opposite categories, one obtains the notion of a \emph{cocartesian} arrow and a \emph{cocartesian} (resp.\ \emph{left}) \emph{fibration} (some alternative characterizations also appear in Section \ref{sec:pre-cocart}).

We will be interested in the following variant of this notion for double categories:
\begin{definition}\label{def:fibrations}
A map $\dcat{D}\rt \dcat{C}$ of double categories is called a \emph{(left, cart)-fibration} if the following conditions hold:
\begin{enumerate}
\item\label{it:hor left} The square of categories
$$\begin{tikzcd}
\dcat{D}([1], -)\arrow[r]\arrow[d] & \dcat{D}(\{0\}, -)\arrow[d]\\
\dcat{C}([1], -)\arrow[r] & \dcat{C}(\{0\}, -)
\end{tikzcd}$$
is cartesian. Equivalently, each $\dcat{D}(-, n)\rt \dcat{C}(-, n)$ is a left fibration of categories. 

\item\label{it:vert cart} $\dcat{D}(0, -)\rt \dcat{C}(0, -)$ is a cartesian fibration of categories.
\end{enumerate}
A \emph{strong} map between (left, cart)-fibrations is a commuting square
$$\begin{tikzcd}
\dcat{D}\arrow[r, "f"]\arrow[d, "p"{swap}] & \dcat{D}'\arrow[d, "p'"]\\
\dcat{C}\arrow[r] & \dcat{C}'
\end{tikzcd}$$
where $f$ preserves cartesian vertical arrows. We will write $\cat{Fib}^{\mm{left, cart}}\subseteq \Fun([1], \DCat)$ for the subcategory of (left, cart)-fibrations and strong maps between them.
\end{definition}
\begin{remark}\label{rem:behaviour on squares}
By condition \ref{it:hor left}, the map $\cat{D}([1], -)\rt \cat{C}([1], -)$ is a cartesian fibration as well and $\cat{D}([1], -)\rt \cat{D}(\{0\}, -)$ preserves cartesian arrows. The other face map $\cat{D}([1], -)\rt \cat{D}(\{1\}, -)$ is \emph{not} required to preserve cartesian arrows.
\end{remark}
\begin{remark}\label{rem:basic properties of leftcart fibs}
The class of (left, cart)-fibrations is stable under composition and base change. Furthermore, a map $\dcat{D}\rt \ast$ is a (left, cart)-fibration if and only if $\dcat{D}$ is constant in the horizontal direction. In particular, the fibers of a (left, cart)-fibration are categories concentrated in the vertical direction (and constant in the horizontal direction).
\end{remark}
Informally, a map of double categories $p\colon \dcat{D}\rt \dcat{C}$ is a (left, cart)-fibration if the following conditions hold:
\begin{enumerate}[label=(\alph*)]
\item For every horizontal arrow $\alpha\colon c_0\to c_1$ in $\dcat{C}$ and $d_0\in p^{-1}(c_0)$, there exists a unique horizontal arrow $d_0\rt \alpha_!(d_0)$ lifting $\alpha$. Furthermore, given any square $\sigma$ in $\dcat{C}$ together with a lift of its left vertical arrow
$$\begin{tikzcd}
d_{00}\arrow[r, dotted]\arrow[d, \vertcolor]\arrow[rd, phantom, "\tilde{\sigma}"] & \alpha_!(d_{00})\arrow[d, \vertcolor, dotted]\\
d_{01}\arrow[r, dotted] & \alpha'_!(d_{01})
\end{tikzcd}\quad\longmapsto\quad \begin{tikzcd}
c_{00}\arrow[r, "\alpha"]\arrow[d, \vertcolor]\arrow[rd, phantom, "\sigma"] & c_{10}\arrow[d, \vertcolor]\\
c_{01}\arrow[r, "\alpha'"{swap}] & c_{11}
\end{tikzcd}$$ 
there exists a unique dotted lift $\tilde{\sigma}$ as indicated (this lift includes the objects $\alpha_!(d_{00})$ and $\alpha'_!(d_{01})$). 

\item For every vertical arrow $\beta\colon c_0\vto c_1$ in $\dcat{D}$ and $d_1\in p^{-1}(c_1)$, there exists a vertically $p$-cartesian arrow $\beta^*(d_1)\vto d_1$ covering $\alpha$.
\end{enumerate}
Let us now fix a square $\sigma$ in $\dcat{C}$ together with an object $d_{01}\in p^{-1}(c_{01})$. A combination of (a) and (b) then implies that $\sigma$ admits a unique lift $\tilde{\sigma}$ of the form
$$\begin{tikzcd}
\beta^*(d_{01})\arrow[r, dotted]\arrow[d, \vertcolor, dotted]\arrow[rd, phantom, "\tilde{\sigma}"] & \alpha_!\beta^*(d_{01})\arrow[d, \vertcolor, dotted]\\
d_{01}\arrow[r, dotted] & \alpha'_!(d_{01})
\end{tikzcd}\quad\longmapsto\quad \begin{tikzcd}
c_{00}\arrow[r, "\alpha"]\arrow[d, \vertcolor, "\beta"{swap}]\arrow[rd, phantom, "\sigma"] & c_{10}\arrow[d, \vertcolor, "\beta'"]\\
c_{01}\arrow[r, "\alpha'"{swap}] & c_{11}
\end{tikzcd}$$ 
which is vertically $p$-cartesian, in the sense that any other square factors uniquely over $\tilde{\sigma}$ in the vertical direction (Remark \ref{rem:behaviour on squares}). The resulting vertical arrow $\alpha_!\beta^*(d_{01})\rt \alpha'_!(d_{01})$ is not necessarily $p$-cartesian and decomposes as
$$\begin{tikzcd}
\alpha_!\beta^*(d_{01})\arrow[r, \vertcolor, "h_{\sigma}"] & \beta'^*\alpha'_!(d_{01})\arrow[r, \vertcolor] & \alpha'_!(d_{01}).
\end{tikzcd}$$
The first map $h_{\sigma}$ is contained in the fiber $p^{-1}(c_{10})$ and provides a natural comparison between the two ways to change fibers over the square $\sigma$, from ``horizontally after vertically'' to ``vertically after horizontally''. 

To illustrate Definition \ref{def:fibrations}, let us consider the following example:
\begin{proposition}\label{prop:orthocart}
Let $\cat{C}$ and $\cat{D}$ be categories and let $p=(p_1, p_2)\colon \cat{X}\rt \cat{C}\times \cat{D}$ be a functor such that:
\begin{enumerate}[label=(\roman*)]
\item\label{it:ortho1} $p_1$ is a cocartesian fibration and all $p_1$-cocartesian arrows in $\cat{X}$ map to equivalences in $\cat{D}$.

\item\label{it:ortho2} $p_2$ is a cartesian fibration and all $p_2$-cartesian arrows in $\cat{X}$ map to equivalences in $\cat{C}$.
\end{enumerate} 
Associated to $p$ is a map of double categories $p'\colon \dcat{X}\rt \cat{C}\boxtimes\cat{D}$, where $\dcat{X}(m, n)$ is the space of diagrams
\begin{equation}\label{diag:orthofib}\begin{tikzcd}
{[m]\times [n]}\arrow[r, "\sigma"]\arrow[rd, "\alpha\times \beta"{swap}] & \cat{X}\arrow[d, "p"]\\
& \cat{C}\times \cat{D}
\end{tikzcd}\end{equation}
where $\alpha\colon [m]\rt \cat{C}, \beta\colon [n]\rt \cat{D}$ and $\sigma$ sends each $[m]\times \{j\}$ to $p_1$-cocartesian arrows in $\cat{X}$. Then $p'$ is a (left, cart)-fibration of double categories.
\end{proposition}
Functors $p\colon \cat{X}\rt \cat{C}\times \cat{D}$ satisfying conditions \ref{it:ortho1} and \ref{it:ortho2} have been studied in detail in \cite{heb20, hau21}, where they are called \emph{local orthofibrations}. Let us point out the close analogy between the discussion preceding Proposition \ref{prop:orthocart} and the discussion in \cite[Construction 2.3.5]{hau21}.
\begin{proof}
To see condition \ref{it:hor left}, consider the pullback square of categories
$$\begin{tikzcd}
\Fun'([n], \cat{X})\arrow[r, "q"]\arrow[d] & \cat{C}\times \Map_{\Cat}([n], \cat{D})\arrow[d, "\Delta\times \iota"]\\
\Fun([n], \cat{X})\arrow[r] & \Fun([n], \cat{C})\times \Fun([n], \cat{D}) 
\end{tikzcd}$$
where $\Delta$ takes the constant diagram and $\iota$ is the inclusion of the space of objects. Since $(p_1, p_2)$ is a local orthofibration, the induced map 
$$(p'_1, p'_2)\colon \Fun([n], \cat{X})\rt \Fun([n], \cat{C})\times \Fun([n], \cat{D})
$$
is a local orthofibration as well: a natural transformation in $\Fun([n], \cat{X})$ is $p'_1$-cocartesian if and only if it is given pointwise by $p_1$-cocartesian maps in $\cat{X}$, and similarly for $p'_2$-cartesian natural transformations (this follows e.g.\ from Lemma \ref{lem:pre-cocart stable}(3)). In particular, the base change $q\colon \Fun'([n], \cat{X})\rt \cat{C}\times \Map_{\Cat}([n], \cat{D})$ is a cocartesian fibration. 

Unraveling the definitions, one sees that the space of maps $[m]\rt \Fun'([n], \cat{X})$ is equivalent to the space of diagrams $\sigma\colon [m]\times [n]\rt \cat{X}$ of the form \eqref{diag:orthofib}. Furthermore, $[m]\rt \Fun'([n], \cat{X})$ sends each arrow to a $q$-cocartesian arrow if and only if $\sigma$ sends each $[m]\times \{i\}$ to $p_1$-cocartesian arrows in $\cat{X}$. It follows that there is a factorization
$$\begin{tikzcd}
\dcat{X}(-, n)\arrow[r, hookrightarrow] & \Fun'([n], \cat{X})\arrow[r, "q"] & \cat{C}\times \Map([n], \cat{D})=\big(\cat{C}\boxtimes \cat{D}\big)(-, n)
\end{tikzcd}$$
where the first map is the inclusion of the wide subcategory spanned by the $q$-cocartesian arrows. In particular, this shows that each $\dcat{X}(-, n)\rt \big(\cat{C}\boxtimes \cat{D}\big)(-, n)$ is a left fibration, as desired.

For condition \ref{it:vert cart}, it suffices to observe that the map $\dcat{X}(0, -)\rt \big(\cat{C}\boxtimes \cat{D}\big)(0, -)$ coincides with top horizontal map in the pullback diagram
$$\begin{tikzcd}
\dcat{X}(0, -)\arrow[r]\arrow[d] & \core_0(\cat{C})\times \cat{D}\arrow[d]\\
\cat{X}\arrow[r] & \cat{C}\times \cat{D}.
\end{tikzcd}$$
This map is a cartesian fibration by condition \ref{it:ortho2} (cf.\ \cite[Observation 2.3.2]{hau21}).
\end{proof}

\begin{remark}\label{rem:(left,cart)-straightened}
This remark only serves to motivate Definition \ref{def:fibrations}, and will not be used in the rest of the text (making it more precise would require a form of straightening). 

Elaborating on the discussion following Definition \ref{def:fibrations}, one can informally think of a (left, cart)-fibration $p\colon \dcat{D}\rt \dcat{C}$ as a certain type of diagram of categories indexed by $\dcat{C}$. Indeed, by Remark \ref{rem:basic properties of leftcart fibs}, the fibers of $p$ are all categories (in the vertical direction). Furthermore, each horizontal arrow $\alpha\colon c_0\rt c_1$ induces a map of categories $\alpha_!\colon\dcat{D}_{c_0}\rt \dcat{D}_{c_1}$ and each vertical arrow $\beta\colon c_0\vto c_1$ induces a map of categories $\beta^*\colon \dcat{D}_{c_1}\rt \dcat{D}_{c_0}$. Finally, each diagram $\sigma\colon [1, 1]\rt \dcat{C}$ gives rise to a square of categories and functors commuting up to a natural transformation
$$\begin{tikzcd}
\dcat{D}_{c_{00}}\arrow[r, "\alpha_!"] \arrow[rd, Rightarrow, "h_\sigma", start anchor={[xshift=2pt, yshift=-2pt]}, end anchor={[xshift=-2pt, yshift=2pt]}] & \dcat{D}_{c_{10}}\\
\dcat{D}_{c_{01}}\arrow[r, "\alpha'_!"{swap}]\arrow[u, "\beta^*"] & \dcat{D}_{c_{11}}.\arrow[u, "\beta'^*"{swap}]
\end{tikzcd}$$
In other words, a (left, cart)-fibration $p\colon \dcat{D}\rt \dcat{C}$ should arise as the unstraightening of a diagram of double categories $\dcat{C}^{2-\op}\rt\mathbb{C}\cat{at}^{\mm{oplax}}$, where $\mathbb{C}\cat{at}^{\mm{oplax}}$ is a certain double category whose objects are categories, horizontal and vertical morphisms are functors and squares are oplax squares
$$\begin{tikzcd}
\cat{D}_{00}\arrow[d, \vertcolor]\arrow[r] & \cat{D}_{01}\arrow[d, \vertcolor]\\
\cat{D}_{10}\arrow[r]\arrow[ru, Rightarrow, start anchor={[xshift=2pt, yshift=2pt]}, end anchor={[xshift=-2pt, yshift=-2pt]}] & \cat{D}_{11}.
\end{tikzcd}$$
\end{remark}
\begin{lemma}\label{lem:equivalence fiberwise}
Consider a strong map of (left, cart)-fibrations over $\dcat{C}$
$$\begin{tikzcd}
\dcat{D}\arrow[rr, "f"]\arrow[rd, "p"{swap}] & & \dcat{E}\arrow[ld, "q"]\\
& \dcat{C}.
\end{tikzcd}$$
Then $f$ is an equivalence if and only if for any object $c\in \dcat{C}$, the induced map between the fibers $\dcat{D}_c\rt \dcat{E}_c$ is an equivalence of (vertical) categories.
\end{lemma}
\begin{proof}
Since $p$ and $q$ are left fibrations in the horizontal direction, the map $f$ is an equivalence if and only if the map $f_0\colon \dcat{D}(0, -)\rt \dcat{E}(0, -)$ is an equivalence of cartesian fibrations over $\dcat{C}(0, -)$; in turn, the map of cartesian fibrations $f_0$ is an equivalence if and only if it induces an equivalence on fibers \cite[Corollary 2.4.4.4]{lur09}.
\end{proof}
\begin{variant}\label{var:fibrations}
Taking opposites in the horizontal or vertical direction, and exchanging the vertical and horizontal direction, one obtains seven more types of fibrations. For example, a map $\dcat{D}\rt \dcat{C}$ is a (cocart, right)-fibration if $\dcat{D}(-, 0)\rt \dcat{C}(-, 0)$ is a cocartesian fibration and each $\dcat{D}(m, -)\rt \dcat{C}(m, -)$ is a right fibration. Strong maps between such fibrations are always required to preserve (co)cartesian arrows, either in the horizontal or vertical direction.
\end{variant}
For example, one can also associate a (cocart, right)-fibration to a two-variable fibration $(p_1, p_2)\colon \cat{X}\rt \cat{C}\times \cat{D}$ as in Proposition \ref{prop:orthocart}, by considering maps $\sigma$ as in \eqref{diag:orthofib} that send each $\{i\}\times [n]$ to $p_2$-cartesian arrows in $\cat{X}$.
\begin{remark}\label{rem:(cocart,right)-straightened}
Similarly to Remark \ref{rem:(left,cart)-straightened}, one can informally think of a (cocart, right)-fibration $p\colon \dcat{D}\rt \dcat{C}$ as a certain diagram of categories indexed by $\dcat{C}$. Indeed, in this case all fibers of $p$ are categories in the \emph{horizontal} direction (and constant in the vertical direction) and again each horizontal arrow $\alpha\colon c_0\rt c_1$ induces a map of categories $\alpha_!\colon \dcat{D}_{c_0}\rt \dcat{D}_{c_1}$ and each vertical arrow $\beta\colon c_0\vto c_1$ induces a functor $\beta^*\colon \dcat{D}_{c_1}\rt \dcat{D}_{c_0}$.

For the behaviour on a square $\sigma\colon [1, 1]\rt \dcat{C}$, we can repeat the analysis following Definition \ref{def:fibrations}: for each such square and a lift $d_{10}\in p^{-1}(c_{10})$, there exists a lift to a square 
$$\begin{tikzcd}
\beta^*(d_{01})\arrow[d, \vertcolor]\arrow[r] & \alpha_!\beta^*(d_{01})\arrow[d, \vertcolor]\\
d_{10}\arrow[r] & \alpha_!(d_{10})
\end{tikzcd}$$
in $\dcat{D}$ which is $p$-cocartesian in the horizontal direction. The left vertical arrow is $p$-cartesian but the right vertical arrow is not in general, so that one again obtains a natural map $h_\sigma\colon \alpha_!\beta^*(d_{01})\rt \beta'^*\alpha_!(d_{01})$. 

In other words, one can think of a (cocart, right)-fibration $p\colon \dcat{D}\rt \dcat{C}$ as the unstraightening of a map of double categories $\dcat{C}^{2-\op}\rt \mathbb{C}\cat{at}^{\mm{oplax}}$, where $\mathbb{C}\cat{at}^{\mm{oplax}}$ is the double category described informally in Remark \ref{rem:(left,cart)-straightened}.
\end{remark}

\subsection{The reflection of a (left, cart)-fibration}\label{sec:reflection pointwise}
We will now turn to the main combinatorial construction of this text: we will associate to each (left, cart)-fibration $p\colon \dcat{D}\rt \dcat{C}$ a (cocart, right)-fibration $\Psi^\perp(\dcat{D})\rt \dcat{C}$, which we will refer to as the \emph{reflection} of $p$. In Section \ref{sec:reflecting}, we will study the functoriality of this construction in more detail.

Before we begin, note that at a heuristic level, Remark \ref{rem:(left,cart)-straightened} and Remark \ref{rem:(cocart,right)-straightened} show that (left, cart)-fibrations and (cocart, right)-fibrations over a double category $\dcat{C}$ encode the same kind of data: both should be considered as unstraightened versions of the datum of a map of double categories
$$\begin{tikzcd}
\dcat{C}^{2-\op}\arrow[r] & \mathbb{C}\cat{at}^{\mm{oplax}}.
\end{tikzcd}$$
One should therefore be able to straighten a (left, cart)-fibration to such a map of double categories and then unstraighten it to obtain a (cocart, right)-fibration. Instead of making this more precise, we will give a purely combinatorial description of (what should be) the composite functor $\Psi^\perp$, inspired by the explicit formula for the Grothendieck construction and the combinatorics appearing in \cite{heb20}. Informally, the functor $\Psi^\perp$ will first remove the vertical arrows in each fiber of $p\colon \dcat{D}\rt \dcat{C}$, so that only the cartesian vertical arrows remain, and then replaces them by fiberwise horizontal arrows instead (which can be composed with the cocartesian horizontal arrows already present in $\dcat{D}$).

Let us start by considering the following construction of a double category out of a (left, cart)-fibration, which will appear repeatedly in the text:
\begin{lemma}\label{lem:reflection of left is dcat}
Let $p\colon \dcat{D}\rt \dcat{C}$ be a map of double categories such that each $\dcat{D}(-, n)\rt \dcat{C}(-, n)$ is a left fibration. Consider the bisimplicial space $X$ whose space of $(m, n)$-simplices is given by the space of commuting squares
\begin{equation}\label{diag:simplex in reflection}\begin{tikzcd}
\Ar[m]\times [0, n]\arrow[r]\arrow[d, "\pi_\mm{hor}\times \mm{id}"{swap}] & \dcat{D}\arrow[d, "p"]\\
{[m, n]}\arrow[r] & \dcat{C}.
\end{tikzcd}\end{equation}
Then $X$ is a double category.
\end{lemma}
\begin{remark}\label{rem:refl heuristic}
Informally, $X$ is a double category with the same objects as $\dcat{D}$, whose squares are given by diagrams in $\dcat{D}$ of the form
$$\begin{tikzcd}
{}\arrow[d, \vertcolor]\arrow[r] & {} \arrow[d, \vertcolor]\arrow[r, \vertcolor] & {}\arrow[d, \vertcolor]\\
{}\arrow[r] & {}\arrow[r, \vertcolor] & {}
\end{tikzcd}$$
where the right square is a \emph{commuting} square of vertical arrows in which the top and bottom arrow are sent to degenerate arrows in $\dcat{C}$. In other words, we keep the same vertical arrows but replace the horizontal arrows by formal composites $\beta\circ\alpha\colon \cdot\rt \cdot\vto \cdot$ of a horizontal and a fiberwise vertical arrow. Diagrams of the above form have an evident vertical composition. For the composition of horizontal arrows, consider the concatenation of two formal composites
$$\begin{tikzcd}
\cdot \arrow[r, "\alpha_1"] & \cdot\arrow[r, "\beta_1", \vertcolor] & \cdot\arrow[r, "\alpha_2"] & \cdot \arrow[r, "\beta_2", \vertcolor] & \cdot
\end{tikzcd}$$
Then $\beta_1$ and $\alpha_2$ fit into an essentially unique $(1, 1)$-cell of $\dcat{D}$ that covers the vertically degenerate (1, 1)-cell $p(\alpha_2)$ in $\dcat{C}$. The other two sides of this square give a formal composite $\beta'_1\circ \alpha'_2\colon \cdot \rt\cdot  \vto\cdot$. The composed horizontal arrow in $X$ is then given by the formal composite $(\beta_2\beta'_1)\circ (\alpha'_2\alpha_1)$.
\end{remark}
\begin{proof}
To see that $X$ is a category in the vertical direction, we use that $\DCat$ is cartesian closed (Proposition \ref{prop:dcat cart closed}): the simplicial space $X(m, -)$ can then be identified with the category $\big(\dcat{D}^{\Ar[m]}\times_{\dcat{C}^{\Ar[m]}} \dcat{C}^{[m, 0]}\big)(0, -)$.

To see that each $X(-, n)$ is a category, it suffices to treat the case $n=0$: indeed, the general case then follows by considering the functor $\dcat{D}^{[0, n]}\rt \dcat{C}^{[0, n]}$, which is still a left fibration in the horizontal direction.
Unraveling the definitions shows that the Segal condition for $X(-, 0)$ translates into the unique lifting problem
$$\begin{tikzcd}
\Ar[0\leq 1]\amalg_{\{11\}} \Ar[1\leq  \dots\leq m]  \arrow[d]\arrow[rr] & & \dcat{D}\arrow[d, "\pi"]\\
\Ar[0\leq \dots\leq m]\arrow[r] \arrow[rru, dashed, end anchor={[xshift=-0.5ex, yshift=-0.8ex]}] & {[m, 0]}\arrow[r] & \dcat{C}.
\end{tikzcd}$$
In other words, given a solid diagram in $\dcat{D}$ of the form
$$\begin{tikzcd}
x_{00}\arrow[r] & x_{01}\arrow[d, \vertcolor]\arrow[r, dashed] & x_{02}\arrow[r, dashed]\arrow[d, \vertcolor, dashed] & \dots\arrow[d, \vertcolor, dashed ]\arrow[r, dashed] & x_{0m}\arrow[d, \vertcolor, dashed] \\
& x_{11}\arrow[r] & x_{12}\arrow[d, \vertcolor] \arrow[r] & \dots \arrow[d, \vertcolor]\arrow[r] & x_{1m}\arrow[d, \vertcolor]\\
& & x_{22} \arrow[r] & \dots\arrow[d, \vertcolor]\arrow[r] & x_{2m}\arrow[d, \vertcolor]\\
& & & \ddots \arrow[r] & \vdots \arrow[d, \vertcolor]\\
& & & & x_{mm} 
\end{tikzcd}$$
whose image in $\dcat{C}$ has a fixed dashed extension (i.e.\ including the top squares), there exists a contractible space of dashed extensions as indicated. This follows immediately from the fact that each $\dcat{D}(-, n)\rt \dcat{C}(-, n)$ is a left fibration of categories.

For completeness of $X(-, 0)$, recall the simplicial set $H=\Delta[3]/\sim$ in the definition \ref{it:completeness} of completeness, and note that a map
$$\begin{tikzcd}
H=\Delta[3]/\sim\arrow[r] & X(-, 0)
\end{tikzcd}$$
corresponds to a diagram in $\dcat{D}$ of the form
$$\begin{tikzcd}
x_{00}\arrow[r] & x_{01}\arrow[d, \vertcolor]\arrow[r] & x_{02}\arrow[r]\arrow[d, \vertcolor] & x_{03}\arrow[d, \vertcolor]\\
& x_{11}\arrow[r] & x_{12}\arrow[d, \vertcolor] \arrow[r] & x_{13}\arrow[d, \vertcolor]\\
& & x_{22} \arrow[r] & x_{23}\arrow[d, \vertcolor]\\
& & & x_{33} 
\end{tikzcd}$$
in which $x_{00}\rt x_{02}$ and $x_{02} \vto{} x_{22}$ are degenerate and $x_{11}\rt x_{13}$ and $x_{13}\vto{} x_{33}$ are degenerate, and whose image in $\dcat{C}$ consists only of degenerate arrows (since $\dcat{C}$ is complete). We need to show that the above diagram is degenerate both in the horizontal and the vertical direction. Using that each $\dcat{D}(-, n)\rt \dcat{C}(-, n)$ is a left fibration, one sees that the above diagram is degenerate in the horizontal direction. Furthermore, the rightmost column is degenerate in the vertical direction, since both $x_{03}\vto{}x_{23}$ and $x_{13}\vto x_{33}$ are degenerate and $\dcat{D}(0, -)$ is a category. 
\end{proof}
\begin{definition}\label{def:reflection}
Let $p\colon \dcat{D}\rt \dcat{C}$ be a (left, cart)-fibration and let $X$ be the double category from Lemma \ref{lem:reflection of left is dcat}. We will denote by $\Psi^\perp(\dcat{D})\subseteq X$ the bisimplicial subspace whose $(m, n)$-simplices correspond to diagrams \eqref{diag:simplex in reflection} with the following property: the top map \mbox{$\Ar[m]\times [0, n]\rt \dcat{D}$} sends each $\{ii\}\times [0, n]$ into the subspace of cartesian vertical arrows of $\dcat{D}$. 
We will refer to the evident projection $\Psi^\perp(p)\colon \Psi^\perp(\dcat{D})\rt \dcat{C}$ as the \emph{reflection} of $p$.
\end{definition}
In terms of the heuristic description of $X$ from Remark \ref{rem:refl heuristic}, one can identify $\Psi^\perp(\dcat{D})\subseteq X$ with the largest double subcategory of $X$ with the property that any vertical morphism $[0, 1]\to\Psi^\perp(\dcat{D})$ corresponds to a $p$-cartesian vertical arrow of $\dcat{D}$. For the reason behind the terminology, see Proposition \ref{prop:why its called reflection}.
\begin{proposition}\label{prop:reflection is cocart-left}
Let $p\colon \dcat{D}\rt \dcat{C}$ be a (left, cart)-fibration. Then the reflection $\Psi^\perp(\dcat{D})\rt \dcat{C}$ is a (cocart, right)-fibration between double categories.
\end{proposition}
\begin{proof}
We verify each of the conditions, starting with the fact that $\Psi^\perp(\dcat{D})$ is a double category.

\subsubsection*{Segal conditions}
Recall the double category $X$ from Lemma \ref{lem:reflection of left is dcat}. Since each $\Psi^\perp(\dcat{D})(-, n)\subseteq X(-, n)$ is a full subcategory of the category $X(-, n)$, it follows that $\Psi^\perp(\dcat{D})(-, n)$ is a category (i.e.\ complete Segal space) itself.
In the other direction, note that there is a pullback square of simplicial spaces
$$\begin{tikzcd}
\Psi^\perp(\dcat{D})([m], -)\arrow[r, hook]\arrow[d] & X([m], -)\arrow[d]\\
\prod_{i=0}^m \Psi^\perp(\dcat{D})(\{i\}, -)\arrow[r, hook] & \prod_{i=0}^m X(\{i\},-).
\end{tikzcd}$$
Since $X([m], -)$ is a category by Lemma \ref{lem:reflection of left is dcat}, it suffices to verify that $\Psi^\perp(\dcat{D})(0, -)\hooklongrightarrow X(0, -)$ is a subcategory. But this can simply be identified with the inclusion of the wide subcategory of $\dcat{D}(0, -)$ whose morphisms are cartesian morphisms.

\subsubsection*{Vertical right fibration}
Next, let us verify that each $\Psi^\perp(\dcat{D})(m, -)\rt \dcat{C}(m, -)$ is a right fibration. By the horizontal Segal condition, it suffices to treat the case $m=0$ and $m=1$. Unraveling the definitions, we have to verify that there are contractible spaces of diagonal lifts
$$\begin{tikzcd}
\{1\}\arrow[d]\arrow[r] & \dcat{D}\arrow[d, "\pi"] & & \Ar[1]\times \{1\}\arrow[rr]\arrow[d] & & \dcat{D}\arrow[d, "\pi"]\\
{[0, 1]}\arrow[r]\arrow[ru, dashed] & \dcat{C} & & \Ar[1]\times [0, 1]\arrow[r]\arrow[rru, dashed, end anchor={[xshift=-0.5ex, yshift=-0.8ex]}] & {[1, 1]} \arrow[r] & \dcat{C}
\end{tikzcd}$$
with the following property: the vertical arrow in $[0, 1]$, respectively the vertical arrows in $\{00, 11\}\times [0, 1]$, are sent to cartesian vertical arrows in $\dcat{D}$. For the left square, there is a unique such lift since $\dcat{D}(0, -)\rt \dcat{C}(0, -)$ is a cartesian fibration. For the right square, we have to verify that for any solid diagram
$$\begin{tikzcd}[column sep=0.9pc, row sep=0.8pc]
& x_{00, 0}\arrow[dl, \vertcolor, dashed, "\circ" marking]\arrow[rr, dashed] & & x_{01, 0}\arrow[dl,\vertcolor,dashed]\arrow[dd, dashed, \vertcolor]\\
x_{00, 1}\arrow[rr] & & x_{01, 1}\arrow[dd,\vertcolor] & & {} \arrow[rr] & & {} & \dcat{D}\arrow[ddd, end anchor={[yshift=1.5ex]}, start anchor={[yshift=-1.5ex]}]\\
& \arrow[d, end anchor={[yshift=-2ex]}, xshift=-2ex, bend right, mapsto] & & x_{11, 0}\arrow[ld, \vertcolor, dashed, "\circ" marking]\\
&{}  & x_{11, 1}\\
& c_{0, 0}\arrow[rr]\arrow[ld, \vertcolor] & & c_{1, 0}\arrow[ld, \vertcolor] & \arrow[rr] & & {} &  \dcat{C}\\
c_{0, 1}\arrow[rr] & & c_{1, 1}
\end{tikzcd}$$
there exist unique objects $x_{00, 0}, x_{01, 0}, x_{11, 0}$ together with a dashed extension as indicated, such that the two arrows $x_{00, 0}\vto x_{00, 1}$ and $x_{11, 0}\vto x_{11, 1}$ are cartesian. To see this, first take these two arrows to be the (unique) cartesian lifts of their image in $\dcat{C}$. Then the top square can be filled uniquely using that $\dcat{D}\rt \dcat{C}$ is a left fibration in the horizontal direction. The right square (living entirely in the vertical direction) can then be filled uniquely since $x_{11, 0}\vto x_{11, 1}$ was cartesian.

\subsubsection*{Horizontal cocartesian fibration}
Note that the space $\Psi^\perp(\dcat{D})(1, 0)$ of horizontal arrows in $\Psi^\perp(\dcat{D})$ is the space of diagrams in $\dcat{D}$ of the form
$$\begin{tikzcd}[sep=small]
x\arrow[r] & y\arrow[d, \vertcolor]\\
& z
\end{tikzcd}$$
covering a horizontal arrow in $\dcat{C}$. Let us say that a horizontal arrow in $\Psi^\perp(\dcat{D})$ is \emph{marked} if the corresponding vertical map $y\vto z$ is an equivalence in $\dcat{D}$. We claim that all marked horizontal arrows in $\Psi^\perp(\dcat{D})$ are $\Psi^\perp(p)$-cocartesian. In other words, we have to find a unique diagonal lift for every diagram of categories
$$\begin{tikzcd}
\Lambda^0[2]\arrow[r]\arrow[d] & \Psi^\perp(\dcat{D})(-, 0)\arrow[d, "\Psi^\perp(p)"]\\
{[2]}\arrow[r]\arrow[ru, dashed] & \dcat{C}(-, 0)
\end{tikzcd}$$ 
where the top map sends $0\rt 1$ to a marked arrow. Unraveling the definitions, this corresponds to a solid diagram in $\dcat{D}$ covering a horizontal $2$-simplex in $\dcat{C}$ 
$$\begin{tikzcd}
x_{00}\arrow[r]\arrow[rr, bend left] & x_{01}\arrow[d, \vertcolor, "\sim"{swap}] \arrow[r, dashed] & x_{02}\arrow[dd, bend left, \vertcolor]\arrow[d, dashed, \vertcolor]\\
& x_{11}\arrow[r, dashed] & x_{12}\arrow[d, dashed, \vertcolor]\\
& & x_{22}
\end{tikzcd}$$
where the map $x_{01}\vto x_{11}$ is an equivalence, of which one has to find a unique dashed extension as indicated. To find the desired unique extension, using that $\dcat{D}\rt \dcat{C}$ is a left fibration in the horizontal direction, it follows that there is a unique way to extend the square making the top triangle (living purely in the horizontal direction) commute. Since $x_{01} \vto x_{11}$ was an equivalence, the resulting vertical arrow $x_{02} \vto x_{12}$ is an equivalence as well. It then follows that there is a unique vertical map $x_{12}\vto x_{22}$ with a filling of the right triangle (purely in the vertical direction). We conclude that all marked arrows in $\Psi^\perp(\dcat{D})$ are cocartesian. Finally, for every horizontal arrow $f\colon c\rt c'$ in $\dcat{C}$ and a lift $x\in \Psi^\perp(\dcat{D})_c$, there exist a marked lift $\tilde{f}\colon x\rt x'$, using that $\dcat{D}(-, 0)\rt \dcat{C}(-, 0)$ was a left fibration: indeed, one can take $\tilde{f}$ to correspond to the diagram $x\rt f_!x{\color{\vertcolor} \, = \,} f_!x$ in $\dcat{D}$.
\end{proof}

\section{Reflecting fibrations of double categories}
\label{sec:reflecting}
This section provides a more detailed analysis of the reflection of (left, cart)-fibrations between double categories (Definition \ref{def:reflection}). More precisely, our goal will be to prove our main technical result:
\begin{theorem}\label{thm:main theorem}
There is an equivalence of cartesian fibrations
$$\begin{tikzcd}[column sep=1.5pc]
\Psi^{\perp}\colon \cat{Fib}^{\mm{left, cart}}\arrow[rr, "\sim"{below}, yshift=1ex]\arrow[rd, "\mm{codom}"{swap}] & & \cat{Fib}^{\mm{cocart, right}}\colon \Psi^{\top}\arrow[ld, "\mm{codom}"]\arrow[ll, yshift=-1ex]\\
& \DCat
\end{tikzcd}$$
between the categories of (left, cart)-fibrations and (cocart, right)-fibrations. In particular, for any double category $\dcat{C}$ this restricts to an equivalence $\cat{Fib}^{\mm{left, cart}}(\dcat{C})\simeq \cat{Fib}^{\mm{cocart, right}}(\dcat{C})$.
\end{theorem}
This should not be unexpected: in light of Remark \ref{rem:(left,cart)-straightened} and Remark \ref{rem:(cocart,right)-straightened}, both (left, cart)-fibrations and (cocart, right)-fibrations should be equivalent to functors of double categories $\dcat{C}^{2-\op}\rt \mathbb{C}\cat{at}^{\mathrm{oplax}}$ under some form of straightening.

At the level of objects, the functor $\Psi^\perp$ is given by Definition \ref{def:reflection}. The functor $\Psi^{\top}$ is defined analogously, taking opposites and reversing the roles of the horizontal and vertical directions. To extend $\Psi^\perp$ and $\Psi^\top$ to functors, it will be convenient to describe the $\infty$-categories of (left, cart)-fibrations and (cocart, right)-fibrations in terms of marked bisimplicial spaces (Section \ref{sec:marked stuff}). This allows one to describe the functor $\Psi^\perp$ very explicitly in terms of combinatorial data (Section \ref{sec:functoriality}).
Finally, we will provide explicit natural equivalences exhibiting $\Psi^\perp$ and $\Psi^\top$ as mutual inverses (Section \ref{sec:equivalence}). We will do this by explicitly writing down a zigzag of equivalences connecting their composition to the identity.

\begin{remark}
In model categorical terms, both functors $\Psi^\perp$ and $\Psi^\top$ arise as right Quillen functors between the arrow categories of marked bisimplicial spaces. In particular, all of our arguments can be reformulated without difficulties into such model-categorical terms. 
\end{remark}

\subsection{Marked bisimplicial spaces}\label{sec:marked stuff}
For technical reasons, it will be convenient to give a 
description of the category $\cat{Fib}^{\mm{left, cart}}$ using markings, analogous to the model for cartesian fibrations using marked simplicial sets in \cite{lur09}.
\begin{notation}\label{not:marked}
Let $\Del^{\times 2}_{\downarrow}$ denote the category obtained from $\Del^{\times 2}$ by freely adding a factorization of the vertical degeneracy map $[0, 1]\rt [0, 0]$ as follows:
$$\begin{tikzcd}[column sep=1.4pc, row sep=0.9pc]
{[0, 1]}\arrow[rr]\arrow[rd] & & {[0, 0]}\\
& {[0, 1]^\sharp}.\arrow[ru]
\end{tikzcd}$$
A presheaf $X\colon (\Del^{\times 2}_{\downarrow})^{\op}\rt \sS$ is said to be a \emph{vertically marked bisimplicial space} if the map $X([0, 1]^\sharp)\rt X(0, 1)$ is a subspace inclusion. We will write $\bisAn_{\downarrow}\subseteq \Fun(\Del^{\times 2, \op}_{\downarrow}, \sS)$ for the full subcategory spanned by the vertically marked bisimplicial spaces, and typically use the same symbol for a vertically marked bisimplicial space and its underlying bisimplicial space. 
\end{notation}
\begin{remark}\label{rem:forgetting marking}
The forgetful functor $\bisAn_{\downarrow}\rt \bisAn$ induces a subspace inclusion
$$\begin{tikzcd}
\Map_{\bisAn_{\downarrow}}(X, Y)\arrow[r, hook] & \Map_{\bisAn}(X, Y)
\end{tikzcd}$$
whose image consists of the path components of maps of bisimplicial spaces $f\colon X\rt Y$ with the property that $X(0, 1)\rt Y(0, 1)$ preserves the subspace of marked vertical arrows.
\end{remark}
\begin{definition}
A \emph{vertically marked double category} is a vertically marked bisimplicial space whose underlying bisimplicial space is a double category. We will write $\DCat_{\downarrow}\subseteq \bisAn_{\downarrow}$ for the full subcategory of vertically marked double categories.

Likewise, we write $\DCat_{\to}\subseteq \bisAn_{\to}$ for the full subcategory of \emph{horizontally marked bisimplicial spaces} spanned by the \emph{horizontally marked double categories}.
\end{definition}
\begin{notation}
As usual, the forgetful functor $\DCat_{\downarrow}\rt \DCat$ admits a fully faithful right adjoint, sending a double category $\dcat{C}$ to the vertically marked double category $\dcat{C}^{\sharp}$ where all vertical arrows are marked. We will write also write $\dcat{C}^{\sharp}$ for $\dcat{C}$ with all \emph{horizontal} arrows marked; it will always be clear from the context whether we mark horizontal or vertical arrows. 
\end{notation}
\begin{definition}\label{def:global category of fibs}
Consider the following full subcategories of $\Fun\big([1], \bisAn_{\downarrow}\big)$:
\begin{enumerate}
\item Let $\Fun^\sharp\big([1], \bisAn_{\downarrow}\big)$ be spanned by the maps $Y\rt X$ where $X$ is maximally marked, i.e.\ every $[0, 1]$-simplex of $X$ is marked.
\item Let $\Fun^\sharp\big([1], \DCat_{\downarrow}\big)$ be spanned by the maps $\dcat{D}\rt \dcat{C}^\sharp$ between marked double categories, with $\dcat{C}^\sharp$ maximally marked.
\item Let $\cat{Fib}^{\mm{left, cart}}$ be spanned by the maps $p\colon \dcat{D}^\natural\rt \dcat{C}^\sharp$ between marked double categories, whose underlying map of double categories is a (left, cart)-fibration and where a vertical arrow in $\dcat{D}^\natural$ is marked if and only if it is $p$-cartesian.
\end{enumerate}
The codomain projections (forgetting the maximal marking) then define a diagram of cartesian fibrations and maps preserving cartesian arrows
$$\begin{tikzcd}
\cat{Fib}^{\mm{left, cart}}\arrow[d, "\mm{codom}"{swap}]\arrow[r, hook] & \Fun^\sharp\big([1], \DCat_{\downarrow}\big)\arrow[d, "\mm{codom}"]\arrow[r, hook] & \Fun^\sharp\big([1], \bisAn_{\downarrow}\big)\arrow[d, "\mm{codom}"]\\
\DCat\arrow[r, "="] & \DCat\arrow[r, hook] & \bisAn.
\end{tikzcd}$$
In a similar way, one can define fully faithful inclusions
$$\begin{tikzcd}
\cat{Fib}^{\mm{cocart, right}}\arrow[d, "\mm{codom}"{swap}]\arrow[r, hook] & \Fun^\sharp\big([1], \DCat_{\to}\big)\arrow[d, "\mm{codom}"]\arrow[r, hook] & \Fun^\sharp\big([1], \bisAn_{\to}\big)\arrow[d, "\mm{codom}"]\\
\DCat\arrow[r, "="] & \DCat\arrow[r, hook] & \bisAn.
 \end{tikzcd}$$
\end{definition}
\begin{lemma}
Forgetting the marking defines a functor $\cat{Fib}^{\mm{left, cart}}\rt \Fun\big([1], \DCat)$ which is an equivalence onto the subcategory whose objects are (left, cart)-fibrations and whose morphisms are strong morphisms (Definition \ref{def:fibrations}).
\end{lemma}
\begin{proof}
By Remark \ref{rem:forgetting marking}, the forgetful functor $\Fun\big([1], \DCat_{\downarrow})\rt \Fun\big([1], \DCat)$ induces inclusions of path components on mapping spaces, whose images consist of those commuting squares of double categories such that all maps preserve marked arrows. The result follows immediately from this.
\end{proof}
\begin{remark}
The full subcategory $\cat{Fib}^{\mm{left, cart}}\hooklongrightarrow \Fun\big([1], \bisAn_{\downarrow}\big)$ can be obtained as a left Bousfield localization (and in particular is a presentable category). One can also present $\cat{Fib}^{\mm{left, cart}}$ in terms of model $(1, 1)$-categories, as a left Bousfield localization of the injective model structure on functors $\Del^{\times 2, \op}_{\downarrow}\rt \sSet$ (where $\sSet$ carries the Kan--Quillen model structure). All arguments appearing below can be carried out in this setting as well.
\end{remark}

\begin{construction}\label{con:constructing marked functors}
Let us fix the following data:
\begin{enumerate}[label=(\alph*)]
\item\label{it:kernel} for each $m, n\geq 0$, a map of vertically marked double categories $S[m, n]\rt [m, n]^\sharp$, natural in the sense that it defines a functor fitting into a commuting triangle
$$\begin{tikzcd}
& \Fun^\sharp\big([1], \bisAn_{\downarrow}\big)\arrow[d, "\mm{codom}"]\\
\Del^{\times 2}\arrow[r, hook]\arrow[ru, "S"] & \bisAn.
\end{tikzcd}$$
\item\label{it:extra marking} in the vertically marked double category $S[1, 0]$, an \emph{additional} set of path components of marked vertical edges. We will write $S[1, 0]^+$ for the resulting vertically marked double category.
\end{enumerate}
Associated to this data is a natural map of cartesian fibrations
\begin{equation}\label{diag:restriction along kernel}
\begin{tikzcd}
\Fun^\sharp\big([1], \bisAn_{\downarrow}\big)\arrow[rr, "\Psi_S"]\arrow[rd, "\mm{codom}"{swap}] & & \Fun^\sharp\big([1], \bisAn_{\to}\big)\arrow[ld, "\mm{codom}"]\\
& \bisAn.
\end{tikzcd}\end{equation}
To construct this diagram, observe that the codomain projection admits a left adjoint $L\colon \bisAn\rt \Fun^\sharp([1], \bisAn_{\downarrow})$, sending $\dcat{C}$ to $\emptyset\rt \dcat{C}^\sharp$; the same holds in the horizontally marked case. Now consider the commuting squares
$$
\begin{tikzcd}[column sep=2pc]
{[1]^{\op}\times \Del^{\times 2}_{\to}}\arrow[r, "\sigma"] & \Fun^\sharp\big([1], \bisAn_{\downarrow}\big) & \Fun\big([1], \mm{PSh}(\Del_{\to}^{\times 2})\big)\arrow[r, "\Phi_S"] & \Fun^\sharp\big([1], \bisAn_{\downarrow}\big)\\
\{1\}\times \Del^{\times 2}\arrow[r, hook, "h"]\arrow[u, hook] & \bisAn\arrow[u, "L"{swap}] & \bisAn\arrow[r, "="]\arrow[u, "L"] & \bisAn\arrow[u, "L"{swap}] .
\end{tikzcd}$$
In the left square, $h$ is the Yoneda embedding and $\sigma$ is the functor sending
$$
\big(1, [m, n]\big)\mapsto \big(\emptyset\rt [m, n]^\sharp\big) \qquad\qquad \big(0, [m, n]\big)\mapsto \big(S[m, n]\rt [m, n]^\sharp\big).
$$
Furthermore, $\sigma$ sends $(0, [1, 0]^\sharp)$ to $S[1, 0]^+\rt [1, 0]^\sharp$ and $(1, [1, 0]^\sharp)$ to $\emptyset \rt [1, 0]^\sharp$. This induces the commuting square of left adjoint functors on the right, using that $h$ and $\sigma$ induce unique colimit-preserving functors out of the corresponding presheaf categories \cite{dug01b}, \cite[Theorem 5.1.5.6]{lur09}. 

Let us write $\Psi_S$ for the right adjoint to $\Phi_S$. Explicitly, let $p\colon Y\rt X^\sharp$ be a map of vertically marked bisimplicial spaces whose target is maximally marked. Then $\Psi_S(p)\colon \Psi_S(Y)\rt X^\sharp$ can be described as follows: the underlying bisimplicial space of $\Psi_S(Y)$ has $(m, n)$-simplices determined by
\begin{equation}\label{diag:simplices in restriction}\begin{tikzcd}
& \Psi_S(Y)\arrow[d, "\Psi_S(p)"]\\
{[m, n]}\arrow[r, "\alpha"]\arrow[ru, dotted] & X
\end{tikzcd}
\qquad \Longleftrightarrow \qquad
\begin{tikzcd} S[m, n]\arrow[r, dotted]\arrow[d] & Y\arrow[d, "p"]\\ {[m, n]^\sharp}\arrow[r, "\alpha"] & X^\sharp
\end{tikzcd}\end{equation}
Furthermore the space of marked $[1, 0]$-simplices in $\Psi_S(Y)$ is the subspace consisting of those maps $S[1, 0]\rt Y$ that also preserve the additional marked arrows from \ref{it:extra marking}. In particular, this implies that $\Psi_S$ sends vertically marked bisimplicial spaces into horizontally marked bisimplicial spaces (i.e.\ marked edges form a subspace). We therefore obtain the desired commuting diagram \eqref{diag:restriction along kernel}. Furthermore, it follows immediately from Equation \eqref{diag:simplices in restriction} that $\Psi_K$ preserves cartesian morphisms: for every $Y\rt X^\sharp$ and $X'\rt X$, we have an equivalence
$$\begin{tikzcd}
\Psi_S(Y\times_{X^\sharp} X'^\sharp)\arrow[r, "\sim"] & \Psi_S(Y)\times_{X^\sharp} X'^\sharp.
\end{tikzcd}$$
\end{construction}

\subsection{Construction of the functors}\label{sec:functoriality}
We will use Construction \ref{con:constructing marked functors} to define the functor $\Psi^\perp$, as follows:
\begin{construction}\label{con:functor K}
For any $[m, n]\in \Del^{\times 2}$, consider the natural map of vertically marked double categories
$$\begin{tikzcd}[column sep=2.8pc]
{K[m, n]=\Ar[m]\times [0, n]} \arrow[r, "\pi_{\mm{hor}}\times \mm{id}"] & {[m, 0]\times [0, n]^\sharp=[m, n]^\sharp}.
\end{tikzcd}$$
Here $K[m, n]$ comes equipped with the marking where for each $ii\in \Ar[m]$ (cf.\ Example \ref{ex:arrow dcat}), all vertical arrows in $\{ii\}\times [0, n]$ are marked. Furthermore, $K[1, 0]^+$ is the double category $\Ar[1]$, where in addition the vertical arrow $01\vto 11$ is marked.
\end{construction}
Using this, Construction \ref{con:constructing marked functors} produces a functor over $\bisAn$
$$\begin{tikzcd}
\Psi_K\colon \Fun^\sharp\big([1], \bisAn_{\downarrow}\big)\arrow[r] & \Fun^\sharp\big([1], \bisAn_{\to}\big).
\end{tikzcd}$$
\begin{proposition}\label{prop:reflection functor}
Let $p\colon \dcat{D}^\natural\rt \dcat{C}^\sharp$ be a (left, cart)-fibration with its natural marking by the $p$-cartesian vertical arrows. Then $\Psi_K(p)\colon \Psi_K(\dcat{D})\rt \dcat{C}^\sharp$ coincides with the (cocart, right)-fibration $\Psi^\perp(p)$ from Definition \ref{def:reflection}, with its natural horizontal marking. Consequently, $\Psi_K$ restricts to a map of cartesian fibrations that we will denote by
$$\begin{tikzcd}
\cat{Fib}^{\mm{left, cart}}\arrow[rr, "\Psi^\perp"]\arrow[rd, "\mm{codom}"{swap}, start anchor={[xshift=-3ex]}] & & \cat{Fib}^{\mm{cocart, right}}\arrow[ld, "\mm{codom}", start anchor={[xshift=-2ex]}]\\
& \DCat.
\end{tikzcd}$$
\end{proposition}
\begin{proof}
Unraveling the definition, one sees that the map of bisimplicial spaces $\Psi_K(\dcat{D})\rt \dcat{C}$ is precisely given by the reflection $\Psi^\perp(p)\colon \Psi^\perp(\dcat{D})\rt \dcat{C}$ from Definition \ref{def:reflection}; in particular, it is a (cocart, right)-fibration by Proposition \ref{prop:reflection is cocart-left}.

It remains to identify the marked horizontal arrows in $\Psi_K(\dcat{D}^{\natural})$; these correspond to diagrams $x\rt y\vto z$ in $\dcat{D}$ with the property that $y\vto z$ is a marked arrow covering an equivalence in $\dcat{C}$. This just means that $y\vto z$ is an equivalence, so that the marked horizontal arrows coincide with the marked horizontal arrows introduced in the proof of Proposition \ref{prop:reflection is cocart-left}. We have already seen there that these marked arrows are precisely the cocartesian horizontal arrows.
\end{proof}
In particular, restricting to the fibers over a double category $\dcat{C}$, we obtain a functor $\Psi^\perp\colon \cat{Fib}^{\mm{left, cart}}(\dcat{C})\rt \cat{Fib}^{\mm{cocart, right}}(\dcat{C})$ that we will refer to as the \emph{reflection functor}. To motivate the terminology, let us investigate more precisely the behaviour of $\Psi^\perp$ when $\dcat{C}$ is a space:
\begin{proposition}\label{prop:why its called reflection}
Let $S$ be a space and let $\Cat/S$ be the category of categories over $S$. Consider the endofunctor of $\Cat/S$ given by the composite
$$\begin{tikzcd}[column sep=2.7pc]
\cat{Cat}/S \arrow[r, "\sim"{swap}, "{[0]\boxtimes(-)}"] & \cat{Fib}^{\mm{left, cart}}(S)\arrow[r, "\Psi^\perp"] & \cat{Fib}^{\mm{right, cocart}}(S)  & \cat{Cat}/S\arrow[l, "\sim", "{(-)\boxtimes [0]}"{swap}]
\end{tikzcd}$$
where the left equivalence sends a category $\cat{C}$ to the (left, cart)-fibration $[0]\boxtimes \cat{C}\rt [0]\boxtimes S=S$ and the right equivalence to the (cocart, right)-fibration $\cat{C}\boxtimes [0]\rt S\boxtimes [0]=S$ (see Example \ref{ex:generating dcat} and Remark \ref{rem:basic properties of leftcart fibs}). Then the above endofunctor of $\cat{Cat}/S$ is homotopic to the identity.
\end{proposition}
In particular, at the level of fibers the reflection functor $\Psi^\perp$ simply exchanges the horizontal and vertical direction.
\begin{proof}
Consider the natural map $K[m, n]=\Ar[m]\times [0, n]\rt \Ar[m]\rto{\pi_\mm{ver}} [0, m]$, where the second map is the vertical projection (see Example \ref{ex:arrow dcat}). This map sends the marked vertical arrows in $K[m, n]$ from Construction \ref{con:functor K} to degenerate edges. For each $\dcat{D}\in \cat{Fib}^{\mm{left, cart}}$, we therefore obtain a natural transformation of double categories, given in degree $(m, n)$ by
$$\begin{tikzcd}
\Map_{\DCat}([0, m], \dcat{D})\arrow[r] & \Map_{\DCat_{\downarrow}}(K[m, n], \dcat{D})=\Psi_K(\dcat{D})(m, n).
\end{tikzcd}$$
Taking $\dcat{D}=[0]\boxtimes \cat{C}$, this produces the desired map $\cat{C}\boxtimes [0]\rt \Psi^\perp([0]\boxtimes \cat{C})$. To see that this is an equivalence, note that the domain and codomain are both constant in the vertical direction and complete Segal spaces in the horizontal direction. In degree $(0, 0)$, the above is an equivalence since $K[0, 0]\rt [0, 0]$ is an isomorphism. In degree $(1, 0)$, it suffices to note that the map $\pi_\mm{ver}\colon \Ar[1]\rt [0, 1]$ induces an equivalence on mapping spaces into a double category of the form $[0]\boxtimes \cat{C}$.
\end{proof}
One can obtain variants of the functor $\Psi^\perp$ for different kinds of fibrations, by taking opposites in the horizontal or vertical direction and exchanging horizontal and vertical directions.
\begin{variant}\label{var:other reflections}
Let us introduce the following two variants of the functor $\Psi^\perp$.
\begin{enumerate}
\item There is a functor $\Psi^\top\colon \cat{Fib}^{\mm{cocart, right}}\rt \cat{Fib}^{\mm{left, cart}}$, obtained by conjugating $\Psi^\perp$ with the functor sending $\dcat{C}\mapsto (\dcat{C}^\mm{rev})^{(1, 2)-\op}$. Unraveling the definitions, one sees that $\Psi^\top=\Psi_{K'}$ is defined as in Construction \ref{con:constructing marked functors} (with the roles of horizontal and vertical reversed), but this time using the maps
$$\begin{tikzcd}[column sep=3pc]
{K'[m, n]=[m, 0]\times \Ar[n]}\arrow[r, "\mm{id}\times \pi_{\mm{ver}}"] & {[m, 0]\times [0, n]=[m, n]. }
\end{tikzcd}$$
Here we mark the horizontal arrows in $K'[m, n]$ contained in $[m, 0]\times \{ii\}$ for each $i\in [n]$. Furthermore, $K'([0, 1]^\sharp)$ is given by $\Ar[1]$ where we mark all horizontal arrows.

\item There is another functor $\Psi^{\dagger}\colon \cat{Fib}^{\mm{cocart, left}}\rt \cat{Fib}^{\mm{left, cocart}}$, obtained by conjugating $\Psi^\perp$ with the functor sending $\dcat{C}\mapsto (\dcat{C}^{\mm{rev}})^{1-\op}$. Unraveling the definitions, one sees that $\Psi^\dagger=\Psi_L$ is defined as in Construction \ref{con:functor K}, but the role of the vertically marked double category $K[m, n]$ is now played by
$$\begin{tikzcd}[column sep=3pc]
{L[m, n]=[m, 0]\times \Tw\big([n]^{\op}\big)}\arrow[r, "\mm{id}\times \pi_{\mm{ver}}"] & {[m, 0]\times [0, n]=[m, n].}
\end{tikzcd}$$
In $L[m, n]$, we mark the horizontal arrows in each $[m, 0]\times \{ii\}$. Furthermore, $L([1, 0]^\sharp)$ is given by $\Tw([1]^{\op})$ with all horizontal arrows marked.
\end{enumerate}
\end{variant}
\begin{remark}
Proposition \ref{prop:why its called reflection} holds for the functor $\Psi^\top$ as well: indeed, the functor $(-)^\mm{rev}$ commutes with the functor $\dcat{C}\mapsto (\dcat{C}^{\mm{rev}})^{(1, 2)-\op}$. In contrast, it does not apply to the functor $\Psi^\dagger$: instead, the resulting endofunctor of $\cat{Cat}$ takes opposite categories. 
\end{remark}

\subsection{Proof of Theorem \ref{thm:main theorem}}\label{sec:equivalence}
Consider the functors from Proposition \ref{prop:reflection functor} and Variant \ref{var:other reflections}
$$\begin{tikzcd}
\Psi^\perp\colon \cat{Fib}^{\mm{left, cart}}\arrow[r] &  \cat{Fib}^{\mm{cocart, right}}, & \Psi^\top\colon \cat{Fib}^{\mm{cocart, right}}\arrow[r] & \cat{Fib}^{\mm{left, cart}}.
\end{tikzcd}$$
We will show that these are mutually inverse, up to homotopy. More precisely, we will prove that for every (left, cart)-fibration $\dcat{D}\rt \dcat{C}$, there is a natural zig-zag of equivalences between (left, cart)-fibrations over $\dcat{C}$
$$\begin{tikzcd}[column sep=1.3pc]
\dcat{D}\arrow[rrrd]\arrow[rr, "\zeta^*"] & & \Psi_A(\dcat{D}) \arrow[rd] & & \Psi_B(\dcat{D})\arrow[ld]\arrow[ll, "\eta^*"{above}]\arrow[rr, "\theta^*"] & & \Psi^\top(\Psi^\perp(\dcat{D}))\arrow[llld]\\
& & & \dcat{C}.
\end{tikzcd}$$
To motivate the precise combinatorial constructions that will follow, let us already disclose that a typical (1, 1)-cell in each of these four double categories will correspond to a diagram in $\dcat{D}$ of the following form:
$$\begin{tikzcd}[column sep=1.9pc]
                       &              &  &                        &                        &              &  &                        &                        &              &  & {} \arrow[d, "\sim"{swap}] \arrow[r] & {} \arrow[d, "\sim"{swap}] \arrow[r, \vertcolor, "\sim"] & {} \arrow[d, "\sim"] \\
                       &              &  &                        &                        &              &  & {} \arrow[d, \vertcolor, rightarrowtail] \arrow[r] & {} \arrow[d, \vertcolor, rightarrowtail] \arrow[r, \vertcolor, "\sim"] & {} \arrow[d, \vertcolor, rightarrowtail] &  & {} \arrow[d, \vertcolor, rightarrowtail] \arrow[r] & {} \arrow[d, \vertcolor, rightarrowtail] \arrow[r, \vertcolor, "\sim"] & {} \arrow[d, \vertcolor, rightarrowtail] \\
{} \arrow[d, \vertcolor] \arrow[r] & {} \arrow[d, \vertcolor] \arrow[rr, shorten=3.5ex, mapsto, bend left=30, yshift=-1.2pc, "\zeta^*"{yshift=2pt}] &  & {} \arrow[d, \vertcolor] \arrow[r] & {} \arrow[r, \vertcolor, "\sim"] \arrow[d, \vertcolor] & {} \arrow[d, \vertcolor] &  & \arrow[ll, shorten=3.5ex, mapsto, bend right=30, yshift=-1.2pc, "\eta^*"{swap, yshift=2pt}] {} \arrow[d, \vertcolor, "\circ"{marking}] \arrow[r] & {} \arrow[r, \vertcolor, rightarrowtail] \arrow[d, \vertcolor, "f"] & {} \arrow[d, \vertcolor, "\circ"{marking}] \arrow[rr, shorten=3.5ex, mapsto, bend left=30, yshift=-1.2pc, "\theta^*"{yshift=2pt}] &  & {} \arrow[d, \vertcolor, "\circ"{marking}] \arrow[r] & {} \arrow[d, \vertcolor] \arrow[r, \vertcolor, rightarrowtail] & {} \arrow[d, \vertcolor, "\circ"{marking}] \\
{} \arrow[r]           & {}           &  & {} \arrow[r]           & {} \arrow[r, \vertcolor, "\sim"]           & {}           &  & {} \arrow[r]           & {} \arrow[r, \vertcolor, "\sim"]           & {}           &  & {} \arrow[r]           & {} \arrow[r, \vertcolor, "\sim"]           & {}          
\end{tikzcd}$$
Here $\begin{tikzcd}[cramped, column sep=1.5pc] {}\arrow[r, \vertcolor, "\circ"{marking}] & {}\end{tikzcd}$ denotes a $p$-cartesian vertical arrow in $\dcat{D}$ and $\begin{tikzcd}[cramped, column sep=1.5pc] {}\arrow[r, rightarrowtail, \vertcolor] & {}\end{tikzcd}$ denotes a map that is sent to a degenerate arrow in $\dcat{C}$. The rightmost diagram can be obtained by explicitly unraveling the definition of the composite functor $\Psi^{\top}\circ\Psi^{\perp}$. The maps $\zeta^*$ and $\theta^*$ will then be the evident equivalences which simply add degenerate squares on the right, resp.\ on the top. 

The map $\eta^*$ takes vertical compositions and has an inverse given informally as follows (cf.\ the proof of Proposition \ref{prop:reflection is cocart-left}). First, one decomposes the left and right vertical maps into fiberwise maps followed by $p$-cartesian maps. The middle column is then uniquely determined by the left column and the right half of the diagram is determined uniquely because the right vertical map was $p$-cartesian. Notice that the map $f$ need not be $p$-cartesian; this is why we need to add a degenerate square using $\zeta^*$ before we can decompose the left and right vertical arrow.

Let us now start by defining the endofunctors $\Psi_A, \Psi_B\colon \cat{Fib}^{\mm{left, cart}}\rt \cat{Fib}^{\mm{left, cart}}$ more precisely, using (the vertically marked version of) Construction \ref{con:constructing marked functors}.
\begin{construction}\label{con:functor B}
For each $[m, n]\in \Del^{\times 2}$, let 
$$\begin{tikzcd}[column sep=3pc]
A[m, n]=\Ar[m]\times [0, n]\arrow[r, "\pi_{\mm{hor}}\times\mm{id}"] & {[m, n]}
\end{tikzcd}$$
In $A[m, n]$, we mark all vertical arrows contained in $\Ar[m]\times \{i\}$, for all $i\in [n]$. Note that this marking is different from Construction \ref{con:functor K}. In addition, we define $A([0, 1])^+=[0, 1]^\sharp$.

Likewise, for each $[m, n]\in \Del^{\times 2}$, let us denote
$$\begin{tikzcd}[column sep=3.2pc]
B[m, n]=\Ar[m]\times \big([0]\boxtimes \Fun([1], [n])\big)\arrow[r, "\pi_{\mm{hor}}\times\mm{dom}"] & {[m, 0]\times [0, n]=[m, n]}
\end{tikzcd}$$
In $B[m, n]=\Ar[m]\times \big([0]\boxtimes \Fun([1], [n])\big)$, we mark the following vertical arrows:
\begin{enumerate}[label=(\alph*)]
\item\label{it:Bmark 1} for each object $aa\in \Ar[m]$ with $0\leq a\leq m$, we mark all edges in $\{aa\}\boxtimes \Fun([1], [n])$ of the form $ik\rt jk$.

\item\label{it:Bmark 2} for each $ii\in \Fun([1], [n])$ with $0\leq i\leq n$, we mark all vertical arrows in $\Ar[m]\times \{ii\}$.
\end{enumerate}
In addition, define $B([0, 1])^+$ to be $[0]\boxtimes \Fun([1], [1])$, where all vertical arrows are marked.

Note that these vertically marked double categories are related by natural transformations
$$\begin{tikzcd}[column sep=3.7pc, row sep=2pc]
{[m, n]}\ar[rd, "="{swap}] & A[m, n]\arrow[l, "\zeta"{swap}]\arrow[d, "\pi_\mm{hor}\times \mm{id}"] \arrow[r, "\eta"] & B[m, n]\arrow[ld, "\pi_\mm{hor}\times \mm{dom}"]\\
& {[m, n]^{\sharp}}
\end{tikzcd}$$
Here $\zeta\colon \Ar[m]\times [0, n]\rt [m, n]$ just arises from the horizontal projection $\Ar[m]\rt [m, 0]$. This sends the vertical arrows within each $\Ar[m]\times \{i\}$ to degenerate arrows and hence preserves markings. The map $\eta\colon \Ar[m]\times [0, n]\rt \Ar[m]\times \big([0]\boxtimes \Fun([1], [n])\big)$ arises from the map $\mm{cst}\colon [n]\rt \Fun([1], [n])$ sending each $i\in [n]$ to the constant arrow $i\leq i$.

Furthermore, the maps $\zeta$ and $\eta$ induce natural maps $[0, 1]^\sharp \lt A[0, 1]^+\rt B[0, 1]^+$ preserving the additional marked vertical arrows. Using the naturality of Construction \ref{con:constructing marked functors}, we obtain a zig-zag of natural transformations between functors $\Fun^\sharp([1], \bisAn_{\downarrow})\rt \Fun^\sharp([1], \bisAn_{\downarrow})$ of the form
$$\begin{tikzcd}[column sep=2.5pc]
\mm{id}\arrow[r, "\zeta^*"] & \Psi_A & \Psi_B.\arrow[l, "\eta^*"{swap}]
\end{tikzcd}$$
\end{construction}
\begin{proposition}\label{prop:adding fluff}
Let $p\colon \dcat{D}^\natural\rt \dcat{C}^\sharp$ be a (left, cart)-fibration. Then the natural transformations $\zeta^*$ and $\eta^*$ above are equivalences of marked bisimplicial spaces. In particular, $\Psi_A(\dcat{D}^\natural)\rt \dcat{C}^\sharp$ and $\Psi_B(\dcat{D}^\natural)\rt \dcat{C}^\sharp$ are (left, cart)-fibrations of double categories.
\end{proposition}
\begin{proof}
Recall from the general description \eqref{diag:simplices in restriction} that the space of $(m, n)$-simplices in $\Psi_A(\dcat{D}^\natural)$ is given by the space of diagrams of vertically marked double categories
$$\begin{tikzcd}
A[m, n]\arrow[d]\arrow[r] & \dcat{D}^\natural\arrow[d, "p"]\\
{[m, n]^\sharp}\arrow[r] & \dcat{C}^\sharp
\end{tikzcd}$$ 
and likewise for $\Psi_B(\dcat{D}^\natural)$. The maps $\zeta^*$ and $\eta^*$ simply arise by restriction along the natural maps $\zeta\colon A[m, n]\rt [m, n]$ and $\eta\colon A[m, n]\rt B[m, n]$.

\subsubsection*{$\zeta^*$ is an equivalence}
Note that $\Psi_A(\dcat{D}^\natural)\hooklongrightarrow X$ is a bisimplicial subspace of the bisimplicial space $X$ from Lemma \ref{lem:reflection of left is dcat}. Indeed, the space of maps $[m, n]\rt \Psi_A(\dcat{D}^\natural)$ is the subspace of those diagrams \eqref{diag:simplex in reflection} such that $\Ar[m]\times [0, n]\rt \dcat{D}$ sends the vertical arrows in each $\Ar[m]\times \{i\}$ with $0\leq i\leq n$ to equivalences.

\sloppy By Lemma \ref{lem:reflection of left is dcat}, $X$ is double category and one sees that each simplicial subspace \mbox{$\Psi_A(\dcat{D}^\natural)(m, -)\hooklongrightarrow X(m, -)$} is the inclusion of a full subcategory. In the horizontal direction, the proof of Lemma \ref{lem:reflection of left is dcat} shows that the simplicial subspace $\Psi_A(\dcat{D}^\natural)(-, n)\hooklongrightarrow X(-, n)$ satisfies the Segal conditions and is an equivalence on spaces of objects. This implies that $\Psi_A(\dcat{D}^\natural)(-, n)$ is a category, so that $\Psi_A(\dcat{D}^\natural)$ is a double category.

Let us now show that the map of marked double categories $\zeta^*\colon \dcat{D}\rt \Psi_A(\dcat{D}^\sharp)$ is an equivalence. Since $\zeta$ induces an isomorphism $A[0, 1]^+\rto{\sim} [0, 1]^\sharp$, $\zeta^*$ is an equivalence on marked vertical arrows. It then remains to verify that the map on $(1,1)$-simplices is an equivalence. Unraveling the definitions, a $(1,1)$-simplex in $\Psi_A(\dcat{D}^\natural)$ is a diagram in $\dcat{D}$
$$\begin{tikzcd}[sep=small, column sep=1.5pc]
x_0\arrow[r]\arrow[d, \vertcolor] & y_0\arrow[r, \vertcolor, "\sim"]\arrow[d, \vertcolor] & z_0\arrow[d, \vertcolor]\\
x_1\arrow[r] & y_1\arrow[r, \vertcolor, "\sim"{swap}] & z_1
\end{tikzcd}$$
where the right square covers a vertical arrow in $\dcat{C}$. The map $\zeta^*\colon \dcat{D}(1, 1)\rt \Psi_A(\dcat{D})(1, 1)$ is simply the inclusion of the diagrams of the above form where the right square is degenerate; this is an equivalence by completeness of $\dcat{D}$. 

\subsubsection*{$\eta^*$ is an equivalence}
Let us again consider the double category $X$ from Lemma \ref{lem:reflection of left is dcat}. For each fixed $n$, consider the mapping double category $X^{[0]\boxtimes \Fun([1], [n])}$ and note that there is a diagram of simplicial spaces
$$\begin{tikzcd}
\Psi_B(\dcat{D}^\natural)(-, n)\arrow[r, hook] & P\arrow[r, hook]\arrow[d] & X^{[0]\boxtimes \Fun([1], [n])}\big(-, 0\big)\arrow[d]\\
& \prod_{i=0}^n \Psi_A(\dcat{D}^\natural)(-, 0)\arrow[r, hook] & \prod_{i=0}^n X^{\{ii\}}\big(-, 0\big)
\end{tikzcd}$$
where $P$ denotes the pullback. Since the bottom map is the inclusion of a wide subcategory (as we have seen above), $P$ is a wide subcategory of $X^{[0]\boxtimes \Fun([1], [n])}\big(-, 0\big)$. Unraveling the definitions, we then see that the inclusion $\Psi_B(\dcat{D}^\natural)(-, n)\hooklongrightarrow P$ is the inclusion of a full subcategory: it is the inclusion of the full subcategory on those diagrams $[0]\boxtimes \Fun([1], [n])\rt X$ such that each vertical arrow $ik\rt jk$ in $\Fun([1], [n])$ is sent to a cartesian arrow in $X(0, -)\simeq \dcat{D}(0, -)$. 

It follows that $\Psi_B(\dcat{D}^\natural)$ satisfies the (complete) Segal conditions in the horizontal direction. It therefore suffices to verify that $\eta^*\colon \Psi_B(\dcat{D}^\natural)(1, -)\rt \Psi_A(\dcat{D})(1, -)$ is an equivalence of simplicial spaces. 
Unraveling the definitions, a $(1, n)$-simplex in $\Psi_B(\dcat{D}^\natural)$ is given by a diagram of the form
\begin{equation}\label{diag:huge diagram}
\begin{tikzcd}[column sep=0.6pc, row sep=0.47pc]
                                  &                                                      & x_{00} \arrow[ld] \arrow[rrr, \vertcolor] \arrow[rrrddd, dotted, \vertcolor] &                                   &                                                                      & x_{01} \arrow[ld] \arrow[rrr, \vertcolor] \arrow[ddd, dotted, \vertcolor, "\circ"{marking}]                    &                                  &                                                                           & \dots\vphantom{x} \arrow[ld] \arrow[rrr, \vertcolor] \arrow[ddd, dotted, \vertcolor, "\circ"{marking}]                    &                    &                               & x_{0n} \arrow[ld] \arrow[ddd, \vertcolor, "\circ"{marking}] \\
                                  & y_{00} \arrow[ld, \vertcolor, "\sim"{swap}] \arrow[rrr, \vertcolor] \arrow[rrrddd, dotted, \vertcolor] &                                                      &                                   & y_{01} \arrow[ld, \vertcolor] \arrow[rrr, \vertcolor] \arrow[ddd, dotted, \vertcolor]                    &                                                                      &                                  & \dots\vphantom{x} \arrow[ld, \vertcolor] \arrow[rrr, \vertcolor] \arrow[ddd, dotted, \vertcolor]                          &                                                                     &                    & y_{0n} \arrow[ld, \vertcolor] \arrow[ddd, \vertcolor] &                               \\
z_{00} \arrow[rrr, \vertcolor] \arrow[rrrddd, \vertcolor] &                                                      &                                                      & z_{01} \arrow[rrr, \vertcolor] \arrow[ddd, \vertcolor, "\circ"{marking}]    &                                                                      &                                                                      & \dots\vphantom{x} \arrow[rrr, \vertcolor] \arrow[ddd, \vertcolor, "\circ"{marking}]    &                                                                           &                                                                     & z_{0n} \arrow[ddd, \vertcolor, "\circ"{marking}] &                               &                               \\
                                  &                                                      &                                                      &                                   &                                                                      & x_{11} \arrow[ld, dotted] \arrow[rrr, dotted, \vertcolor] \arrow[rrrddd, dotted, \vertcolor] &                                  &                                                                           & \dots\vphantom{x} \arrow[ld, dotted] \arrow[rrr, dotted, \vertcolor] \arrow[ddd, dotted, \vertcolor, "\circ"{marking}]    &                    &                               & x_{1n} \arrow[ld] \arrow[ddd, \vertcolor, "\circ"{marking}] \\
                                  &                                                      &                                                      &                                   & y_{11} \arrow[ld, dotted, \vertcolor, "\sim"{swap}] \arrow[rrr, dotted, \vertcolor] \arrow[rrrddd, dotted, \vertcolor] &                                                                      &                                  & \phantom{x} \arrow[ld, dotted, \vertcolor] \arrow[rrr, dotted, \vertcolor] \arrow[ddd, dotted, \vertcolor]    &                                                                     &                    & y_{1n} \arrow[ld, \vertcolor] \arrow[ddd, \vertcolor] &                               \\
                                  &                                                      &                                                      & z_{11} \arrow[rrr, \vertcolor] \arrow[rrrddd, \vertcolor] &                                                                      &                                                                      & \dots\vphantom{x} \arrow[rrr, \vertcolor] \arrow[ddd, \vertcolor, "\circ"{marking}]    &                                                                           &                                                                     & z_{1n} \arrow[ddd, \vertcolor, "\circ"{marking}] &                               &                               \\
                                  &                                                      &                                                      &                                   &                                                                      &                                                                      &                                  &                                                                           & \dots\vphantom{x} \arrow[ld, dotted] \arrow[rrr, dotted, \vertcolor] \arrow[rrrddd, dotted, \vertcolor] &                    &                               & \dots\vphantom{x} \arrow[ld] \arrow[ddd, \vertcolor, "\circ"{marking}]  \\
                                  &                                                      &                                                      &                                   &                                                                      &                                                                      &                                  & \phantom{x} \arrow[ld, dotted, \vertcolor, "\sim"] \arrow[rrr, dotted, \vertcolor] \arrow[rrrddd, dotted, \vertcolor] &                                                                     &                    & \dots\vphantom{x} \arrow[ld, \vertcolor] \arrow[ddd, \vertcolor]  &                               \\
                                  &                                                      &                                                      &                                   &                                                                      &                                                                      & \dots\vphantom{x} \arrow[rrr, \vertcolor] \arrow[rrrddd, \vertcolor] &                                                                           &                                                                     & \dots\vphantom{x} \arrow[ddd, \vertcolor, "\circ"{marking}]  &                               &                               \\
                                  &                                                      &                                                      &                                   &                                                                      &                                                                      &                                  &                                                                           &                                                                     &                    &                               & x_{nn} \arrow[ld]             \\
                                  &                                                      &                                                      &                                   &                                                                      &                                                                      &                                  &                                                                           &                                                                     &                    & y_{nn} \arrow[ld, \vertcolor, "\sim"]             &                               \\
                                  &                                                      &                                                      &                                   &                                                                      &                                                                      &                                  &                                                                           &                                                                     & z_{nn}             &                               &                              
\end{tikzcd}
\end{equation}
Here each of the slices labeled by $x, y$ and $z$ is a diagram in $\dcat{D}(0, -)$ parametrized by $\Fun([1], [n])$, while the maps $x_{ij}\rt y_{ij}\vto z_{ij}$ define a map $\Ar[1]\rt \dcat{D}$. Furthermore, the condition that $B[1, n]=\Ar[1]\times \Big([0]\boxtimes \Fun([1], [n])\Big)\rt \dcat{D}$ preserves marked arrows translates into the following conditions:
\begin{enumerate}[label=(\alph*)]
\item\label{it:diagonal equiv} The maps $y_{ii}\vto z_{ii}$ are equivalences for all $0\leq i\leq n$, as indicated.

\item\label{it:vertical arrows cartesian} All downward maps in the $x$-slice and the $z$-slice are $p$-cartesian, indicated by $\begin{tikzcd}[cramped, column sep=1.5pc] {}\arrow[r, "\circ"{marking}, \vertcolor] & {}\end{tikzcd}$.
\end{enumerate}
The map $\eta^*\colon \Psi_B(\dcat{D}^\natural)(1, n)\rt \Psi_A(\dcat{D}^{\natural})(1, n)$ restricts the above diagram to the diagonal face, spanned by all $x_{ii}, y_{ii}$ and $z_{ii}$. We have to prove that there is a unique way to extend the above diagram (with properties \ref{it:diagonal equiv} and \ref{it:vertical arrows cartesian}) from this diagonal face (covering the map $[m, n]\rt\dcat{C}$).
Let us first look at the $x$- and $z$-slices, both of which correspond to diagrams 
$$\begin{tikzcd}
\Fun([1], [n])\arrow[r, "{x, z}"]\arrow[d, "\mm{dom}"{swap}] & \dcat{D}(0, -)\arrow[d, "p"]\\
{[n]}\arrow[r] & \dcat{C}(0, -).
\end{tikzcd}$$
Let $\mm{cst}\colon [n]\hooklongrightarrow \Fun([1], [n])$ be the inclusion of the diagonal. For every $ij\in \Fun([1], [n])$, the comma category $ij/\mm{cst}$ admits an initial object, corresponding to the map $ij\rt jj$. Using this, one sees that condition \ref{it:vertical arrows cartesian} is equivalent to assertion that the two diagrams $x, z\colon \Fun([1], [n])\rt \dcat{D}$ are relative right Kan extensions of their restriction to the diagonal.

Consequently, starting from the diagonal face there is a unique way to fill the $x$- and $z$-slices in \eqref{diag:huge diagram} so that condition \ref{it:vertical arrows cartesian} holds \cite[Proposition 4.3.2.1]{lur09}. Using that $\dcat{D}\rt \dcat{C}$ is a left fibration in the horizontal direction, there is then a unique extension to a horizontal transformation between the $x$ and $y$-slices $(x\Rightarrow y)\colon [1]\boxtimes \Fun([1], [n])\rt \dcat{D}$. Finally, the natural transformation $y_{ii}\vto z_{ii}$ on the diagonal then extends uniquely to a natural transformation $y_{ij}\vto z_{ij}$, since the target is the relative right Kan extension of its restriction to the diagonal. We conclude that $\eta^*\colon \Psi_B(\dcat{D}^\sharp)\rt \Psi_A(\dcat{D}^\sharp)$ is an equivalence on the underlying bisimplicial spaces. Finally, it identifies the marking since $A[0, 1]^+\simeq B[0, 1]^+\simeq [0, 1]^\sharp$.
\end{proof}
Our next goal is to describe the composite functor $\Psi^\top\circ \Psi^\perp$ in terms of Construction \ref{con:constructing marked functors}.
\begin{construction}\label{con:functor T}
For each $m, n\geq 0$, let $I[m, n]$ denote the $(1, 1)$-category whose objects are quadruples of maps of linear orders
\begin{equation}\label{eq:abcd}
[a]\rto{\alpha} [m] \qquad \qquad [b]\star [a]\rto{\beta} [n] \qquad\qquad [t]\star [s]\rto{\gamma} [a]\qquad\qquad [t]\rto{\delta} [b].
\end{equation}
Given two such quadruples, a map between them is given by a quadruple of maps 
$$
[a]\rt [a'], \qquad [b]\rt [b'], \qquad [s]\rt [s'], \qquad [t]\rt [t']
$$
that intertwines $(\alpha, \beta, \gamma, \delta)$ and $(\alpha', \beta', \gamma', \delta')$. Note that $I[m, n]$ depends functorially on $[m, n]\in \Del^{\times 2}$, by postcomposing $\alpha$ and $\beta$. 

Now consider the natural map of bisimplicial spaces (in fact, sets)
$$\begin{tikzcd}
\pi\colon T[m, n]=\colim\limits_{I[m, n]} \, [s, t]\arrow[r] & {[m, n]}
\end{tikzcd}$$
defined as follows: to each $[s, t]\rt T[m, n]$ determined by a tuple $(\alpha, \beta, \gamma, \delta)$, we associate the map $\sigma\boxtimes \tau_0\colon [s, t]\rt [m, n]$ determined by 
$$
\sigma\colon [s] \rt [t]\star [s] \rto{\gamma} [a]\rto{\alpha} [m] \qquad \qquad \qquad \tau_0\colon [t]\rto{\delta} [b]\rt [b]\star[a]\rto{\beta} [n].
$$
We view $T[m, n]$ as a marked bisimplicial space by marking those $(0, 1)$-simplices corresponding to a quadruple $(\alpha, \beta, \gamma, \delta)$ (with $s=0, t=1$) such that either:
\begin{enumerate}[label=(\alph*)]
\item the map $\gamma\colon [1]\star [0]\rt [a]$ factors over a $0$-simplex of $[a]$.
\item the map $\beta\colon [b]\star[a]\rt [n]$ factors over a $0$-simplex of $[n]$.
\end{enumerate}
For every map $[m, n]\rt [m', n']$, the induced map $T[m, n]\rt T[m', n']$ preserves such marked arrows. 

Finally, let $T[0, 1]^+$ denote $T[0, 1]$, where \emph{all} vertical arrows are marked.
\end{construction}
\begin{lemma}\label{lem:composition via T}
There is a natural equivalence $\Psi_T\simeq \Psi_{K'}\circ \Psi_K$. 
\end{lemma}
\begin{proof}
Recall from Construction \ref{con:constructing marked functors} that $\Psi_{K'}$ arises as the right adjoint to a functor $\Phi_{K'}\colon \Fun^\sharp\big([1], \bisAn_{\downarrow}\big)\rt \Fun^\sharp\big([1], \bisAn_{\to}\big)$. The functor $\Phi_{K'}$ is uniquely determined by the fact that it preserves colimits and sends $\emptyset\rt [m, n]^\sharp$ to $\emptyset\rt [m, n]^\sharp$ and
$$
\begin{tikzcd} {[m, n]}\arrow[d] \\ {[m, n]^\sharp}\end{tikzcd}\longmapsto \begin{tikzcd} {K'[m, n]}\arrow[d] \\ {[m, n]^\sharp}\end{tikzcd}\qquad\qquad\qquad\qquad \begin{tikzcd} {[0, 1]^\sharp}\arrow[d] \\ {[0, 1]^\sharp}\end{tikzcd}\longmapsto \begin{tikzcd} {K'[1, 0]^+}\arrow[d] \\ {[0, 1]^\sharp}\end{tikzcd}
$$
Here $K'[m, n]$ was described in Variant \ref{var:other reflections}. Using Example \ref{ex:arrow dcat}, we can write $K'[m, n]$ as a (canonical) colimit of bisimplices 
$$
K'[m, n]=\colim_{\substack{\alpha\colon [a]\to [m]\\ \beta\colon [b]\star [a]\to [n]}} [a, b].
$$
Unraveling the definitions, the marked $(1, 0)$-simplices of $K'[m, n]$ correspond precisely to tuples of $\alpha\colon [1]\rt [m]$ and $\beta\colon [0]\star [1]\rt [n]$ where $\beta$ is degenerate. Furthermore, all $(1, 0)$-simplices in $K'[0, 1]$ are marked.

We can apply the same reasoning to the left adjoint $\Phi_K\circ \Phi_{K'}$: this functor preserves colimits and sends $\emptyset\rt [m, n]^\sharp$ to $\emptyset\rt [m, n]^\sharp$. It is therefore uniquely determined by its values on $[m, n]\rt [m, n]^\sharp$ and $[0, 1]^\sharp\rt [0, 1]^\sharp$. The value on $[m, n]$ is given by
$$
\Phi_K\big(K'[m, n]\big)=\colim_{\substack{\alpha\colon [a]\to [m]\\ \beta\colon [b]\star [a]\to [n]}} K[a, b]\simeq \colim_{\substack{\alpha\colon [a]\to [m]\\ \beta\colon [b]\star [a]\to [n]}} \Big(\colim_{\substack{\gamma\colon [t]\star [s]\to [a]\\ \delta\colon [t]\to [b]}} [s, t]\Big).
$$
In other words, it is precisely given by $T[m, n]$; unraveling the definitions shows that the marking coincides with that of Construction \ref{con:functor T}. Likewise, $\Phi_K\big(K'[0, 1]^+\big)$ can be identified with $T[0, 1]$ where all vertical arrows are marked. It follows that $\Phi_K\circ \Phi_{K'}\simeq \Phi_T$, so that by adjunction $\Psi_{T}\simeq \Psi_{K'}\circ\Psi_K$.
\end{proof}
\begin{proof}[Proof (of Theorem \ref{thm:main theorem})]
By Proposition \ref{prop:reflection functor} (and Variant \ref{var:other reflections}), we have functors
$$
\Psi^\perp\colon \cat{Fib}^{\mm{left, cart}}\rt \cat{Fib}^{\mm{cocart, right}}\qquad \text{and} \qquad \Psi^\top\colon \cat{Fib}^{\mm{cocart, right}}\rt \cat{Fib}^{\mm{left, cart}}.
$$
To see that they are mutually inverse, it suffices to prove that $\Psi^{\top}\circ \Psi^{\perp}\simeq \mm{id}$. Indeed, in this case, conjugating by the functor $\dcat{C}\mapsto (\dcat{C}^\mm{rev})^{(1, 2)-\op}$ provides the desired equivalence $\Psi^\perp\circ\Psi^\top\simeq \mm{id}$.
We construct the desired natural equivalence as a zig-zag
$$\begin{tikzcd}
\mm{id}\arrow[r, "\sim"{swap}, "\zeta^*"] & \Psi_A & \Psi_B\arrow[l, "\sim", "\eta^*"{swap}]\arrow[r, "\theta^*"] & \Psi_T\simeq \Psi^\top\circ \Psi^\perp.
\end{tikzcd}$$
Here the first two equivalences follow from Proposition \ref{prop:adding fluff} and the last identification follows from Lemma \ref{lem:composition via T}. The remaining natural transformation $\theta^*$ is induced by a natural map $\theta\colon T[m, n]\rt B[m, n]$ of vertically marked bisimplicial spaces over $[m, n]$. This map is defined as follows: for every $(s, t)$-simplex of $T[m, n]$, determined by a quadruple $(\alpha, \beta, \gamma, \delta)$ as in \eqref{eq:abcd}, consider the map
$$\begin{tikzcd}
\tau_1\colon [t]\arrow[r] & {[t]\star [s]}\arrow[r, "\gamma"] & {[a]}\arrow[r] & {[b]\star [a]}\arrow[r, "\beta"] & {[n]}.
\end{tikzcd}$$
Clearly $\tau_0(i)\leq \tau_1(i)$ for all $i$, where $\tau_0$ is defined in Construction \ref{con:functor T}. We therefore obtain a map $(\tau_0\leq \tau_1)\colon [s, t]\rt [0]\boxtimes \Fun([1], [n])$. Likewise, note that $\alpha\circ\gamma$ defines a map $\alpha\circ\gamma\colon [s, t]\rt \Ar[m]$ (cf.\ Example \ref{ex:arrow dcat}). Combining these two, we obtain a natural map
$$\begin{tikzcd}
\big(\alpha\circ\gamma, \tau_0\leq \tau_1\big)\colon [s, t]\arrow[r] & \Ar[m]\times \Big([0]\boxtimes \Fun([1], [n])\Big)=B[m, n]
\end{tikzcd}$$
for every $(s, t)$-simplex in $T[m, n]$. This defines the desired natural map of bisimplicial spaces $\theta\colon T[m, n]\rt B[m, n]$, compatible with the projections onto $[m, n]$. Notice that $\theta$ sends a marked $(0, 1)$-simplex in $T[m, n]$ (Construction \ref{con:functor T}) to a marked $(0, 1)$-simplex in $B[m, n]$ (Construction \ref{con:functor B}):
\begin{enumerate}[label=(\alph*)]
\item Given a marked $(0, 1)$-simplex corresponding to $(\alpha, \beta, \gamma, \delta)$ with $\gamma\colon [1]\star [0]\rt [a]$ constant, both $\alpha\circ\gamma$ and $\tau_1$ are constant. Consequently, its image in $B[m, n]$ is a marked $(0, 1)$-simplex as in part \ref{it:Bmark 1} of Construction \ref{con:functor B}.

\item Given a marked $(0, 1)$-simplex corresponding to $(\alpha, \beta, \gamma, \delta)$ with $\beta\colon [b]\star [a]\rt [n]$ constant, both $\tau_0$ and $\tau_1$ are constant with the same value. Consequently, the image in $B[m, n]$ is a marked $(0, 1)$-simplex as in part \ref{it:Bmark 2} of Construction \ref{con:functor B}.
\end{enumerate} 
Finally, note that in both $T[0, 1]^\sharp$ and $B[0, 1]^\sharp$, all vertical arrows are marked, so $\theta$ induces a natural map $T[0, 1]^\sharp\rt B[0, 1]^\sharp$ as well. 

For any (left, cart)-fibration $p\colon \dcat{D}^\natural\rt \dcat{C}^\sharp$, restriction along $\theta$ therefore defines a natural map of marked bisimplicial spaces
$$\begin{tikzcd}
\Psi_B(\dcat{D}^\natural)\arrow[rr, "\theta^*"]\arrow[rd] & & \Psi_T(\dcat{D}^\natural)\arrow[ld]\\
& \dcat{C}^\sharp.
\end{tikzcd}$$
Since both downwards pointing maps are (left, cart)-fibrations (with marked arrows being the cartesian arrows), it suffices to verify that $\theta^*$ induces equivalences between spaces of vertices and vertical arrows. On vertices, note that $T[0, 0]\simeq B[0, 0]\simeq [0, 0]$, so that the spaces of objects of $\Psi_B(\dcat{D}^\natural)$ and $\Psi_T(\dcat{D}^\natural)$ are both equivalent to the space of objects of $\dcat{D}$.

Unraveling the definitions, a vertical arrow in $\Psi_T(\dcat{D}^\natural)\simeq \Psi^\top\big(\Psi^\perp(\dcat{D}^\natural)\big)$ covering a vertical arrow $c_0\vto c_1$ in $\dcat{C}^\sharp$ is given by a diagram in $\dcat{D}$ of the form
$$\begin{tikzcd}
d_0\arrow[r] & d'_0\arrow[r, \vertcolor] & d''_0 \arrow[r, \vertcolor] & d_1.
\end{tikzcd}$$
Here first two maps live in the fiber over $c$ and the vertical arrow $d''_0\vto d_1$ is a cartesian lift of $c_0\vto c_1$. Since $\dcat{D}\rt \dcat{C}$ was a horizontal left fibration, this means that the arrow $d_0\rt d_0'$ is an equivalence.

On the other hand, a vertical arrow in $\Psi_B(\dcat{D}^\natural)$ is a sequence $d'_0\vto d''_0 \vto d_1$ of vertical maps where the first map lives in the fiber over $c$ and $d''_0\vto d_1$ is cartesian. The map $\theta^*$ can be identified with the map sending this sequence to $d_0 = d'_0\vto d''_0 \vto d_1$; it is then an equivalence by completeness of $\dcat{D}$.
\end{proof}

\section{Recollections on higher categories}\label{sec:higher cats}
In the remainder of the text, we will use Theorem \ref{thm:main theorem} to establish a version of straightening and unstraightening for cocartesian fibrations of $d$-categories. To facilitate this, we will start in this section by recalling some preliminary results on $d$-categories; our default model for $d$-categories will be in terms of iterated Segal spaces \cite{bar21}. Section \ref{sec:point set} provides a reminder on point-set models for $d$-categories; its main purpose is to allow us to construct explicit examples of higher categories (such as the $(d+1)$-category of $d$-categories) from enriched categories.

Sections \ref{sec:copresheaves} and \ref{sec:corep} describe the theory of copresheaves (of $d$-categories) on $(d+1)$-categories; this is mostly a recollection of the work of Boavida \cite{boa18}. For our purposes, it will be convenient to reformulate the results of loc.\ cit.\ as Proposition \ref{prop:space representability}: this provides a natural equivalence between the \emph{external} definition of copresheaves over a $(d+1)$-category $\cat{C}$, in terms of certain fibrations of $(d+1)$-fold simplicial spaces $X\rt \cat{C}$, and the \emph{internal} definition of copresheaves, as maps of $(d+1)$-categories $\cat{C}\rt \hcat{Cat}_{d}$.

\subsection{Higher categories via Segal spaces}
Recall the iterated Segal space model for $d$-categories \cite{bar21}:
\begin{definition}\label{def:d-cat}
A \emph{d-category} $\cat{C}$ is a $d$-fold category with the property that 
for each $\vec{n}_{k-1}$ in $\Del^{\times k-1}$, the $(d-k)$-fold category $\cat{C}(\vec{n}_{k-1}, 0, -, \dots, -)$ is a space, i.e.\ constant. We will write $\dCat\hooklongrightarrow \dfCat$ for the full subcategory spanned by the $d$-categories.
\end{definition}
\begin{remark}\label{rem:completeness}
The above definition of a $d$-category is a priori more restrictive than the usual one from \cite{bar21}: we ask all $\cat{C}(\vec{n}_{k-1}, -, \vec{n}_{d-k})$ to be complete, while this is usually only required if $\vec{n}_{d-k}=\vec{0}_{d-k}$. It follows from \cite[Lemma 2.8]{joh17} that the two definitions coincide, essentially by an Eckmann--Hilton type argument.
\end{remark}
If $\cat{C}$ is a $d$-category, recall that the (constant $(d-1)$-fold simplicial) space $\cat{C}(0)$ describes the \emph{space of objects} of $\cat{C}$. Given two objects $x_0, x_1\in \cat{C}$, we will write
$$
\Map_{\cat{C}}(x_0, x_1)= \cat{C}(1)\times_{\cat{C}(\{0\})\times \cat{C}(\{1\})} \{(x_0, x_1)\}
$$
for the \emph{mapping $(d-1)$-category} between them. For any triple of objects $x_0, x_1, x_2$, there is a composition map, determined by the commuting square
$$\begin{tikzcd}
\Map_{\cat{C}}(x_1, x_2)\times \Map_{\cat{C}}(x_0, x_1) \arrow[r, "\circ"] & \Map_{\cat{C}}(x_0, x_2) \\
\cat{C}(2)\times_{\cat{C}(\{0\})\times \cat{C}(\{1\})\times\cat{C}(\{2\})} \{(x_0, x_1, x_2)\}\arrow[u, "\sim"]\arrow[r, "d_1"] & \cat{C}(1)\times_{\cat{C}(\{0\})\times \cat{C}(\{1\})} \{(x_0, x_2)\}\arrow[u, "="{swap}]
\end{tikzcd}$$
Here the left vertical map is an equivalence by the Segal conditions.

To efficiently deal with the additional constancy conditions appearing in Definition \ref{def:d-cat}, let us introduce the following notation:
\begin{notation}
For a category $\cat{I}$, let us denote by $\lambda\colon \Del\times \cat{I}\rt \Del_{\cat{I}}$ the ($(\infty, 1)$-categorical) localization of $\Del\times \cat{I}$ at the morphisms in $\{0\}\times \cat{I}$. Recall that restriction and Kan extension determine adjoint pairs
$$\begin{tikzcd}[column sep=1.6pc]
\Fun\big((\Del\times \cat{I})^{\op}, \sS\big)\arrow[r, yshift=1ex, "\lambda_!"] \arrow[r, hookleftarrow, yshift=-1ex, "\lambda^*"{swap}] & \Fun(\Del_{\cat{I}}^{\op}, \sS)\arrow[r, hookrightarrow, yshift=1ex, "\lambda^*"] & \Fun\big((\Del\times \cat{I})^{\op}, \sS\big)\arrow[l, yshift=-1ex, "\lambda_*"]
\end{tikzcd}$$
that exhibit the category of presheaves on $\Del_{\cat{I}}$ both as a localization and a colocalization of the category of presheaves on $\Del\times \cat{I}$.
\end{notation}
\begin{lemma}\label{lem:simplex localization}
Let $\cat{I}$ be a $(1, 1)$-category with a terminal object $t$ and let $\lambda\colon \Del\times \cat{I}\rt \Del_{\cat{I}}$ denote the localization at the morphisms in $\{0\}\times \cat{I}$. Then the following assertions hold:
\begin{enumerate}
\item $\Del_{\cat{I}}$ is a $(1, 1)$-category with a terminal object.
\item For any $X\colon \big(\Del\times \cat{I}\big)^{\op}\rt \sS$, consider the natural map of presheaves on $\Del\times \cat{I}$
$$\begin{tikzcd}
X\times_{i_*i^*X} i_*\big(X(0, t)\big)\arrow[r] & X
\end{tikzcd}$$
where $i_*$ denote the right Kan extension along $i\colon \{0\}\times \cat{I}\hookrightarrow \Del\times \cat{I}$ and $i^*X\rt X(0, t)$ denotes the natural map from the presheaf $X(0, -)$ on $\cat{I}$ to the constant presheaf with value $X(0, t)$. This map exhibits the counit of the adjoint pair $(\lambda^*, \lambda_*)$.
\end{enumerate}
\end{lemma}
In fact, in this case $\Del_{\cat{I}}$ can be identified with a (non-full) subcategory of $\Del\wr \cat{I}$, given by the image of the functor $\delta\colon \Del\times \cat{I}\rt \Del\wr \cat{I}$ from \cite{ber07}.
\begin{proof}
For each $([m], i)\in \Del\times \cat{I}$, let $h_{[m], i}=\Map\big(-, ([m], i)\big)$ be the representable presheaf. For assertion (1), consider the following (homotopy) pushout of (space-valued) presheaves
$$\begin{tikzcd}
\coprod_{[0]\to [m]} h_{[0], i}\arrow[r]\arrow[d] & h_{[m], i}\arrow[d]\\
\coprod_{[0]\to [m]} h_{[0], t}\arrow[r] & F.
\end{tikzcd}$$
Here the top map is given pointwise by the inclusion of the subset of maps $([n], j)\rt ([m], i)$ such that the map $[n]\rt [m]$ factors over $[0]$. Using this, it follows immediately that the presheaf $F$ takes values in sets and sends maps in $\{0\}\times \cat{I}$ to isomorphisms. Consequently, $F\simeq \lambda^*G$ for some set-valued presheaf $G$ on $\Del_{\cat{I}}$. 

On the other hand, the left Kan extension functor $\lambda_!$ sends the left vertical map to an equivalence, so that we find $h_{\lambda([m], i)}\simeq \lambda_!h_{[m], i}\simeq \lambda_!F$. We therefore find that the representable presheaf $h_{\lambda([m], i)}$ is equivalent to the set-valued presheaf $\lambda_!\lambda^*G\simeq G$, so that $\Del_{\cat{I}}$ is a $(1, 1)$-category. In the special case where $[m]=0$, we see that $F$ is the terminal presheaf, so that the image of $([0], i)$ defines a terminal object in $\Del_{\cat{I}}$.

For (2), let us abbreviate $Y=X\times_{i_*i^*X} i_*\big(X(0, t)\big)$. Since $i$ is fully faithful, $i^*i_*\simeq \mm{id}$ and the restriction of $Y$ along $i\colon \{0\}\times \cat{I}\rt \Del\times\cat{I}$ coincides with the constant presheaf with value $X(0, t)$. This means that $Y$ is contained in the essential image of $\lambda^*$, or equivalently, that $Y$ is a $\lambda^*\lambda_*$-colocal object.

It therefore remains to verify that the projection $Y\rt X$ is a colocal equivalence, i.e.\ that it is sent to an equivalence by $\lambda_*$. Since $Y\rt X$ is the base change of the map $i_*(X(0, t))\rt i_*i^*X$, it suffices to to verify that $\lambda_*i_*(X(0, t))\rt \lambda_*i_*i^*X$ is an equivalence. By part (1), the functor $\lambda i$ decomposes as $\cat{I}\rt \ast\rt \Del_{\cat{I}}$, where the second functor is the inclusion of the terminal object. The result then follows from the fact that the natural map
$$\begin{tikzcd}
X(0, t)\simeq \lim_{\cat{I}^{\op}} X(0, t)\arrow[r] & \lim_{\cat{I}^{\op}} i^*X
\end{tikzcd}$$
is an equivalence since $t$ was the terminal object of $\cat{I}$.
\end{proof}
\begin{notation}\label{not:localization of delta-d}
We inductively define $\Del_d$ by $\Del_0=\ast$ and $\Del_{d}=\Del_{\Del_{d-1}}$, so that there is composition of localization functors of $(1, 1)$-categories
\begin{equation}\label{diag:sequence of localizations}\begin{tikzcd}
\delta_d\colon \Del^{\times d}\arrow[r] & \Del^{\times d-2}\times \Del_2\arrow[r] & \dots \arrow[r] & \Del\times \Del_{d-1}\arrow[r] & \Del_d.
\end{tikzcd}\end{equation}
By construction, restriction along $\delta_d$ determines a fully faithful inclusion
$$\begin{tikzcd}
\delta_d^*\colon \Fun(\Del_d^{\op}, \sS)\arrow[r, hook] & \Fun(\Del^{\times d, \op}, \sS)
\end{tikzcd}$$
whose essential image consists of $d$-fold simplicial spaces $X$ such that each $(d-k)$-fold simplicial space $X(\vec{n}_{k-1}, 0, -, \dots, -)$ is constant. We will systematically omit $\delta_d^*$ and simply identify presheaves on $\Del_d$ with $d$-fold simplicial spaces using this inclusion. Likewise, we can identify presheaves on $\Del_d$ with simplicial object in presheaves on $\Del_{d-1}$ which are constant in simplicial degree $0$, etcetera.
\end{notation}
In these terms, we have a pullback square of fully faithful functors
$$\begin{tikzcd}
\dCat\arrow[r, hook]\arrow[d, hook] & \Cat^{\otimes d}\arrow[d, hook] \\
\Fun(\Del_d^{\op}, \sS)\arrow[r, hook, "\delta_d^*"] & \Fun(\Del^{\times d, \op}, \sS).
\end{tikzcd}$$
Since the bottom and the right functors are right adjoint functors between presentable categories, the inclusion $\Cat_d\rt \Cat^{\otimes d}$ admits a left adjoint. It also admits a right adjoint, by the following observation:
\begin{lemma}\label{lem:dcat in d-fold cat}
Let $r_d\colon \Fun(\Del^{\times d, \op}, \sS)\rt \Fun(\Del_d, \sS)$ be the right adjoint to $\delta_d^*$, taking right Kan extension along $\delta_d\colon \Del^{\times d, \op}\rt \Del_d^{\op}$. Then $r_d$ maps the full subcategory $\Cat^{\otimes d}$ to $\Cat_d$.
\end{lemma}
\begin{proof}
We proceed by induction, the case $d=1$ being clear. Let us now take $d>1$ and assume that $r_{d-1}$ sends $(d-1)$-fold categories to $d$-categories. Suppose that $X\colon \Del^{\times d, \op}\rt \sS$ is a $d$-fold category and factor the map $\delta_d$ as
$$\begin{tikzcd}[column sep=3pc]
\Del^{\times d}\arrow[r, "\mm{id}\times \delta_{d-1}"] & \Del\times\Del_{d-1}\arrow[r, "\lambda"] &  \Del_{\Del_{d-1}}=\Del_d.
\end{tikzcd}$$ 
The right Kan extension along the functor $\mm{id}\times \delta_{d-1}$ sends $X$ to the simplicial object $X'\colon \Del^{\op}\rt \Fun(\Del_{d-1}^{\op}, \sS)$ given by $X'(n)=r_{d-1}(X(n))$. Each $X'(n)$ is then a $(d-1)$-category by inductive hypothesis and the simplicial object $X'$ satisfies the complete Segal conditions since $r_{d-1}$ is a right adjoint.

Next, we compute $r_d(X)=\lambda_*(X')$ using Lemma \ref{lem:simplex localization}(2). Denoting by $i\colon \{0\}\times \Del_{d-1}\hookrightarrow \Del\times \Del_{d-1}$ the inclusion, we have that
$$
r_d(X)=X'\times_{i_*i^*X'} i_*(X'(\vec{0}_d)),
$$
at least after precomposing with the localization $\lambda\colon \Del\times\Del_{d-1}\rt \Del_d$. The value of this fiber product at $[m]\in \Del$ is given by the presheaf on $\Del_{d-1}$ of the form
$$
r_d(X)(m, -) = X'(m, -)\times_{X'(0, -)^{\times m+1}} X'\big(0, \vec{0}_{d-1}\big)^{\times m+1}.
$$
Here the map $X'(m, -)\rt X'(0, -)^{\times m+1}$ arises from the vertex inclusions $\{k\}\hookrightarrow [m]$ for all $k\in [m]$. This implies that $r_d(X)$ satisfies the Segal condition in the first simplicial direction (as the spine inclusions induce bijections on objects). Furthermore, each $r_d(X)(m)$ is a $(d-1)$-category, since all terms in the above pullback diagram are $(d-1)$-categories.

It remains to verify the completeness condition. By Remark \ref{rem:completeness}, it suffices to verify the completeness condition after setting the last variables equal to $\vec{0}_{d-1}$. But in that case the above formula shows that $r_d(X)(-, \vec{0}_{d-1})\simeq X'(-, \vec{0}_{d-1})$, for which we already verified that it was a complete Segal space.
\end{proof}
%
\begin{remark}
In fact, one sees that for any $d$-fold category $\dcat{C}$, the counit map $r_d(\dcat{C})\hooklongrightarrow \dcat{C}$ is a $d$-fold simplicial subspace, given in degree $(m_1, \dots, m_d)$ by the path components $x$ satisfying (inductively) the following two conditions:
\begin{enumerate}
\item $x$ is contained in $r_{d-1}(\dcat{C}(m_1))\hooklongrightarrow \dcat{C}(m_1)$.
\item For each $i=0, \dots, m_1$, the image of $x$ in $\dcat{C}(\{i\}, m_2, \dots, m_d)$ is contained in the path components $\dcat{C}(\{i\}, 0, \dots, 0)\subseteq \dcat{C}(\{i\}, m_2, \dots, m_d)$.
\end{enumerate}
\end{remark}
\begin{corollary}\label{cor:internal hom}
The category $\Cat_d$ is cartesian closed. For any two $d$-categories $\cat{C}$ and $\cat{D}$, the internal mapping object $\Fun_d(\cat{C}, \cat{D})$ in $\Cat_d$ is the simplicial subspace of the internal mapping object $\cat{D}^{\cat{C}}$ in $d$-fold simplicial spaces given by
$$
\Fun_d(\cat{C}, \cat{D})=r_d\big(\cat{D}^{\cat{C}}\big)\subseteq \cat{D}^{\cat{C}}.
$$
\end{corollary}
\begin{example}\label{ex:interval}
Let $\cat{A}$ be a $(d-1)$-category. We define $[1]_{\cat{A}}$ to be the $d$-category associated to the $d$-fold category $[1]\boxtimes \cat{A}$ by the left adjoint of Lemma \ref{lem:dcat in d-fold cat}. In other words, maps of $d$-categories $[1]_{\cat{A}}\rt \cat{C}$ are equivalent to maps of $(d-1)$-categories $\cat{A}\rt \cat{C}(1)$. Explicitly, $[1]_{\cat{A}}$ is the $d$-fold simplicial space $[1]\boxtimes \cat{A}\amalg_{\{0, 1\}\boxtimes \cat{A}} \{0, 1\}$, i.e.\ it is a $d$-category with two objects $0, 1$ with trivial endomorphisms and $\Map_{[1]_{\cat{A}}}(0, 1)=\cat{A}$.

If $\cat{C}$ is a $d$-category, then the $d$-category 
$$
\Fun_d\big([1]_{\cat{A}}, \cat{C}\big)\simeq r_d\big(\cat{C}^{[1]\boxtimes \cat{A}}\big)
$$
can be described as follows. Its space of objects is given by the space of maps $\alpha\colon [1]_{\cat{A}}\rt \cat{C}$, i.e.\ of tuples $\alpha_0, \alpha_1\in \cat{C}$ and $\alpha\colon \cat{A}\rt \Map_{\cat{C}}(\alpha_0, \alpha_1)$. Unraveling the definitions, one sees that the $(d-1)$-category of maps between two such $\alpha, \beta\colon [1]_{\cat{A}}\rt \cat{C}$ is given by the pullback
$$\begin{tikzcd}
\Map_{\Fun_d([1]_{\cat{A}}, \cat{C})}(\alpha, \beta)\arrow[r]\arrow[d] & \Map_{\cat{C}}(\alpha_0, \beta_0)\arrow[d, "\beta_*"]\\
 \Map_{\cat{C}}(\alpha_1, \beta_1)\arrow[r, "\alpha^*"] &  \Fun_{d-1}\big(\cat{A}, \Map_{\cat{C}}(\alpha_0, \beta_1)\big).
\end{tikzcd}$$
\end{example}

\begin{remark}\label{rem:(d-1)-cat vs d-cat}
There is a fully faithful functor $\iota\colon \Cat_{d-1}\hooklongrightarrow \Cat_d$, sending a $(d-1)$-category $\cat{C}$ to the $d$-fold simplicial space $\iota(\cat{C})(m_1, \dots, m_d)=\cat{C}(m_1, \dots, m_{d-1})$ which is constant in the last variable. This functor has a left adjoint $|-|$ sending a $d$-fold category $\cat{C}$ to the $(d-1)$-fold category 
$$
|\cat{C}|(m_1, \dots, m_{d-1})=\colim_{m_d\in \Del^{\op}}\cat{C}(m_1, \dots, m_d)
$$
and a right adjoint $\core_{d-1}$ given by $\core_{d-1}(\cat{C})(m_1, \dots, m_{d-1})=\cat{C}(m_1, \dots, m_{d-1}, 0)$, which we will refer to as the \emph{$(d-1)$-core}.
\end{remark}

\subsection{Point-set models}\label{sec:point set}
To produce examples of $d$-categories, it will be more useful to think of $d$-categories as categories \emph{enriched} in $(d-1)$-categories. This is probably best done using the theory of enriched $\infty$-categories \cite{gep15, hin20}. We will take a more rigid, model-categorical approach instead (which has the advantage of not relying on Lurie's (un)straightening equivalence). Recall from Notation \ref{not:localization of delta-d} how we view presheaves on $\Del_d$ as $d$-fold simplicial spaces using the localizations $\Del^{\times d}\rt \Del\times \Del_{d-1}\rt \Del_d$.

\begin{definition}\label{def:categorical algebra}
A functor $X\colon \Del_d\rt \sS$ is said to be a \emph{$d$-categorical algebra} if it satisfies the following conditions:
\begin{enumerate}
\item Each $X(n)\colon \Del_{d-1}\rt \sS$ is a $(d-1)$-category (and $X(0)$ is a space).

\item $X$ satisfies the Segal condition, i.e.\ $X(n)\rt X(1)\times_{X(0)} \dots\times_{X(0)} X(1)$ is an equivalence.
\end{enumerate}
\end{definition}
\begin{remark}\label{rem:dwyer-kan equivalence}
The category of $d$-categories can be identified with the full subcategory of the category of $d$-categorical algebras, on those $d$-categorical algebras that are furthermore complete \cite{gep15, lur09g}. The inclusion of $d$-categories into $d$-categorical algebras admits a left adjoint $L_\mm{DK}$, which localizes at the class of \emph{Dwyer--Kan equivalences}, i.e.\ maps $Y\rt X$ with the following two properties (see loc.\ cit.):
\begin{enumerate}
\item \emph{fully faithfulness}: $Y(1)\rt X(1)\times_{X(0)^{\times 2}} Y(0)^{\times 2}$ is an equivalence of $(d-1)$-categories.

\item \emph{essential surjectivity}: the induced functor on homotopy categories $\ho(Y)\rt \ho(X)$ is essentially surjective.
\end{enumerate}
In particular, for each $d$-categorical algebra $X$, the map $X\rt L_\mm{DK}(X)$ is a Dwyer--Kan equivalence. In fact, the explicit construction shows that $X\rt L_\mm{DK}(X)$ is not just essentially surjective, but that $X(0)\rt L_\mm{DK}(X)(0)$ is already surjective on $\pi_0$ \cite[Proposition 1.2.27]{lur09g}.
\end{remark}
\begin{definition}
Recall that there is a localization functor $\sSet\rt \sS$ inverting the weak equivalences for the Kan--Quillen model structure. We will say that a functor $X\colon \Del_d^{\op}\rt \sSet$ is a \emph{$d$-category model} (resp.\ $d$-categorical algebra model) if the composed functor $\Del_d^{\op}\rt \sSet\rt \sS$ is a $d$-category (resp.\ $d$-categorical algebra).
\end{definition}
\begin{construction}\label{con:d-cat model structure}
The \emph{$d$-categorical model structure} on $\mm{sPSh}(\Del_d)=\Fun(\Del_d^{\op}, \sSet)$ is the left proper, combinatorial model structure characterized uniquely by the following properties \cite[Proposition 1.5.4]{lur09g}:
\begin{enumerate}
\item The cofibrations are the monomorphisms, and in particular every object is cofibrant.

\item An object $X$ is fibrant if and only if it is fibrant in the injective model structure and it is a $d$-category model.

\item A weak equivalence between $d$-category models is a levelwise weak equivalence of simplicial sets (in the Kan--Quillen model structure).
\end{enumerate}
We will write $\CompSegS_d$ for the full subcategory of $\mm{sPSh}(\Del_d)$ on the injectively fibrant $d$-category models. 
\end{construction}
\begin{remark}
The $d$-categorical model structure is a left Bousfield localization of the injective model structure on $\mm{sPSh}(\Del_d)$. Since the latter is a model for the $\infty$-category of functors $\Del_d^{\op}\rt \sS$ \cite[Proposition A.3.4.13]{lur09}, it follows immediately that $\dCat$ is equivalent to the localization of $\CompSegS_d$ at the $d$-categorical (equivalently: pointwise) weak equivalences.
\end{remark}
\begin{remark}\label{rem:internal hom model}
The injective model structure on $\mm{sPSh}(\Del_d)$ is a monoidal model structure for the cartesian product \cite[Proposition 2.9]{rez09}, which presents the cartesian monoidal structure on $\Fun(\Del_d^{\op}, \sS)$. Corollary \ref{cor:internal hom} shows that $\Fun(\Del_d^{\op}, \sS)\leftrightarrows \dCat$ is a monoidal localization for the cartesian product; this implies that the $d$-categorical model structure on $\mm{sPSh}(\Del_d)$ is monoidal as well. In particular, this implies that the category $\CompSegS_d$ is cartesian closed.
\end{remark}
\begin{construction}\label{con:catalg}
Let $\SegS_d\subseteq \mm{sPSh}(\Del_d)\subseteq \Fun\big(\Del^{\op}, \mm{sPSh}(\Del_{d-1})\big)$ denote the full subcategory spanned by the functors $X\colon \Del^{\op}\rt \CompSegS_{d-1}$ that are $d$-categorical algebra models, i.e.\ explicitly:
\begin{itemize}
\item Each $X(n)\colon \Del_{d-1}^{\op}\rt \sSet$ is an injectively fibrant $(d-1)$-category model and $X(0)$ is constant on a Kan complex.
\item $X$ satisfies the Segal conditions in the first simplicial variable.
\end{itemize}
We will say that a map $Y\rt X$ is a Dwyer--Kan equivalence if the corresponding map of categorical algebras is a Dwyer--Kan equivalence as in Remark \ref{rem:dwyer-kan equivalence}. Furthermore, we will say that a map $Y\rt X$ is an \emph{isofibration} if:
\begin{enumerate}
\item Each $Y(n)\rt X(n)$ is a fibration in $\CompSegS_{d-1}$, i.e.\ an injective fibration of simplicial presheaves on $\Del_{d-1}$.

\item The induced map on homotopy categories $\ho(Y)\rt \ho(X)$ is an isofibration.
\end{enumerate}
With these classes of maps, $\SegS_d$ becomes a category of fibrant objects as in Appendix \ref{sec:fib obj}. A map $Y\rt X$ is an \emph{acyclic fibration} if and only if each $Y(n)\rt X(n)$ is an injective fibration, $Y(0)\rt X(0)$ is surjective on $\pi_0$ and $Y(1)\rt X(1)\times_{X(0)^\times 2} Y(0)^{\times 2}$ is a weak equivalence. To factor a map $Y\rt X$ in $\SegS_{d}$ into a weak equivalence followed by a fibration, one can simply factor it into a weak equivalence followed by a fibration with respect to the $d$-categorical model structure on $\mm{sPSh}(\Del_d)$; such fibrations are isofibrations by the right lifting property against $\{0\}\rt H$, where $H$ is the simplicial set from Section \ref{sec:dcat}.
\end{construction}
\begin{construction}
Let us write $\cat{Cat}(\CompSegS_{d-1})$ for the $(1, 1)$-category of categories (strictly) enriched over the cartesian monoidal category $\CompSegS_{d-1}$. There is a fully faithful functor $\rN\colon \Cat(\CompSegS_{d-1})\hooklongrightarrow \SegS_d$ given by the \emph{nerve}: it sends an enriched category $\scat{C}$ to the simplicial diagram in $\CompSegS_{d-1}$ given by
$$
\rN(\scat{C})(n) = \coprod_{c_0, \dots, c_n} \Map_{\scat{C}}(c_{n-1}, c_n)\times \dots\times \Map_{\scat{C}}(c_0, c_1).
$$
In particular, $\rN(\scat{C})(0)$ is simply the set of objects of $\scat{C}$. We will say that a map of enriched categories is a Dwyer--Kan equivalence, resp.\ an isofibration, if its image under the nerve functor is such. This makes $\cat{Cat}(\CompSegS_{d-1})$ a category of fibrant objects.

More precisely, the category of categories enriched in $\mm{sPSh}(\Del_{d-1})$ carries a model structure \cite[Proposition A.3.2.4]{lur09}, since the $(d-1)$-categorical model structure on $\mm{sPSh}(\Del_{d-1})$ is excellent in the sense of \cite[Definition A.3.2.16]{lur09}. It follows from \cite[Theorem A.3.2.24]{lur09} that $\Cat(\CompSegS_{d-1})$ (with the above classes of Dwyer--Kan equivalences and isofibrations) is the category of fibrant objects associated to this model category.
\end{construction}
\begin{proposition}\label{prop:comparing models}
The three categories of fibrant objects considered  above fit into a diagram of fully faithful inclusions
$$\begin{tikzcd}
\cat{Cat}(\CompSegS_{d-1})\arrow[r, hook, "\rN"]  & \SegS_d \arrow[r, hookleftarrow] & \CompSegS_d
\end{tikzcd}$$
that preserve fibrations and weak equivalences. The induced functors on localizations are equivalences, so that each of these relative categories is a model for $\cat{Cat}_d$.
\end{proposition}
\begin{proof}
The right inclusion $\CompSegS_d\rt \SegS_{d}$ has a homotopy inverse $L_\mm{DK}$, given by taking a fibrant replacement of a $d$-categorical algebra model in the $d$-categorical model structure. The resulting objects are then related by a Dwyer--Kan equivalence.

Furthermore, the composite functor $L_\mm{DK}\circ \rN\colon \cat{Cat}(\CompSegS_{d-1})\rt \CompSegS_d$ can be identified with the composite
$$\begin{tikzcd}[column sep=3pc]
\cat{Cat}(\mm{sPSh}(\Del_{d-1}))\arrow[r, "\rN"] & \cat{Seg}_{\mm{sPSh}(\Del_{d-1})}\arrow[r, "\mm{UnPre}"] & \mm{sPSh}(\Del_{d})\arrow[r, "(-)^\mm{fib}"] & \CompSegS_d
\end{tikzcd}$$
where $\cat{Seg}_{\mm{sPSh}(\Del_{d-1})}$ is the category of Segal precategories from \cite[\S 2]{lur09g}.
Here the first functor $\rN$ is a right Quillen equivalence preserving weak equivalences \cite[Theorem 2.2.16, Remark 2.2.19]{lur09g}. The functor $\mm{UnPre}$ is a left Quillen equivalence preserving weak equivalences by \cite[Proposition 2.3.1, 2.3.9, Lemma 2.3.14]{lur09g}: this uses the fact that the subcategory inclusion $\mm{sPSh}(\Del_{d})\hooklongrightarrow \Fun\big(\Del^{\op}, \mm{sPSh}(\Del_{d-1})\big)$ is a Quillen equivalence between the $d$-categorical model structure and the `complete Segal model structure' of loc.\ cit., which detects cofibrations and weak equivalences. The last functor takes a fibrant replacement, so it follows that $L_\mm{DK}\circ \rN$ induces an equivalence on localizations.
\end{proof}
\begin{remark}
Let us write $\widehat{\rN}\colon \cat{Cat}(\CompSegS_{d-1})\rt \dCat$ for the composite functor $L_\mm{DK}\circ N$. By Remark \ref{rem:dwyer-kan equivalence}, $\widehat{\rN}(\scat{C})$ is a $d$-category which comes with a Dwyer--Kan equivalence $\rN(\scat{C})\rt \widehat{\rN}(\scat{C})$ from the nerve of $\scat{C}$; in particular, every object of $\widehat{\rN}(\scat{C})$ is homotopic to an object from $\scat{C}$ and for any two objects in $\scat{C}$, $\Map_{\scat{C}}(c_0, c_1)$ is a model for the $(d-1)$-category of maps between their images in $\widehat{\rN}(\scat{C})$.
\end{remark}

\begin{definition}\label{def:d+1cat of dcats}
The (large) \emph{$(d+1)$-category $\mb{Cat}_{d}$ of $d$-categories} is given by $\widehat{\rN}(\CompSegS_{d})$, where $\CompSegS_{d}$ is enriched over itself as in Remark \ref{rem:internal hom model}.
\end{definition}
\begin{remark}
Note that the $1$-category underlying the $(d+1)$-category $\mb{Cat}_{d}$ arises from the simplicially enriched category of fibrant-cofibrant objects in the $d$-categorical model structure; in particular, it is equivalent to $\dCat$ \cite[Lemma A.3.6.17]{lur09}.

It follows from \cite{hau15} that our definition of $\mb{Cat}_d$ (via point-set models) coincides with the definition of $\mb{Cat}_d$ as the $d$-category-enriched $\infty$-category associated to $\dCat$, viewed as enriched over itself using the cartesian closed structure.
\end{remark}

\subsection{Copresheaves}\label{sec:copresheaves}
If $\cat{C}$ is a $(d+1)$-category, we define the $(d+1)$-category of \emph{copresheaves} (of $d$-categories) of $\cat{C}$ to be the functor $(d+1)$-category $\Fun_{d+1}(\cat{C}, \mb{Cat}_{d})$, i.e.\ as the internal mapping object in the category of (large) $(d+1)$-categories. This $(d+1)$-category can be presented explicitly in terms of the enriched categories model for $(d+1)$-categories:
\begin{proposition}[{\cite[Proposition A.3.4.13]{lur09}}]\label{prop:mapping enriched categories}
Let $\scat{C}$ be a small $\CompSegS_{d}$-enriched category and denote by $\Fun_{\enr}\big(\scat{C}, \CompSegS_{d}\big)^\circ$ the $\CompSegS_{d}$-enriched category of enriched functors $\scat{C}\rt \mm{sPSh}(\Del_d)$ that are fibrant-cofibrant in the projective model structure (for the $d$-categorical model structure on simplicial presheaves). Then the map
$$\begin{tikzcd}
\mm{ev}\colon \widehat{\rN}\Big(\Fun_{\enr}\big(\scat{C}, \CompSegS_{d}\big)^\circ\Big)\times \widehat{\rN}\big(\scat{C}\big)\arrow[r] & \widehat{\rN}(\CompSegS_{d})
\end{tikzcd}$$
exhibits $\widehat{\rN}(\Fun_{\enr}\big(\scat{C}, \CompSegS_{d}\big)^\circ)$ as the internal mapping object in the category of (large) $(d+1)$-categories. 
\end{proposition}
In particular, using the point-set model of $\CompSegS_d$-enriched categories, a copresheaf $\cat{C}\rt \hcat{Cat}_d$ can be modeled by an enriched functor $\scat{C}\rt \CompSegS_d$. In addition to this `internal' definition of copresheaves, one can also describe copresheaves externally in terms of Segal objects, following \cite{boa18} (see also \cite{ras21} for an in-depth discussion):
\begin{definition}\label{def:lfib}
Let $\cat{C}$ be a $(d+1)$-categorical algebra. A map of $(d+1)$-fold simplicial spaces $X\rt \cat{C}$ is said to be a \emph{Segal copresheaf} if the following two conditions hold:
\begin{enumerate}
\item The $d$-fold simplicial space $X(0)$ is a $d$-category.
\item For each $[n]\in \Del$, there is a pullback square of $d$-categories
$$\begin{tikzcd}
X(n)\arrow[r]\arrow[d] & X(\{0\})\arrow[d]\\
\cat{C}(n)\arrow[r] & \cat{C}(\{0\}).
\end{tikzcd}$$
\end{enumerate}
Let us write $\SegCoPSh_{d}\subseteq \Fun\big([1]\times \Del^{\times d+1, \op}, \sS\big)$ for the full subcategory spanned by the Segal copresheaves $X\rt \cat{C}$ where $\cat{C}$ is a $(d+1)$-category. The codomain projection $\pi\colon \SegCoPSh_{d}\rt \cat{Cat}_{d+1}$ is a cartesian fibration, since the pullback of a Segal copresheaf $X\rt \cat{C}$ along a map of $(d+1)$-categories $\cat{C}'\rt \cat{C}$ is again a Segal copresheaf. We will write $\SegCoPSh_{d}(\cat{C})$ for the fiber over $\cat{C}$.
\end{definition}
\begin{warning}\label{war:lfib}
When $d=0$, the domain of a Segal copresheaf over a $1$-category is itself a $1$-category \cite[Corollary 1.19]{boa18}; in this case, Segal copresheaves can be identified with left fibrations (or, in the terminology of Section \ref{sec:cocart fib}, \emph{$0$-cocartesian fibrations}). Note that for $d\geq 1$, the domain of a Segal copresheaf $X\rt \cat{C}$ over a $(d+1)$-category is typically \emph{not} itself a $(d+1)$-category: if the domain were a $(d+1)$-category, then all the fibers of the fibration would be spaces.
\end{warning}
In the remainder of this section, we will recall the work of Boavida \cite{boa18}, which shows that the category of Segal copresheaves $\SegCoPSh_{d}(\cat{C})$ coincides with the $1$-category underlying $\Fun_{d+1}(\cat{C}, \mb{Cat}_{d})$. We will make use of point-set models to compare to Proposition \ref{prop:mapping enriched categories}:
\begin{construction}\label{cons:enriched functor fibration}
\sloppy Consider the functor $\Cat(\CompSegS_{d})^{\op}\rt \cat{FibCat}$ sending each $\CompSegS_{d}$-enriched category $\scat{C}$ to the category $\Fun_{\enr}(\scat{C}, \CompSegS_{d})$ of enriched functors. Note that this is the full subcategory of fibrant objects in the projective model structure on \mbox{$\Fun\big(\scat{C}, \mm{sPSh}(\Del_{d})\big)$}.
We will write
$$\begin{tikzcd}
\EnrFun_d\arrow[r] & \cat{Cat}(\CompSegS_{d})
\end{tikzcd}$$
for its Grothendieck construction. Explicitly, an object $(\scat{C}, F)\in\EnrFun_d$ is a $\CompSegS_{d}$-enriched category $\scat{C}$ together with an enriched functor $F\colon \scat{C}\rt \CompSegS_{d}$. 
\end{construction}
\begin{lemma}\label{lem:enrfun is ho cart fib}
The map $\EnrFun_d\rt \cat{Cat}(\CompSegS_{d})$ is a homotopy cartesian fibration between categories of fibrant objects, in the sense of Definition \ref{def:homotopically cart fib}.
\end{lemma}
\begin{proof}
We verify conditions \ref{it:homotopy invariance} and \ref{it:factorization} of Lemma \ref{lem:fibcat grothendieck}. For \ref{it:homotopy invariance}, note that for each weak equivalence $f\colon \scat{C}\rt \scat{D}$, $f^*\colon \Fun_{\enr}(\scat{D}, \CompSegS_{d})\rt \Fun_{\enr}(\scat{C}, \CompSegS_{d})$ is induced by a Quillen equivalence \cite[Proposition A.3.3.8]{lur09}. 
For \ref{it:factorization}, we use Remark \ref{rem:factorization model cat case}: we can factor any $f\colon \scat{C}\rt \scat{D}$ into a trivial cofibration $i\colon \scat{C}\rt \scat{C}'$ (in the model structure on enriched categories) followed by a fibration $p\colon \scat{C}'\twoheadrightarrow \scat{D}$. Then $i_!\colon \Fun_{\enr}(\scat{C}, \mm{sPSh}(\Del_d))\leftrightarrows \Fun_{\enr}(\scat{C}', \mm{sPSh}(\Del_d))\colon i^*$ is a Quillen equivalence where $i_!$ preserves all weak equivalences. Indeed, this follows from the fact that $i_!$ is left Quillen for the injective model structure as well, as a consequence of \cite[Remark A.2.9.27, Proposition A.3.3.9]{lur09}.
\end{proof}
\begin{construction}
Let $X\colon \Del^{\op}\rt \CompSegS_{d}$ be an object of $\SegS_{d+1}$, i.e.\ $X$ satisfies the Segal conditions and $X(0)\in \mm{sPSh}(\Del_{d})$ is constant on a simplicial set. The category $\mm{sPSh}(\Del\times \Del_{d})_{/X}$ can be endowed with the \emph{projective covariant model structure}, in which:
\begin{itemize}
\item trivial fibrations are maps $Y\rt Y'$ (over $X$) such that each $Y(n)\rt Y'(n)$ is an injective trivial fibration in $\mm{sPSh}(\Del_{d})$.
\item fibrant objects are $Y\rt X$ such that each $Y(n)\rt X(n)$ is an injective fibration between injectively fibrant $(d-1)$-category models, such that \mbox{$Y(n)\rt X(n)\times_{X(\{0\})}^h Y(\{0\})$} is a weak equivalence.
\end{itemize}
The identity functor determines a Quillen equivalence to the \emph{injective covariant model structure}, defined similarly but as a left Bousfield localization of the injective model structure on $\mm{sPSh}(\Del\times\Del_d)_{/X}$.

We will write $\SegCoPShproj_d(X)$ for the category of fibrant objects in the projective covariant model structure. This determines a (pseudo-)functor $\SegS_{d+1}^{\op}\rt \cat{FibCat}$ sending $X\mapsto \SegCoPShproj_d(X)$, where $\SegS_{d+1}$ is a category of fibrant objects as in Construction \ref{con:catalg}. Let us write $\SegCoPShproj_d\rt \SegS_{d+1}$ for its Grothendieck construction.
\end{construction}
\begin{lemma}\label{lem:left fib proj}
The map $\SegCoPShproj_d\rt \SegS_{d+1}$ is a homotopy cartesian fibration between categories of fibrant objects, in the sense of Definition \ref{def:homotopically cart fib}.
\end{lemma}
\begin{proof}
We verify conditions \ref{it:homotopy invariance} and \ref{it:factorization} of Lemma \ref{lem:fibcat grothendieck}: for each Dwyer--Kan equivalence $f\colon X\rt X'$, there is an adjoint pair $f_!\colon \mm{sPSh}(\Del\times \Del_{d})_{/X}\leftrightarrows \mm{sPSh}(\Del\times \Del_{d})_{/X'}\colon f^*$. This is both a Quillen pair for the projective covariant and injective covariant model structure; in particular $f_!$ preserves all weak equivalences. It now follows from \cite[Proposition 5.5]{boa18} that $(f_!, f^*)$ is a Quillen equivalence: indeed, the result from loc.\ cit.\ shows that the derived unit and counit are already equivalences pointwise in $\Del_{d}$. This immediately gives condition \ref{it:homotopy invariance}, while \ref{it:factorization} follows from Remark \ref{rem:factorization model cat case}.
\end{proof}
\begin{construction}
Finally, for each injectively fibrant $(d+1)$-category model $X$, let $\SegCoPShinj_d(X)\subseteq \mm{sPSh}(\Del\times \Del_{d})_{/X}$ denote the full subcategory of fibrant objects in the injective covariant model structure. This determines a (pseudo-)functor $\CompSegS_{d+1}^{\op}\rt \cat{FibCat}$, where $\CompSegS_{d+1}$ is a category of fibrant objects as in Construction \ref{con:d-cat model structure}. We will write $\SegCoPShinj_d\rt \CompSegS_{d+1}$ for its Grothendieck construction. 

The same proof as in Lemma \ref{lem:left fib proj} shows that $\SegCoPShinj_d\rt \CompSegS_{d+1}$ is a homotopy cartesian fibration between categories of fibrant objects. In fact, unraveling the definitions shows that $\SegCoPShinj_d$ can be identified with the category of fibrant objects for a model structure on $\Fun\big([1], \mm{sPSh}(\Del\times \Del_{d})\big)$: indeed, one can take a left Bousfield localization of the injective model structure whose fibrant objects are injective fibrations $Y\rt X$ that model Segal copresheaves in the sense of Definition \ref{def:lfib}. In particular, the localization of $\SegCoPShinj_d\rt \CompSegS_{d+1}$ is precisely the functor $\SegCoPSh_d\rt \cat{Cat}_{d+1}$ from Definition \ref{def:lfib}.
\end{construction}
\begin{theorem}[{Rectification of Segal copresheaves, cf.\ \cite[Theorem A]{boa18}}]\label{thm:rectification}
For each $d\geq 0$, there is a commuting diagram of categories of fibrant objects
$$\begin{tikzcd}
\EnrFun_d\arrow[r, "\rN"]\arrow[d] & \SegCoPShproj_d\arrow[d] \arrow[r, hookleftarrow] & \SegCoPShinj_d\arrow[d]\\
\Cat(\CompSegS_{d})\arrow[r, "\rN"] & \SegS_{d+1}\arrow[r, hookleftarrow] & \CompSegS_{d+1}
\end{tikzcd}$$
in which the vertical functors are homotopy cartesian fibrations and the horizontal functors induce equivalences upon localization.
\end{theorem}
\begin{proof}
The bottom horizontal functors induce equivalences on localizations by Proposition \ref{prop:comparing models}. By Proposition \ref{prop:localizing homotopy cartesian fibrations} and \cite[Corollary 2.4.4.4]{lur09}, the top horizontal functors induce equivalences on localizations as soon as the induced functors between fibers have this property. For the right square, this is immediate: the maps on fibers are induced by the Quillen equivalence between the injective and projective covariant model structure.

For the left square, one has to verify that for any enriched category $\scat{C}$, the functor $\rN\colon \Fun_{\enr}(\scat{C}, \CompSegS_{d})\rt \SegCoPShproj_d(\rN(\scat{C}))$ induces an equivalence on localizations. This functor sends an enriched functor $F\colon\scat{C}\rt \CompSegS_{d}$ to the obvious left fibration
$$\begin{tikzcd}
\rN(F)(n) = \coprod\limits_{c_0, \dots, c_n} \Map_{\scat{C}}(c_{n-1}, c_n)\times \dots\times \Map_{\scat{C}}(c_0, c_1)\times F(c_0)\arrow[r] & \rN(\scat{C})(n).
\end{tikzcd}$$
Note that this functor arises from a right Quillen functor $\rN\colon \Fun_{\enr}(\scat{C}, \mm{sPSh}(\Del_d))\rt \mm{sPSh}(\Del\times \Del_d)_{/\rN(\scat{C})}$ between the projective model structure on enriched functors and the projective covariant model structure on $\mm{sPSh}(\Del\times \Del_d)_{/\rN(\scat{C})}$. It follows from \cite[Theorem A]{boa18} that the derived unit and counit maps are pointwise weak equivalences in $\Del_d$. Consequently, $\rN$ is a right Quillen equivalence (even when $\mm{sPSh}(\Del_d)$ is endowed with the injective, rather than $d$-categorical model structure) and the result follows.
\end{proof}
\begin{corollary}[{cf.\ \cite[Theorem A]{boa18}}]\label{cor:left fibs are copresheaves}
For each $(d+1)$-category $\cat{C}$, there is an equivalence of $1$-categories
$$\begin{tikzcd}
\core_1\Fun_{d+1}(\cat{C}, \mb{Cat}_{d})\arrow[r, "\sim"] & \SegCoPSh_d(\cat{C}).
\end{tikzcd}$$
\end{corollary}
\begin{proof}
It follows from Theorem \ref{thm:rectification} that the functor $\SegCoPSh_d\rt \cat{Cat}_{d+1}$ is modeled by $\pi\colon \EnrFun_d\rt \cat{Cat}(\CompSegS_{d})$. Proposition \ref{prop:localizing homotopy cartesian fibrations} now implies that for any enriched category $\scat{C}$, the localization of the fiber $\pi^{-1}(\scat{C})$ is a model for the fiber $\SegCoPSh_d(\cat{C})$ over the corresponding $(d+1)$-category. The fiber $\pi^{-1}(\scat{C})$ is precisely the category of enriched functors $\scat{C}\rt \CompSegS_{d}$, whose localization models $\core_1\Fun_{d+1}(\cat{C}, \mb{Cat}_{d})$ by Proposition \ref{prop:mapping enriched categories}.
\end{proof}
\begin{remark}\label{rem:straightening low degrees}
For $d=0$, this gives a version of (un)straightening for left fibrations over $1$-categories. For $d>0$, viewing $\cat{C}$ as a presheaf on $\Del\times\Del_d$, the above result essentially corresponds to (un)straightening for left fibrations, applied pointwise in $\Del_d$.
\end{remark}
\begin{corollary}\label{cor:straightening low degrees}
Let $\cat{C}$ be a $1$-category, viewed as a $2$-category. Then there is an equivalence
$$\begin{tikzcd}
\Fun(\cat{C}, \Cat_1)\arrow[r, "\sim"] & \SegCoPSh_1(\cat{C})\arrow[r, "\Psi^\perp", "\sim"{swap}] & \Cocart_1(\cat{C})
\end{tikzcd}$$
to the category of cocartesian fibrations over $\cat{C}$ and maps preserving cocartesian morphisms.
\end{corollary}
\begin{proof}
We can combine Corollary \ref{cor:left fibs are copresheaves} and Theorem \ref{thm:main theorem}, using that there are obvious equivalences $\SegCoPSh_1(\cat{C})\simeq \cat{Fib}^{\mm{left, cart}}(\cat{C}\boxtimes [0])$ and $\Cocart_1(\cat{C})\simeq \cat{Fib}^{\mm{cocart, right}}(\cat{C}\boxtimes[0])$.
\end{proof}
\begin{example}\label{ex:mapping space functor}
Let $\cat{C}$ be a $(d+1)$-category and let $\mm{Tw}(\cat{C})\colon \Del^{\op}\rt \Cat_{d}$ be the functor given by $\mm{Tw}(\cat{C})(n) = \cat{C}\big([n]^{\op}\star [n]\big)$. Restriction along the inclusions $[n]^{\op}\hooklongrightarrow [n]^{\op}\star [n]\hookleftarrow [n]$ induces a natural map of $(d+1)$-fold simplicial spaces
$$\begin{tikzcd}
\mm{Tw}(\cat{C})\arrow[r] & \cat{C}^{\op}\times \cat{C}.
\end{tikzcd}$$
The Segal conditions on $\cat{C}$ imply that this is a Segal copresheaf. By Corollary \ref{cor:left fibs are copresheaves}, this corresponds to a $(d+1)$-functor that we will denote by
$$\begin{tikzcd}
\Map_{\cat{C}}\colon \cat{C}^{\op}\times \cat{C}\rt \mb{Cat}_{d}.
\end{tikzcd}$$
This functor is easily described using point set models: if $\scat{C}$ is an enriched category, then the above map of simplicial spaces is modeled by the object
$\mm{Tw}(\rN(\scat{C}))\rt \rN(\scat{C})^{\op}\times \rN(\scat{C})$ in $\modcat{coPSh}^{\Seg, \mm{proj}}(\rN(\scat{C}))$. Unraveling the definitions, this left fibration is exactly the image under $N\colon \Fun_{\enr}(\scat{C}^{\op}\times \scat{C}, \CompSegS_{d})\rt \cat{coPSh}^{\Seg, \mm{proj}}(\rN(\scat{C}))$ of the (strict) mapping space functor $\Map_{\scat{C}}(-, -)\colon \scat{C}^{\op}\times \scat{C}\rt \CompSegS_{d}$.
\end{example}
\begin{example}\label{ex:representable}
Consider a map in $\SegCoPSh_d$ corresponding to a commuting square
\begin{equation}\label{eq:representing}\begin{tikzcd}
\ast\arrow[d]\arrow[r, "x"] & X\arrow[d]\\
\ast\arrow[r, "c"] & \cat{C}.
\end{tikzcd}\end{equation}
We will say that $x$ is a \emph{representation} of $X$ if it defines a cocartesian arrow in $\SegCoPSh_d$. Note that if we model $\cat{C}$ by an enriched category $\scat{C}$, then $c^*\colon \Fun_{\enr}(\scat{C}, \CompSegS_d)\rt \CompSegS_d$ admits a left adjoint $c_!$ preserving trivial fibrations. It follows from Proposition \ref{prop:localizing homotopy cartesian fibrations}(4) that the map $\ast\rt \Map_{\scat{C}}(c, -)$ in $\EnrFun_d$ defines such a cocartesian arrow in $\SegCoPSh$. Consequently, \eqref{eq:representing} defines a representation of $X$ if and only if it induces an equivalence $\Map_{\cat{C}}(c, -)\simeq X$. 
\end{example}
We will use the language of Example \ref{ex:representable} to partially address the functoriality of Corollary \ref{cor:left fibs are copresheaves}:
\begin{notation}
For a regular uncountable cardinal $\kappa$, let us write $\SegCoPSh_d(\kappa)$ for the category of Segal copresheaves $X\rt \cat{C}$ whose fibers are essentially $\kappa$-small. The codomain projection $\SegCoPSh_d(\kappa)\rt \Cat_{d+1}$ is a cartesian fibration. Taking the subcategory with only cartesian morphisms, we obtain a right fibration $\pi\colon \SegCoPSh_d(\kappa)^{\cart}\rt \Cat_{d+1}$.

Note that there is a canonical element $u\in \SegCoPSh_d(\kappa)^{\cart}$, corresponding under the equivalence of Corollary \ref{cor:left fibs are copresheaves} to the identity functor $\hcat{Cat}_d(\kappa)\rt \hcat{Cat}_d(\kappa)$ on the small $(d+1)$-category of $\kappa$-small $d$-categories.
\end{notation}
\begin{proposition}\label{prop:space representability}
The element $u\colon \ast\rt \SegCoPSh_d(\kappa)$ defines a representation of the right fibration $\SegCoPSh_d(\kappa)^{\cart}\rt \cat{Cat}_{d+1}$. In other words, for every $(d+1)$-category $\cat{C}$, there is a \emph{$1$-functorial} equivalence between the space of $(d+1)$-functors $\cat{C}\rt \hcat{Cat}_d(\kappa)$ and the space of Segal copresheaves $X\rt \cat{C}$ with $\kappa$-small fibers.
\end{proposition}
\begin{remark}
For later purposes, we will mostly be interested in this result in the large setting, with $\kappa$ being the supremum of all small cardinals: in that case it provides a $1$-functorial equivalence between the spaces of $(d+1)$-functors $\cat{C}\rt \hcat{Cat}_d$ and Segal copresheaves $X\rt \cat{C}$ with small fibers over a large $(d+1)$-category $\cat{C}$.
\end{remark}

The proof uses the following well-known criterion for being a representation (cf.\ \cite[Lemma 1.31]{boa18}, \cite[Lemma 2.3.6]{kaz14}, \cite[Theorem 5.25]{ras21}):
\begin{lemma}\label{lem:representable left fib}
Let $\cat{C}$ be a $(d+1)$-category and consider an arrow in $\SegCoPSh_d$ of the form \eqref{eq:representing}. Then $x$ defines a representation of $p\colon X\rt \cat{C}$ if and only if it determines an initial object in each $1$-category $X(-, \vec{n}_{d})$.
\end{lemma}
\begin{proof}
Assuming $x$ defines a representation of $X\rt \cat{C}$, we can model $\cat{C}$ by an enriched category and identify $x$ with the canonical map $\mm{id}_c\colon \ast \rt N(\Map_{\scat{C}}(c, -))$. One easily verifies that $\mm{id}_c$ defines an initial object in each degree $\vec{n}_{d}$.

Conversely, suppose that $x$ defines an initial object in each $X(-, \vec{n}_d)$. Note that there is a fully faithful inclusion $\SegCoPSh_d\hooklongrightarrow \Fun(\Del_d, \SegCoPSh_0)$ which admits a left adjoint and preserves cartesian arrows. This implies that \eqref{eq:representing} defines a cocartesian arrow in $\SegCoPSh_d$ as soon as it does in $\Fun(\Del_d, \SegCoPSh_0)$. It therefore suffices to verify that each $\ast\rt X(-, \vec{n}_d)$ defines a cocartesian arrow in the category of Segal copresheaves over $1$-categories. This reduces everything to the case $d=0$, where we can consider the diagram
$$\begin{tikzcd}
\ast\arrow[r, "x"]\arrow[d] & X\arrow[d, "\mm{id}"]\arrow[r, "\mm{id}"] & X\arrow[d, "p"]\\
\ast\arrow[r, "x"] & X\arrow[r, "p"] & \cat{C}.
\end{tikzcd}$$
One easily verifies that the right square defines a cocartesian arrow in the category $\SegCoPSh_0$ of left fibrations. On the other hand, since $x\in X$ is an initial object, the left square represents the terminal copresheaf on $X$, i.e.\ it also defines a cocartesian arrow in $\SegCoPSh_0$. The composite is then a cocartesian arrow as well, as desired.
\end{proof}
\begin{proof}[Proof of Proposition \ref{prop:space representability}]
By (the opposite of) Lemma \ref{lem:representable left fib}, it suffices to show that $u$ defines a terminal object in $\SegCoPSh_d(\kappa)^{\cart}$. 

Let us write $\EnrFun_d(\kappa)\subseteq \EnrFun_d$ for the subcategory of tuples $(\scat{C}, F)$ consisting of a small $\CompSegS_d$-enriched category and an enriched functor $F\colon \scat{C}\rt \CompSegS_d$ whose values are homotopy equivalent to $\kappa$-small objects. Lemma \ref{lem:enrfun is ho cart fib} and Theorem \ref{thm:rectification} imply that $\EnrFun_d(\kappa)\rt \Cat(\CompSegS_d)$ is a homotopically cartesian fibration between categories of fibrant objects, whose localization models $\SegCoPSh_d(\kappa)\rt \Cat_{d+1}$. The object $u\in \SegCoPSh_d(\kappa)$ simply corresponds to the tuple $u=(\CompSegS_d(\kappa), \mm{id})$ consisting of the enriched category of $\kappa$-small $d$-category models, together with the identity (or more precisely, with the fully faithful inclusion $\CompSegS_d(\kappa)\hooklongrightarrow \CompSegS_d$).

Proposition \ref{prop:localizing homotopy cartesian fibrations} now implies that $\SegCoPSh_d(\kappa)^{\cart}$  arises as the localization of the $(1, 1)$-category $\cat{H}=\EnrFun_d(\kappa)^{\mm{hocart}}$ of such tuples $(\scat{C}, F)$, with homotopy cartesian maps between them (see Appendix \ref{sec:fib obj} for more details). We have to show that $\Map_{\cat{H}[W^{-1}]}(-, u)$ is the terminal (very large) presheaf. To this end, let $\psi\colon \cat{H}\rt \cat{H}[W^{-1}]$ denote the localization and consider the adjoint pair between categories of very large presheaves
$$\begin{tikzcd}
\psi_!\colon \Fun\big(\cat{H}, \sS^{\text{very large}}\big)\arrow[r, yshift=1ex]\arrow[r, hookleftarrow, yshift=-1ex] & \Fun\big(\cat{H}[W^{-1}], \sS^{\text{very large}}\big)\colon \psi^*.
\end{tikzcd}$$
Here $\psi^*$ is fully faithful with essential image consisting of presheaves sending $W$ to equivalences. We have to show that $\psi_!(u)\simeq \ast$ (where we omit the Yoneda embedding from the notation).

To see this, consider for each small cardinal $\lambda\geq \kappa$ the full $\CompSegS_{d}$-enriched subcategory $\scat{U}_{\lambda, n}\subseteq \Fun([n], \CompSegS_d)$ spanned by sequences of trivial fibrations $A_0\twoheadrightarrow \dots \twoheadrightarrow A_n$ where each $A_i$ is $\lambda$-small and homotopy equivalent to a $\kappa$-small object. Each $\scat{U}_{\lambda, n}$ comes with a canonical enriched functor $p_0\colon \scat{U}_{\lambda, n}\rt \CompSegS_{d}$ sending $(A_0\rt\dots\rt A_n)$ to $A_0$. Let us write $u_{\lambda, n}=(\scat{U}_{\lambda, n}, p_0)$ for the corresponding object in $\EnrFun_d(\kappa)$ and note that $u_{\kappa, 0}$ coincides with $(\CompSegS_{d}(\kappa), \mm{id})$. For each $\lambda\leq \lambda'$, the evident map $u_{\lambda, n}\rt u_{\lambda', n}$ is a weak equivalence and each $u_{\lambda, \bullet}\colon \Del^{\op}\rt \EnrFun_d(\kappa)$ is a simplicial object where each structure map is a weak equivalence; in particular, this defines a simplicial diagram in $\cat{H}$. Taking the (large) filtered colimit over all small cardinals $\lambda$, we then obtain a map of presheaves
\begin{equation}\label{diag:local replacement}\begin{tikzcd}
\Map_{\cat{H}}(-, u)\arrow[r] & \colim_{\lambda}\colim_{\Del^{\op}}\Map_{\cat{H}}(-,  u_{\lambda, \bullet})
\end{tikzcd}\end{equation}
which induces an equivalence upon applying $\psi_!$. It therefore remains to verify that the target is the terminal presheaf (so that its image under $\psi_!$ is terminal as well).

Now fix an object $(\scat{C}, F)\in \cat{H}$ and for each $\lambda$ consider the space 
$$
X_\lambda=\colim_{\Del^{\op}}\Map_{\cat{H}}\big((\scat{C}, F), u_{\lambda, \bullet}\big).
$$
It suffices to verify that this space is contractible for large $\lambda$; taking the (filtered) colimit over $\lambda$ then shows that the target of \eqref{diag:local replacement} is terminal.

Unraveling the definitions, $X_\lambda$ is the classifying space of the (1, 1)-category whose objects are weak equivalences $F\rt F'$ of functors $\scat{C}\rt \CompSegS_d$, where $F'$ takes $\lambda$-small values, and where a morphism is a (pointwise) trivial fibration $F'\twoheadrightarrow F''$ of functors under $F$. 
For large enough $\lambda$ (in particular, so that $F$ takes $\lambda$-small values), the under-category $\cat{M}_\lambda= F/\Fun_{\enr}\big(\scat{C}, \CompSegS_d(\lambda)\big)$ has the structure of a category of fibrant objects in the sense of Brown. The space $X_\lambda$ is then a path component in the classifying space of the wide subcategory $\mm{wf}\cat{M}_\lambda\subseteq \cat{M}_\lambda$ spanned by the trivial fibrations. By \cite[Lemma 14 and 15]{lan17}, the classifying space of $\mm{wf}\cat{M}_\lambda$ is equivalent to that of the wide subcategory $\mm{w}\cat{M}_\lambda\subseteq \cat{M}_\lambda$ spanned by the weak equivalences. Restricting to the correct path component, this implies that for large $\lambda$, $X_\lambda$ is equivalent to the classifying space of the (1, 1)-category whose objects are weak equivalences $F\rt F'$ of enriched functors with $\lambda$-small values, with weak equivalences $F'\rt F''$ as morphisms. This latter category has an initial object ($F$ itself) and is hence contractible.
\end{proof}
\begin{variant}\label{var:maps of presheaves}
Associated to the cartesian fibration $\SegCoPSh_d(\kappa)\rt \cat{Cat}_{d+1}$ is a cartesian fibration of $1$-categories
$$\begin{tikzcd}
\cat{X}=\Fun\big([1], \SegCoPSh_d(\kappa)\big)\times_{\Fun([1], \cat{Cat}_{d+1})} \cat{Cat}_{d+1}\arrow[r] & \cat{Cat}_{d+1}.
\end{tikzcd}$$
The same argument as Theorem \ref{thm:rectification} shows that $p\colon \cat{X}\rt \cat{Cat}_{d+1}$ can be obtained as the localization of a homotopy cartesian fibration $\modcat{X}\rt \Cat(\CompSegS_{d})$ between categories of fibrant objects; here $\modcat{X}$ is the Grothendieck construction of the functor sending a $\CompSegS_{d}$-enriched category $\hcat{C}$ to the full subcategory of $\Fun_{\enr}([1]\times \hcat{C}, \CompSegS_{d})$ on the homotopically $\kappa$-small objects. Repeating the argument from Proposition \ref{prop:space representability} then shows that the underlying right fibration $\cat{X}^{\cart}\rt \cat{Cat}_{d+1}$ is representable by $\Fun_{d+1}([1], \hcat{Cat}_{d}(\kappa))$.
\end{variant}

\subsection{More on representability}\label{sec:corep}
Let $\cat{C}$ be a $(d+1)$-category and consider the $(d+1)$-functor
\begin{equation}\label{diag:yoneda}\begin{tikzcd}
h\colon \cat{C}^{\op}\arrow[r] & \Fun_{d+1}\big(\cat{C}, \hcat{Cat}_{d}\big)
\end{tikzcd}\end{equation}
adjoint to the $(d+1)$-functor $\Map_{\cat{C}}\colon \cat{C}^{\op}\times \cat{C}\rt \hcat{Cat}_{d}$ defined in Example \ref{ex:mapping space functor}. If $\cat{C}$ arises from an enriched category $\scat{C}$, then Proposition \ref{prop:mapping enriched categories} provides an explicit model for $h$: it arises from the usual (strictly) enriched Yoneda embedding
$$\begin{tikzcd}
h\colon \scat{C}^{\op}\arrow[r] & \Fun\big(\scat{C}, \CompSegS_{d}\big)^{\circ}; & c\arrow[r, mapsto] & \Map_{\scat{C}}(c, -).
\end{tikzcd}$$
Since this is (strictly) fully faithful, the functor \eqref{diag:yoneda} is fully faithful as well. An $\infty$-categorical proof of this can also be found in \cite{hin20}.
\begin{definition}
Let $\cat{C}$ be a $(d+1)$-category. A functor $F\colon \cat{C}\rt \mb{Cat}_{d}$ is said to be \emph{representable} if it is contained in the essential image of \eqref{diag:yoneda}.
\end{definition}
\begin{remark}\label{rem:representable Segal copresheaf}
Let $F\colon \cat{C}\rt \mb{Cat}_d$ be a functor and let $p\colon X\rt \cat{C}$ be the corresponding Segal copresheaf. From the above presentation in terms of enriched categories, one sees that $F$ is representable if and only if $p$ admits a representation in the sense of Example \ref{ex:representable}. By Lemma \ref{lem:representable left fib}, this is equivalent to the $1$-category $X(-, \vec{0}_d)$ admitting an initial object and each $X(-, \vec{0}_d)\rt X(-, \vec{n}_{d})$ preserving the initial object. The representing object is the image of this initial object in $\cat{C}$.
\end{remark}
\begin{example}\label{ex:Yoneda lemma}
For each object $c$ of a $(d+1)$-category $\cat{C}$, consider the $(d+1)$-functor 
$$\begin{tikzcd}
\mm{ev}_c\colon \Fun_{d+1}(\cat{C}, \hcat{Cat}_d)\arrow[r] & \Fun_{d+1}(\ast, \hcat{Cat}_d)=\hcat{Cat}_d
\end{tikzcd}$$
given by restriction along $c\colon \ast\rt \cat{C}$. Then the functor $\mm{ev}_c$ is corepresentable by an object of $\Fun_{d+1}(\cat{C}, \hcat{Cat}_d)$, which in turn is given by the corepresentable functor $h_c$. Indeed, this is a special case of the enriched Yoneda lemma, proven in the $\infty$-categorical setting in \cite[6.2.7]{hin20}. In terms of strict models, one can also deduce this from the strict version of the enriched Yoneda lemma: choosing a strictly enriched model $\scat{C}$ for $\cat{C}$, one can model $\mm{ev}_c$ by the strictly enriched functor $\mm{ev}_c\colon \Fun_\mm{enr}(\scat{C}, \scat{Cat}_d)^{\circ}\rt \CompSegS_d$ evaluating a strictly enriched functor at the object $c\in \scat{C}$. By the classical enriched Yoneda lemma, this functor is corepresented by $h_c=\Map_{\scat{C}}(c, -)$.
\end{example}
Recall that $\mb{Cat}_{d+1}$ has a homotopy $(2, 2)$-category, obtained by taking the homotopy $(1, 1)$-categories of all mapping $(d+1)$-categories $\Fun_{d+1}(\cat{C}, \cat{D})$. An \emph{adjunction} between two $(d+1)$-categories is simply an adjunction in this homotopy $(2, 2)$-category \cite{rie16}. We then have the usual criterion for a $(d+1)$-functor $g\colon \cat{C}\rt \cat{D}$ admitting a left adjoint:
\begin{proposition}\label{prop:adj via rep}
Let $g\colon \cat{C}\rt \cat{D}$ be a map of $(d+1)$-categories. Then $g$ admits a left adjoint if and only if for each object $d\in\cat{D}$, the functor $\Map_{\cat{D}}(d, g(-))\colon \cat{C}\rt \mb{Cat}_{d}$ is representable.
\end{proposition}
For example, the map $\cat{C}\rt \ast$ admits a left adjoint if and only if $\cat{C}$ admits an initial object, i.e.\ an object $\emptyset$ such that $\Map_{\cat{C}}(\emptyset, c)\simeq \ast$ for all $c\in \cat{C}$.
\begin{proof}
As a preliminary observation, let us recall that any $2$-functor preserves adjunctions. We are going to apply this in two situations:
\begin{enumerate}[label=(\alph*)]
\item\label{it:adj to psh} 
to the $2$-functor $\ho_{(2, 2)}(\mb{Cat}_d)^{\op}\rt \Cat_{(1, 1)}$ sending $\cat{C}\longmapsto \ho_{(1, 1)}\Fun_d(\cat{C}, \mb{Cat}_{d-1})$, and (2-)morphisms to the corresponding restriction functors and natural transformations.

\item\label{it:adj from strict}
to the localization functor $\cat{Cat}\big(\mm{sPSh}(\Del_{d-1})\big)\rt \ho_{(2, 2)}(\mb{Cat}_d)$ sending each (strictly) \hbox{$\mm{sPSh}(\Del_{d-1})$-enriched category} to the corresponding $d$-category (viewed as an object in the homotopy $(2, 2)$-category of $d$-categories). Note that this is indeed a $2$-functor, since it is naturally tensored over $\Cat_{(1, 1)}$ via the cartesian product.
\end{enumerate}
Let us now start by assuming that $g$ admits a left adjoint $f$. Applying the 2-functor \ref{it:adj to psh} to this adjoint pair, we obtain an adjoint pair
$$\begin{tikzcd}
g^*\colon \ho_{(1, 1)}\Fun_d\big(\cat{D}, \mb{Cat}_{d-1}\big)\arrow[r, yshift=1ex] & \ho_{(1, 1)}\Fun_d\big(\cat{C}, \mb{Cat}_{d-1}\big)\colon f^*.\arrow[l, yshift=-1ex]
\end{tikzcd}$$
We can model $f\colon \cat{D}\rt \cat{C}$ by a functor of $\CompSegS_{d-1}$-enriched categories $f\colon \scat{D}\rt \scat{C}$. The functor $f^*$ then arises from the right Quillen functor 
$$
f^*\colon \Fun_{\enr}\big(\scat{C}, \mm{sPSh}(\Del_d)\big)\rt \Fun_{\enr}\big(\scat{D}, \mm{sPSh}(\Del_d)\big).
$$
In particular, the left adjoint $g^*$ (on homotopy $(1, 1)$-categories) is naturally isomorphic to the left derived functor of $f_!\colon \Fun_{\enr}\big(\scat{D}, \mm{sPSh}(\Del_d)\big)\rt \Fun_{\enr}\big(\scat{C}, \mm{sPSh}(\Del_d)\big)$. Since the left derived functor of $f_!$ preserves representable copresheaves, it follows that $g^*$ preserves representable copresheaves as well; this is precisely the assertion that $\Map_{\cat{D}}(d, g(-))$ is representable for all $d\in \cat{D}$.

For the converse, suppose that each $\Map_{\cat{D}}(d, g(-))\colon \cat{C}\rt \mb{Cat}_{d-1}$ is representable. To prove that $g$ admits a left adjoint in $\ho_{(2, 2)}(\mb{Cat}_d)$, it will suffice to find a map $\scat{C}^{\op}\rt \scat{D}^{\op}$ in $\cat{Cat}\big(\mm{sPSh}(\Del_{d-1})\big)$ which admits a (strict) right adjoint and whose image under the functor \ref{it:adj from strict} is equivalent to the \emph{opposite} of $g\colon \cat{C}\rt \cat{D}$. 

To provide the desired model for $g$, let us choose any cofibration $\gamma\colon \scat{C}_0\rt \scat{D}_0$ between $\CompSegS_{d-1}$-enriched categories that models $g\colon \cat{C}\rt \cat{D}$. This induces a Quillen pair of (in particular) simplicial model categories
$$\begin{tikzcd}
\gamma_!\colon \Fun_{\enr}\big(\scat{C}_0, \mm{sPSh}(\Del_{d-1})\big)\arrow[r, yshift=1ex] & \Fun_{\enr}\big(\scat{D}_0, \mm{sPSh}(\Del_{d-1})\big)\colon \gamma^*\arrow[l, yshift=-1ex]
\end{tikzcd}$$
given by restriction and left Kan extension. Let us now define
$$
\scat{C}^{\op}\subseteq \Fun_{\enr}\big(\scat{C}_0, \mm{sPSh}(\Del_{d-1})\big) \qquad \qquad\text{and}\qquad\qquad \scat{D}^{\op}\subseteq \Fun_{\enr}\big(\scat{D}_0, \mm{sPSh}(\Del_{d-1})\big)
$$
to be the full (enriched) subcategories of all enriched functors that admit a simplicial homotopy equivalence to a corepresentable functor. Since $\gamma_!$ preserves simplicial homotopy equivalences and corepresentable functors, we obtain a commuting square of categories enriched over $\mm{sPSh}(\Del_{d-1})$
$$\begin{tikzcd}
\scat{C}_0^{\op}\arrow[r, "\gamma^{\op}"]\arrow[d, hook, "h"{swap}] & \scat{D}_0^{\op}\arrow[d, hook, "h"]\\
\scat{C}^{\op}\arrow[r, "\gamma_!"] & \scat{D}^{\op}
\end{tikzcd}$$
where the vertical functors $h$ are the Yoneda embeddings. By construction, both of these vertical functors are Dwyer--Kan equivalences of enriched categories, so that $\gamma_!\colon \scat{C}^{\op}\rt \scat{D}^{\op}$ also maps to $g^{\op}$ under the localization functor \ref{it:adj from strict}. Let us point that this would not necessarily be the case if one replaced ``simplicial homotopy equivalence'' by ``weak equivalence'' in the definition of $\scat{C}^{\op}$ (unless one imposes further fibrancy and cofibrancy conditions, which need not be preserved by both $\gamma_!$ and $\gamma^*$).

We claim that the restriction functor $\gamma^*$ also restricts to a functor $\gamma^*\colon \scat{D}^{\op}\rt \scat{C}^{\op}$. Indeed, since $\gamma^*$ preserves simplicial homotopy equivalences, it suffices to verify that $\gamma^*$ sends each corepresentable functor to an object in $\scat{C}^{\op}$. In other words, we have to show that for each $d\in \scat{D}_0$, the enriched functor $\Map_{\scat{D}_0}(d, \gamma(-))$ admits a simplicial homotopy equivalence to a corepresentable object. This enriched functor is a model for $\Map_{\cat{D}}(d, g(-))$ and is hence weakly equivalent to a corepresentable object. Since $\Map_{\scat{D}_0}(d, \gamma(-))$ is projectively fibrant-cofibrant \cite[Proposition A.3.3.9]{lur09}, it then admits a simplicial homotopy equivalence to a corepresentable as well. We conclude that $\gamma_!\colon \scat{C}^{\op}\rt \scat{D}^{\op}$ admits a strict right adjoint $\gamma^*$, so that its image $g^{\op}$ in $\ho_{(2, 2)}(\mb{Cat}_d)$ admits a right adjoint as well, as desired.
\end{proof}
\begin{corollary}\label{cor:reflective localization stable under pullback}
Let $g\colon \cat{C}\rt \cat{D}$ be a map of $d$-categories which admits a fully faithful left adjoint. If $q\colon \cat{D}'\rt \cat{D}$ is any functor, then the base change $g'\colon \cat{C}\times_{\cat{D}} \cat{D}'\rt \cat{D}'$ admits a fully faithful left adjoint as well.
\end{corollary}
\begin{proof}
For $d'\in \cat{D}'$, let $c\in \cat{C}$ be a representing object for $\Map_{\cat{D}}(q(d'), g(-))$. Then $(c, d')\in \cat{C}\times_{\cat{D}} \cat{D}'$ provides the desired representing object for $\Map_{\cat{D}'}(d', g'(-))$.
\end{proof}
\begin{remark}\label{rem:limits of adjointable}
Using a similar argument, one can show the following two assertions:
\begin{enumerate}
\item Consider a pullback diagram $\Lambda^2[2]\rt \Fun([1], \cat{Cat}_d)$ sending each $i$ to a functor $g_i\colon \cat{C}_i\rt \cat{D}_i$ admitting a fully faithful left adjoint. If the functors $\cat{C}_0\rt \cat{C}_2\lt \cat{C}_1$ preserve the essential images of these left adjoints, then the pullback $g\colon \cat{C}_0\times_{\cat{C}_2}\cat{C}_1\rt \cat{D}_0\times_{\cat{D}_2}\cat{D}_1$ admits a fully faithful left adjoint as well.

\item Consider a tower $\mathbb{N}^{\op}\rt \Fun([1], \cat{Cat}_d)$ sending each $i$ to a functor $g_n\colon \cat{C}_n\rt \cat{D}_n$ admitting a fully faithful left adjoint. If the functors $\cat{C}_{n+1}\rt \cat{C}_n$ preserve the essential images of these left adjoints, then the map $\lim_n \cat{C}_n\rt \lim_n\cat{D}_n$ admits a fully faithful left adjoint.
\end{enumerate}
Indeed, in each of these cases consider an object $(d_i)\in \lim \cat{D}_i$ in the limit. For each $i$, there exists an inverse image $c_i\in g_i^{-1}(d_i)$ such that $\Map_{\cat{C}_i}(c_i, -)\rt \Map_{\cat{D}_i}(d_i, g_i(-))$ is an equivalence. The fact that the functors $\cat{C}_i\rt \cat{C}_j$ preserve the essential images of the fully faithful left adjoints implies that we can (inductively) lift this tuple of objects $c_i$ to an object $(c_i)\in \lim\cat{C}_i$ in the limit. This object maps to $(d_i)\in \lim \cat{D}_i$ and has the desired universal property. 
\end{remark}

\subsection{Tensoring}
Let us conclude by recalling that questions about higher categories can often be reduced to questions about their underlying $1$-categories when their enriched structure arises from a (co)tensoring:
\begin{definition}\label{def:tensored cat}
Let $\cat{C}$ be a $(d+1)$-category and consider the $(d+1)$-functor
$$\begin{tikzcd}
\mb{Cat}^{\op}_{d}\times \cat{C}^{\op}\arrow[r] & \Fun_{d+1}(\cat{C}, \mb{Cat}_{d}); \quad (K, c)\arrow[r, mapsto] & \Fun_{d}\big(K, \Map_{\cat{C}}(c, -)\big).
\end{tikzcd}$$
Then $\cat{C}$ is said to be \emph{tensored} over $\mb{Cat}_{d}$ if this takes values in representable functors. In this case, we will write $\otimes\colon \mb{Cat}_{d}\times \cat{C}\rt \cat{C}$ for the induced functor.
\end{definition}
\begin{remark}
In other words, $K\otimes c$ is the unique object of $\cat{C}$ with a natural equivalence 
$$
\Map_{\cat{C}}(K\otimes c, -)\simeq \Fun_{d}(K, \Map_{\cat{C}}(c, -)).
$$
This universal property implies that there are equivalences $K\otimes (L\otimes c)\simeq (K\times L)\otimes c$. Furthermore, note that the identity on $K\otimes c$ induces a natural map of $d$-categories $K\rt \Map_{\cat{C}}(c, K\otimes c)$.
\end{remark}
\begin{definition}\label{def:tensored functor}
Let $f\colon \cat{C}\rt \cat{D}$ be a map between two $(d+1)$-categories which are tensored over $\hcat{Cat}_{d}$. For each $c\in \cat{C}$ and $K\in \hcat{Cat}_{d}$, the composite map 
$$
K\rt \Map_{\cat{C}}(c, K\otimes c)\rt \Map_{\cat{D}}(f(c), f(K\otimes c))$$
determines a natural map $\mu\colon K\otimes f(c)\rt f(K\otimes c)$. We will say that $f$ is \emph{tensored} over $\Cat_{d}$ if all of these maps $\mu$ are equivalences.
\end{definition} 
\begin{remark}
Likewise, a $(d+1)$-category $\cat{C}$ is \emph{cotensored} over $\Cat_{d}$ if $\cat{C}^{\op}$ is tensored over $\Cat_{d}$; in this case we will write $c^K$ for the cotensoring of $c\in \cat{C}$ with $K\in \Cat_{d}$. A functor $f\colon \cat{C}\rt \cat{D}$ is cotensored if all natural maps $\mu\colon f(c^K)\rt f(c)^K$ are equivalences.
\end{remark}
\begin{example}\label{ex:tensoring copresheaves}
The $(d+1)$-category $\hcat{Cat}_{d}$ is evidently tensored and cotensored over $\hcat{Cat}_{d}$ itself. Likewise, every functor $(d+1)$-category $\Fun_{d+1}(\cat{C}, \hcat{Cat}_{d})$ is tensored and cotensored over $\hcat{Cat}_{d}$, with (co)tensoring computed pointwise. This follows conceptually from the description of enriched functor $\infty$-categories by Hinich \cite[Section 6]{hin20}. Less elegantly, one can see this by working with the strict model $\scat{M}^\circ=\Fun_{\enr}\big(\scat{C}, \CompSegS_{d}\big)^\circ$ provided Proposition \ref{prop:mapping enriched categories}. Indeed, for each $K\in \CompSegS_d$ and a fibrant-cofibrant enriched functor $F$, there is a natural isomorphism
$$
\Map_{\CompSegS_d}\big(K, \Map_{\scat{M}^{\circ}}(F, -)\big)\cong \Map_{\scat{M}^{\circ}}\big(F\times \mm{cst}(K), -\big).
$$
so that the pointwise product $F\times \mm{cst}(K)$ with the constant functor determines a (strict) tensoring on $\scat{M}^{\circ}$. Dually, let $Q(F^K)\in \scat{M}^{\circ}$ be a projectively cofibrant replacement of the pointwise fibrant enriched functor $F^K(-)=\Map_{\CompSegS_d}(K, F(-))$. Then there is a natural equivalence
$$\begin{tikzcd}
 \Map_{\scat{M}^{\circ}}\big(-, Q(F^{K})\big) \arrow[r, "\sim"] & \Map_{\CompSegS_d}\big(K, \Map_{\scat{M}^{\circ}}(-, F)\big)
\end{tikzcd}$$
which exhibits $Q(F^K)$ as a cotensoring.
\end{example}
\begin{lemma}\label{lem:tensored equivalence}
Let $f\colon \cat{C}\rt \cat{D}$ be a tensored functor between $(d+1)$-categories that are tensored over $\hcat{Cat}_{d}$. If the underlying functor of $1$-categories is an equivalence, then $f$ is an equivalence of $(d+1)$-categories as well.
\end{lemma}
\begin{proof}
Note that for any $K\in \Cat_{d}$, the map between mapping spaces in $\Cat_{d}$
$$\begin{tikzcd}
\Map_{\Cat_d}(K, \Map_{\cat{C}}(c, d)))\arrow[r] & \Map_{\Cat_d}(K, \Map_{\cat{D}}(f(c), f(d)))
\end{tikzcd}$$
can be identified with $\Map_{\core_1(\cat{C})}(K\otimes c, d)\rt \Map_{\core_1(\cat{D})}(K\otimes f(c), d)$. The latter is an equivalence since $f$ is fully faithful on the underlying $1$-categories, so that $f$ is a fully faithful functor between $(d+1)$-categories as well. It is essentially surjective since it is essentially surjective on the underlying $1$-categories.
\end{proof}
\begin{proposition}\label{prop:representable from cotensored}
Let $\cat{C}$ be a $(d+1)$-category that is tensored over $\mb{Cat}_{d}$ and let $F\colon \cat{C}^{\op}\rt \mb{Cat}_d$ be a $(d+1)$-functor. If $F$ is cotensored and the underlying $1$-functor 
$$\begin{tikzcd}[column sep=3pc]
\core_1(\cat{C})^{\op}\arrow[r] & \Cat_d\arrow[r, "\core_0"] & \sS
\end{tikzcd}$$ 
is representable, then $F$ is representable.
\end{proposition}
\begin{proof}
Let $x_u\in \core_0(F(c_u))$ be the element corresponding (by the Yoneda lemma) to the natural equivalence of $1$-functors $\Map_{\cat{C}}(-, c_u)\rto{\sim} \core_0(F)$. Viewing $x_u$ as an element in $F(c_u)$, the enriched Yoneda lemma provides a natural transformation of $(d+1)$-functors $\Map_{\cat{C}}(-, c_u)\rt F$. To see that this is an equivalence, note that for any $K\in \mb{Cat}_{d}$ and $d\in \cat{C}$, there is a commuting square
$$\begin{tikzcd}[column sep=4.7pc]
\Map_{\cat{C}}(K\otimes d, c_u)\arrow[d, "\sim"{swap}]\arrow[r, "x_u"]  & F(K\otimes d)\arrow[d, "\mu", "\sim"{swap}]\\
\Fun_{d}(K, \Map_{\cat{C}}(d, c_u))\arrow[r, "{\Fun_{d}(K, x_u)}"]  & \Fun_{d}(K, F(d))
\end{tikzcd}$$
The vertical functors are equivalences because $F$ is cotensored. Since the top map induces an equivalence on underlying spaces, it follows that the bottom map induces an equivalence on the underlying spaces for all $K\in \mb{Cat}_{d}$. By the Yoneda lemma, this implies that the map $x_u\colon \Map_{\cat{C}}(d, c_u)\rt F(d)$ is an equivalence.
\end{proof}
\begin{remark}\label{rem:less than full tensoring}
Lemma \ref{lem:tensored equivalence} and Proposition \ref{prop:representable from cotensored} do not require a full tensoring over $\hcat{Cat}_d$. For instance, being (co)tensored over any full subcategory containing $\Theta_d$ (and closed under cartesian products) suffices as well.
\end{remark}

\section{Cocartesian fibrations of higher categories}\label{sec:cocart fib}
In this section we discuss the higher categorical analogues of cocartesian fibrations whose fibers are given by $d$-categories. We will refer to such fibrations as \emph{$d$-cocartesian fibrations}. Most importantly, we show (Section \ref{sec:cat of cocart fibs}) that $d$-cocartesian fibrations between \hbox{$(d+1)$}-categories form themselves the domain of a $(d+1)$-cartesian fibration between $(d+2)$-categories $\hcat{Cocart}_{d}\rt \hcat{Cat}_{d+1}$.

\subsection{Pre-cocartesian fibrations}\label{sec:pre-cocart}
Let us start with the following preliminary definition, which reduces to that of a cocartesian fibration when $d=0$ \cite{aya17, rie21}:
\begin{definition}\label{def:pre-cocart}
Let $p\colon \cat{D}\rt \cat{C}$ be a map of $(d+1)$-categories. Then $p$ is said to be a \emph{pre-cocartesian fibration} if the map of $(d+1)$-categories
\begin{equation}\label{diag:precocart condition}\begin{tikzcd}
\Fun_{d+1}([1], \cat{D})\arrow[r] & \Fun_{d+1}(\{0\}, \cat{D})\times_{\Fun_{d+1}(\{0\}, \cat{C})} \Fun_{d+1}([1], \cat{C}).
\end{tikzcd}\end{equation}
is a right adjoint and the unit of the adjunction is an equivalence. A $1$-morphism in $\cat{D}$ is called \emph{(p-)cocartesian} if it is contained in the essential image of the left adjoint. 
\end{definition}
\begin{remark}\label{rem:pre-cart}
Dually, $p$ is said to be a \emph{pre-cartesian fibration} if the map  
$$\begin{tikzcd}
\Fun_{d+1}([1], \cat{D})\arrow[r] & \Fun_{d+1}(\{1\}, \cat{D})\times_{\Fun_{d+1}(\{1\}, \cat{C})} \Fun_{d+1}([1], \cat{C})
\end{tikzcd}$$
is a left adjoint and the counit of the adjunction is an equivalence. A $1$-morphism in $\cat{D}$ is called \emph{(p-)cartesian} if it is contained in the essential image of the right adjoint.
\end{remark}
This notion of pre-cocartesian fibration coincides with the 2-categorical notion of an abstract cocartesian fibration from Street \cite{str80} and Riehl--Verity \cite{rie21} (interpreted in the $2$-category of $(d+1)$-categories). In particular, we have the following properties:
\begin{lemma}\label{lem:pre-cocart stable}
The class of pre-cocartesian fibrations has the following stability properties:
\begin{enumerate}
\item If $q\colon \cat{E}\rt \cat{D}$ and $p\colon \cat{D}\rt \cat{C}$ are pre-cocartesian fibrations, then $pq\colon \cat{E}\rt \cat{C}$ is a pre-cocartesian fibration as well.

\item If $p\colon \cat{D}\rt \cat{C}$ is a pre-cocartesian fibration and $\cat{C}'\rt \cat{C}$ is any map, then the base change $\cat{D}'=\cat{D}\times_{\cat{C}}\cat{C}'\rt \cat{C}'$ is a pre-cocartesian fibration.

\item If $p\colon \cat{D}\rt \cat{C}$ is a pre-cocartesian fibration, then $\Fun_{d+1}(\cat{B}, \cat{D})\rt \Fun_{d+1}(\cat{B}, \cat{C})$ is a pre-cocartesian fibration for any $(d+1)$-category $\cat{B}$. Furthermore, for any $\cat{B}'\rt \cat{B}$, the restriction functor $\Fun_{d+1}(\cat{B}, \cat{D})\rt \Fun_{d+1}(\cat{B}', \cat{D})$ preserves cocartesian 1-morphisms.

\item Consider the subcategory of $\Fun([1], \cat{Cat}_{d+1})$ whose objects are pre-cocartesian fibrations and morphisms are squares
$$\begin{tikzcd}
\cat{D}\arrow[d, "p"{swap}]\arrow[r, "f"] & \cat{D}'\arrow[d, "p'"]\\
\cat{C}\arrow[r] & \cat{C}'
\end{tikzcd}$$
such that $f$ preserves cocartesian $1$-morphisms. This subcategory is closed under limits.
\end{enumerate}
\end{lemma}
\begin{proof}
(1) is easily verified. For (2), note that the map \eqref{diag:precocart condition} for $\cat{D}'\rt \cat{C}'$ is the base change of the same map for $\cat{D}\rt \cat{C}$. The result then follows from Corollary \ref{cor:reflective localization stable under pullback}.

For (3), note that $\Fun_{d+1}([1], \Fun_{d+1}(\cat{B}, \cat{D}))\cong \Fun_{d+1}(\cat{B}, \Fun_{d+1}([1], \cat{D}))$. Using this, the result follows from the fact that $\Fun_{d+1}(\cat{B}, -)$ preserves limits and adjunctions (with fully faithful left adjoint).

Finally, for (4) one can directly see that the subcategory is closed under products. It then suffices to verify that it is closed under pullbacks and limits of (countable) towers. This comes down to verifying that \eqref{diag:precocart condition} admitting a fully faithful left adjoint is closed under such limits, which follows from Remark \ref{rem:limits of adjointable}. Alternatively, one can also deduce this from Corollary \ref{cor:classical def of pre-cocart}, using an inductive argument to find enough cocartesian lifts in the limits.
\end{proof}
\begin{lemma}\label{lem:cocartesian morphisms}
Let $p\colon \cat{D}\rt \cat{C}$ be a map of $(d+1)$-categories and $\alpha\colon d_0\rt d_1$ a $1$-morphism in $\cat{D}$. Then the following are equivalent:
\begin{enumerate}
\item For every $d_2\in \cat{D}$, the square of mapping $d$-categories
\begin{equation}\label{diag:standard cocartesian condition}\begin{tikzcd}
\Map_{\cat{D}}(d_1, d_2)\arrow[r, "\alpha^*"]\arrow[d, "p"{swap}] & \Map_{\cat{D}}(d_0, d_2)\arrow[d, "p"]\\
\Map_{\cat{C}}(pd_1, pd_2)\arrow[r, "p(\alpha)^*"] & \Map_{\cat{C}}(pd_0, pd_2)
\end{tikzcd}\end{equation}
is cartesian. 

\item For every $1$-morphism $\beta\colon d'_0\rt d'_1$ in $\cat{D}$, the map between mapping $d$-categories
$$\begin{tikzcd}
\Map_{\Fun_{d+1}([1], \cat{D})}(\alpha, \beta)\arrow[r] & \Map_{\cat{D}}(d_0, d_0')\times_{\Map_{\cat{C}}(pd_0, pd'_0)} \Map_{\Fun_{d+1}([1], \cat{C})}(p\alpha, p\beta)
\end{tikzcd}$$
is an equivalence.
\end{enumerate}
\end{lemma}
\begin{proof}
Using the description of the mapping objects in $\Fun_{d+1}([1], \cat{D})$ from Example \ref{ex:interval}, the map appearing in (2) can be identified with
$$\begin{tikzcd}
\Map_{\cat{D}}(d_0, d_0')\mathop{\times}\limits_{\Map_{\cat{D}}(d_0, d'_1)} \Map_{\cat{D}}(d_1, d'_1)\arrow[r] & \Map_{\cat{D}}(d_0, d_0')\mathop{\times}\limits_{\Map_{\cat{C}}(pd_0, pd'_1)} \Map_{\cat{C}}(pd_1, pd'_1)
\end{tikzcd}$$
In particular, condition (1) is equivalent to condition (2) in the case where $\beta$ is the identity on $d_2$. Conversely, the above map arises as the base change of 
$$
\Map_{\cat{D}}(d_1, d_1')\rt \Map_{\cat{D}}(d_0, d'_1)\times_{\Map_{\cat{C}}(pd_0, pd'_1)}\Map_{\cat{C}}(pd_1, pd'_1)
$$
along the map $\beta_*\colon \Map_{\cat{D}}(d_0, d'_0)\rt \Map_{\cat{D}}(d_0, d_1')$. It follows that (1) implies (2).
\end{proof}
\begin{corollary}\label{cor:classical def of pre-cocart}
A map $p\colon \cat{D}\rt \cat{C}$ of $(d+1)$-categories is a pre-cocartesian fibration if and only if the following condition holds: for every map $\gamma\colon c_0\rt c_1$ in $\cat{C}$ and every $d_0\in p^{-1}(c_0)$, there exists an arrow $\alpha\colon d_0\rt d_1$ in $\cat{D}$ with $p(\alpha)=\gamma$, satisfying the equivalent conditions of Lemma \ref{lem:cocartesian morphisms}.
\end{corollary}
In this case, the $p$-cocartesian morphisms are precisely the morphisms satisfying the equivalent conditions of Lemma \ref{lem:cocartesian morphisms}. For the $1$-categorical case, see also \cite{aya17}.
\begin{proof}
Using Proposition \ref{prop:adj via rep}, the existence of lifts $\alpha$ satisfying the second condition of Lemma \ref{lem:cocartesian morphisms} is equivalent to $p$ being a pre-cocartesian fibration.
\end{proof}
\begin{remark}\label{rem:precocart reform}
The condition appearing in Corollary \ref{cor:classical def of pre-cocart} can also be reformulated as follows: the map $\core_1(\cat{D})\rt \core_1(\cat{C})$ is a cocartesian fibration of $1$-categories and each map of $1$-categories $
\cat{D}(-, \vec{0}_{d})\rt \cat{D}(-, \vec{n}_{d})
$ preserves cocartesian arrows. Using this, a map $f\colon \cat{D}\rt \cat{D}'$ preserves cocartesian $1$-morphisms if and only if the map $\core_1(\cat{D})\rt \core_1(\cat{D}')$ preserves cocartesian arrows in the the usual sense.
\end{remark}
\begin{corollary}\label{cor:disjoint unions of pre-cocart}
Consider a square of $(d+1)$-categories
$$\begin{tikzcd}
\cat{D}\arrow[d, "p"{swap}]\arrow[r] & T\arrow[d, "q"]\\
\cat{C}\arrow[r] & S.
\end{tikzcd}$$
where $T$ and $S$ are spaces. Then $p$ is a pre-cocartesian fibration if and only if for any $t\in T$, the map on fibers $p_t\colon \cat{D}_t\rt \cat{C}_{q(t)}$ is a pre-cocartesian fibration. Furthermore, a morphism in $\cat{D}$ is $p$-cocartesian if and only if it defines a $p_t$-cocartesian morphism in some fiber $\cat{D}_t$.
\end{corollary}
\begin{proof}
Suppose that $\alpha\colon d_0\rt d_1$ is an arrow in $\cat{D}$ and let $\tau\colon t_0\rt t_1$ be its image in $T$ (which is an equivalence). For any object $d_2\in \cat{D}$, the square of $(d+1)$-categories \eqref{diag:standard cocartesian condition} then comes with a natural transformation to the square of spaces
$$\begin{tikzcd}
\Map_{T}(t_1, t_2)\arrow[r, "\tau^*"]\arrow[d] & \Map_T(t_0, t_2)\arrow[d]\\
\Map_S(qt_1, qt_2)\arrow[r, "q(\tau)^*"] & \Map_S(qt_0, qt_1).
\end{tikzcd}$$
The horizontal maps in this square are equivalences. The square \eqref{diag:standard cocartesian condition} of $(d+1)$-categories (i.e.\ specific kinds of $(d+1)$-fold simplicial spaces) is then cartesian if and only if for for every $\tau'\colon t_1\rt t_2$ in $T$, the square of fibers over $\tau'$ is a cartesian diagram of $(d+1)$-categories. This means precisely that $\alpha$ defines a cocartesian arrow for $p_t\colon \cat{D}_{t}\rt \cat{C}_{q(t)}$, where $\cat{D}_t$ is the fiber of $\cat{D}$ over the diagram in $T$ of the form 
$$
t_0\rto{\tau} t_1\rto{\tau'} t_2.
$$
Note that such a diagram in $T$ is contractible, i.e.\ it is equivalent to a choice of basepoint $t\in T$. In other words, an arrow in $\cat{D}$ is $p$-cocartesian if and only if it defines a $p_t$-cocartesian arrow where $p_t\colon \cat{D}_t\rt \cat{C}_{q(t)}$ is the map between fibers over $t\in T$. Using this and Corollary \ref{cor:classical def of pre-cocart}, one readily deduces that $p$ is a pre-cocartesian fibration if and only if each $p_t$ is a pre-cocartesian fibration.
\end{proof}

\subsection{\texorpdfstring{$d$}{d}-cocartesian fibrations}
We now turn to the definition of $d$-cocartesian fibrations, which proceeds by induction on $d$:
\begin{definition}\label{def:cocart fib}
Let $p\colon \cat{D}\rt \cat{C}$ be a map of $n$-categories. We then have the following definitions, by induction on $d$:
\begin{enumerate}[label=(\alph*)]
\item $p$ is said to be a \emph{$0$-cocartesian fibration} if $\cat{D}(1)\rt \cat{C}(1)\times_{\cat{C}(\{0\})}\cat{D}(\{0\})$ is an equivalence. A \emph{strong morphism} between two $0$-cocartesian fibrations $p$ and $p'$ is a commuting square
\begin{equation}\label{diag:strong map}\begin{tikzcd}
\cat{D}\arrow[d, "p"{swap}]\arrow[r, "f"] & \cat{D}'\arrow[d, "p'"]\\
\cat{C}\arrow[r] & \cat{C}'
\end{tikzcd}\end{equation}
Dually, $p$ is said to be a \emph{$0$-cartesian fibration} if $p^{\op}\colon \cat{D}^{\op}\rt \cat{C}^{\op}$ is a $0$-cocartesian fibration; strong morphisms between those are commuting squares \eqref{diag:strong map}.

\item $p$ is said to be a \emph{$d$-cocartesian fibration} if it satisfies the following two conditions:
\begin{enumerate}
\item It is a \emph{homwise $(d-1)$-cartesian fibration:} for every $x_0, x_1\in \cat{D}$, the induced map between mapping objects $\Map_{\cat{D}}(x_0, x_1)\rt \Map_{\cat{C}}(p(x_0), p(x_1))$ is a $(d-1)$-cartesian fibration and the composition map determines a strong map of $(d-1)$-cartesian fibrations
\begin{equation}\label{diag:composition preserves cart}\begin{tikzcd}
\Map_{\cat{D}}(x_1, x_2)\times \Map_{\cat{D}}(x_0, x_1)\arrow[d, "p"{swap}]\arrow[r] & \Map_{\cat{D}}(x_0, x_2)\arrow[d, "p"]\\
\Map_{\cat{C}}(p(x_1), p(x_2))\times \Map_{\cat{C}}(p(x_0), p(x_1))\arrow[r] & \Map_{\cat{C}}(p(x_0), p(x_2)).
\end{tikzcd}\end{equation}

\item It is a pre-cocartesian fibration (Definition \ref{def:pre-cocart}).
\end{enumerate}

\item 
$p$ is said to be a \emph{$d$-cartesian fibration} if it satisfies the following two conditions:
\begin{enumerate}[label={({\arabic*}')}]
\item It is a \emph{homwise $(d-1)$-cocartesian fibration:} it induces $(d-1)$-cocartesian fibration on mapping objects and the composition determines a strong map of $(d-1)$-cocartesian fibrations \eqref{diag:composition preserves cart}.

\item It is a pre-cartesian fibration (Remark \ref{rem:pre-cart}).
\end{enumerate}

\item
A \emph{strong morphism} between $d$-(co)cartesian fibrations is a commuting square \eqref{diag:strong map} such that $f$ preserves (co)cartesian $1$-morphisms and for any $x_0, x_1\in \cat{D}$, the induced square
$$\begin{tikzcd}
\Map_{\cat{D}}(x_0, x_1)\arrow[r, "f"]\arrow[d, "p"{swap}] & \Map_{\cat{D}'}(fx_0, fx_1)\arrow[d, "p'" ]\\
\Map_{\cat{C}}(px_0, px_1) \arrow[r] & \Map_{\cat{C}'}(p'(f x_0), p'(f x_1))
\end{tikzcd}$$
is a strong morphism of $(d-1)$-(co)cartesian fibrations.
\end{enumerate}
\end{definition}
\begin{warning}
A $0$-cocartesian fibration can equivalently be viewed as a Segal copresheaf (as in Definition \ref{def:lfib}) whose domain is an $n$-category. Note that the Segal copresheaves that arise in this way form a very restrictive class, since their fibers are spaces (cf.\ Warning \ref{war:lfib}).
\end{warning}
\begin{remark}\label{rem:opposites of cocart fib}
As usual, there are $2^{d+1}$ variants of the notion of a $d$-cocartesian fibration between $(d+1)$-categories, by taking opposites in the various dimensions. We will only make use of the two notions considered above, which are related by taking opposites in every dimension.
\end{remark}
\begin{remark}
The definition of a $d$-cocartesian fibration $p\colon \cat{D}\rt \cat{C}$ exhibits an alternating pattern: odd-dimensional cells are required to admit $p$-cocartesian lifts and even-dimensional cells are required to admit $p$-cartesian lifts. Let us explain heuristically why this should encode a fully \emph{covariant} diagram of $d$-categories indexed by $\cat{C}$.

First, to make sure that the fiber $p^{-1}(c_0)$ over an object $c_0\in \cat{C}$ depends covariantly on $c_0$, we need to be able to associate to each $1$-morphism $\alpha\colon c_0\rt c_1$ in $\cat{C}$ and each $d_0\in p^{-1}(c_0)$ an object $\alpha_!d_0\in p^{-1}(c_1)$. Of course, this is done using the (essentially unique) $p$-cocartesian arrow $\tilde{\alpha}\colon d_0\rt \alpha_!d_0$ covering $\alpha$. 

Next, given two objects $d_0, d_1\in \cat{D}$ with images $c_0, c_1\in \cat{C}$, consider the induced map $p\colon \Map_{\cat{D}}(d_0, d_1)\rt \Map_{\cat{C}}(c_0, c_1)$. For each $\alpha\colon c_0\rt c_1$ in $\cat{C}$, its fiber is given by
$$
p^{-1}(\alpha)\simeq \Map_{p^{-1}(c_1)}(\alpha_!d_0, d_1).
$$
Consequently, if $\alpha_!d_0$ depends \emph{covariantly} on the $1$-morphism $\alpha$, then the fiber $p^{-1}(\alpha)$ will depend \emph{contravariantly} on $\alpha$. Indeed, each $2$-morphism $h\colon \alpha\rt \beta$ in $\cat{C}$ should give rise to a functor
\begin{equation}\label{diag:precomp with 2cell}\begin{tikzcd}
h^*\colon p^{-1}(\beta)\arrow[r] & p^{-1}(\alpha); \quad \big(\beta_!d_0\rto{g} d_1\big)\arrow[r, mapsto] & \big(\alpha_!d_0\xrightarrow{h(d_0)}\beta_!d_0\rto{g} d_1\big)
\end{tikzcd}\end{equation}
that precomposes with the natural transformation induced by $h$. In terms of the fibration $p$, this means that each $2$-morphism $h\colon \alpha\rt \beta$ in $\cat{C}$ and $g\in p^{-1}(\beta)$ should admit a \emph{$p$-cartesian} lift $\tilde{h}\colon h^*(g)\rt g$.

At the level of $2$-morphisms, we then obtain a functor of the form
$$\begin{tikzcd}
p\colon \Map_{\Map_{\cat{D}}(d_0, d_1)}(f, g)\arrow[r] & \Map_{\Map_{\cat{C}}(c_0, c_1)}(\alpha, \beta)
\end{tikzcd}$$
for each $f\in p^{-1}(\alpha)$ and $g\in p^{-1}(\beta)$. For a given $2$-cell $h\colon \alpha\to \beta$ in $\cat{C}$, postcomposition with the $p$-cartesian lift $\tilde{h}$ induces an equivalence 
$$
p^{-1}(h)\simeq \Map_{p^{-1}(\alpha)}(f, h^*g).
$$
If the natural transformation $h(d_0)\colon \alpha_!d_0\rt \beta_!d_0$ depends covariantly on the $2$-morphism $h$, then \eqref{diag:precomp with 2cell} shows that $h^*$ depends covariantly on $h$ as well. Consequently, the fiber $p^{-1}(h)$ depends covariantly on $h$, so that $3$-morphisms should have $p$-cocartesian lifts. In higher dimensions, one can repeat the discussion from the last two paragraphs to arrive at the alternating list of cocartesian and cartesian conditions from Definition \ref{def:cocart fib}.
\end{remark}
%
%
Note that for any map of $n$-categories $p\colon \cat{D}\rt \cat{C}$, the property of being a $d$-cocartesian fibration does not depend on whether we consider $p$ as a map of $n$-categories or $(n+1)$-categories. Similarly, the difference between $d$-cocartesian and $(d+1)$-cocartesian fibrations is only fiberwise:
\begin{lemma}\label{lem:d-cocart as d+1-cocart}
A map $p\colon \cat{D}\rt \cat{C}$ between $n$-categories is a $d$-cocartesian fibration if and only if it is a $(d+1)$-cocartesian fibration and its fibers are $d$-categories.
\end{lemma}
In particular, a functor between $n$-categories is a $d$-cocartesian fibration for some $d>n$ if and only if is an $n$-cocartesian fibration.
\begin{proof}
An inductive argument reduces this to the case $d=0$: if $p$ is a $0$-cocartesian fibration, then certainly its fibers are spaces and it is a homwise $0$-cartesian fibration (since each $\cat{D}(p)\rt \cat{C}(p)$ is the base change of a map between spaces). It is a $1$-cocartesian fibration since every arrow in $\cat{D}$ is $p$-cocartesian and every morphism in $\cat{C}$ can be lifted to an arrow in $\cat{D}$ (since $p$ is a $0$-cocartesian fibration).

Conversely, suppose that $p\colon \cat{D}\rt \cat{C}$ is a $1$-cocartesian fibration whose fibers are spaces. We have to verify that the map $\cat{D}(1)\rt \cat{C}(1)\times_{\cat{C}(\{0\})}\cat{D}(\{0\}) $ is an equivalence of $(n-1)$-categories. Note that this map is obtained from the map of $n$-categories \eqref{diag:precocart condition} by taking $(n-1)$-cores. It particular, it admits a fully faithful left adjoint whose essential image consists of the $p$-cocartesian morphisms. It therefore suffices to verify that every $1$-morphism in $\cat{D}$ is $p$-cocartesian: this follows immediately from the fact that every $1$-morphism $\alpha$ factors as $\alpha=\alpha''\circ \alpha'$, where $\alpha'$ is cocartesian and $\alpha''$ is a fiberwise morphism (and hence an equivalence).
\end{proof}
\begin{remark}\label{rem:highest dimension of generality}
It will follow from straightening that the notion of a $d$-cocartesian fibration essentially stabilizes at the level of $(d+1)$-categories: if $p$ is a $d$-cocartesian fibration between $(d+k)$-categories for $k\geq 2$, then $p$ arises as the base change of a $d$-cocartesian fibration $\cat{D}'\rt |\cat{C}|_{d+1}$, where $|\cat{C}|_{d+1}$ is the $(d+1)$-category obtained by inverting all morphisms in dimension $> d+1$.
\end{remark}

\begin{definition}[Lateral $k$-morphisms]
Let $p\colon \cat{D}\rt \cat{C}$ be a $d$-cocartesian fibration between $n$-categories. For any $1\leq k\leq d$, let $\alpha\in \cat{D}(1, \dots, 1 , \vec{0}_{n-k})$ be a $k$-morphism in $\cat{D}$, considered as a $1$-morphism $\alpha\colon f\rt g$ between two $(k-1)$-morphisms $f, g \colon x\rt y$. Let
$$\begin{tikzcd}
p_{(k-1)}\colon \Map_{\cat{D}}(x, y)\arrow[r] & \Map_{\cat{C}}(p(x), p(y))
\end{tikzcd}$$
be the induced map between $(n+1-k)$-categories of $(k-1)$-morphisms between $x$ and $y$. Then $\alpha$ is said to be ($p$-)\emph{lateral} if it is:
\begin{itemize}
\item a $p_{(k-1)}$-cocartesian $1$-morphism if $k$ is odd.
\item a $p_{(k-1)}$-cartesian $1$-morphism if $k$ is even.
\end{itemize}
When $p$ is an $d$-cartesian fibration, we define \emph{lateral} $k$-morphisms similarly, but with the cases of $k$ even and odd exchanged.
In particular, the meaning depends on whether $p$ is considered as a $d$-cocartesian or $d$-cartesian fibration.
\end{definition}
Unraveling the definitions, a square \eqref{diag:strong map} defines a strong morphism between two $d$-cocartesian fibrations if and only if the map $f$ preserves lateral $k$-morphisms for all $k\leq d$. 

\begin{lemma}\label{lem:stability properties cocart2}
Let $p\colon \cat{D}\rt \cat{C}$ be a map of $n$-categories covering a map of spaces $q\colon T\rt S$, as in Corollary \ref{cor:disjoint unions of pre-cocart}. Then $p$ is a $d$-cocartesian fibration if and only if each map on fibers $p_t\colon \cat{D}_t\rt \cat{C}_{q(t)}$ is a $d$-cocartesian fibration. For each $k\leq d$, a $k$-morphism in $\cat{D}$ covering a point $t\in T$ is $p$-lateral if and only if it is $p_t$-lateral.
\end{lemma}
\begin{proof}
This follows by induction, using Corollary \ref{cor:disjoint unions of pre-cocart}.
\end{proof}
\begin{lemma}\label{lem:local cartesian}
A map $p\colon \cat{D}\rt \cat{C}$ is a homwise $(d-1)$-cartesian fibration if and only if it satisfies the following condition: for every $[m]\in \Del$, the map $\cat{D}(m)\rt \cat{C}(m)$ is a $(d-1)$-cartesian fibration and each $[m]\rt [m']$ induces a strong morphism between $(d-1)$-cartesian fibrations.
\end{lemma}
\begin{proof}
Taking sources and targets defines a natural transformation from the map $\cat{D}(1)\rt \cat{C}(1)$ to the map of spaces $\cat{D}(0)\times \cat{D}(0)\rt \cat{C}(0)\times \cat{C}(0)$. By Lemma \ref{lem:stability properties cocart2}, $\cat{D}(1)\rt \cat{C}(1)$ is then a $(d-1)$-cartesian fibration if and only if for each $d_0, d_1\in \cat{D}$, the map on fibers $\Map_{\cat{D}}(d_0, d_1)\rt \Map_{\cat{C}}(pd_0, pd_1)$ is a $(d-1)$-cartesian fibration. In turn, the Segal conditions imply that each $\cat{D}(m)\rt \cat{C}(m)$ is a $(d-1)$-cartesian fibration as soon as $\cat{D}(1)\rt \cat{C}(1)$ is.

In this case, the simplicial structure maps induce strong morphisms between $(d-1)$-cartesian fibrations if and only if $d_1\colon \cat{D}(2)\rt \cat{D}(1)$ is a strong morphism. Again, Lemma \ref{lem:stability properties cocart2} shows that this can be verified fiberwise over each triple of objects $d_0, d_1, d_2\in \cat{D}(0)$, which means precisely that $\Map_{\cat{D}}(d_1, d_2)\times \Map_{\cat{D}}(d_0, d_1)\rt \Map_{\cat{D}}(d_0, d_2)$ defines a strong morphism of $(d-1)$-cartesian fibrations.
\end{proof}
\begin{lemma}\label{lem:stability properties cocart1}
The following assertions hold:
\begin{enumerate}
\item The class of $d$-cocartesian fibrations is closed under composition and base change.
\item\label{it:functor} If $p\colon \cat{D}\rt \cat{C}$ is a $d$-cocartesian fibration, then $\Fun_n(\cat{B}, \cat{D})\rt \Fun_n(\cat{B}, \cat{C})$ is a $d$-cocartesian fibration as well and restriction along $\cat{B}'\rt \cat{B}$ defines a strong morphism.
\item\label{it:limits of dcocart} The subcategory of $\Fun([1], \Cat_{n})$ spanned by the $d$-cocartesian fibrations and strong morphisms between them is closed under limits.
\end{enumerate}
\end{lemma}
\begin{proof}
All properties follow from Lemma \ref{lem:pre-cocart stable} by induction, the case $d=0$ being readily verified. The only assertion requiring extra care is part \ref{it:functor}. Using part \ref{it:limits of dcocart}, it suffices to verify this in the case where $\cat{B}=[1]_{\cat{A}}$ for an $(n-1)$-category $\cat{A}$ (and likewise for $\cat{B}'$). By Lemma \ref{lem:pre-cocart stable}, $q\colon \Fun_n([1]_{\cat{A}}, \cat{D})\rt \Fun_n([1]_{\cat{A}}, \cat{C})$ is pre-cocartesian and restriction preserves cocartesian $1$-morphisms. To see that it is locally $(d-1)$-cartesian, consider the following diagram of simplicial $(n-1)$-categories, given in simplicial degree $m$ by
$$\begin{tikzcd}
\Fun_{n-1}\big(\cat{A}, \cat{D}\big([1]\times [m]\big)\big)\arrow[d]\arrow[r] &  \Fun_{n-1}\big(\cat{A}, \cat{D}(1)^{\times m+1}\big)\arrow[d] & \core_0\Fun_{n-1}\big(\cat{A}, \cat{D}(1)^{\times m+1}\big)\arrow[d]\arrow[l]\\
 \Fun_{n-1}\big(\cat{A}, \cat{C}\big([1]\times [m]\big)\big)\arrow[r]  & \Fun_{n-1}\big(\cat{A}, \cat{C}(1)^{\times m+1}\big)&  \core_0\Fun_{n-1}\big(\cat{A}, \cat{C}(1)^{\times m+1}\big)\arrow[l]
\end{tikzcd}$$
Using Lemma \ref{lem:local cartesian}, it follows by inductive hypothesis that the vertical maps define simplicial diagrams of $(d-1)$-cartesian fibrations and strong morphisms between them. Furthermore, the horizontal maps define strong morphisms between these $(d-1)$-cartesian fibrations. Taking pullbacks along the rows, we then obtain another simplicial diagram of $(d-1)$-cartesian fibrations and strong morphisms. By Corollary \ref{cor:internal hom}, this is precisely the map $q\colon \Fun_n\big([1]_{\cat{A}}, \cat{D}\big)\rt \Fun_n\big([1]_{\cat{A}}, \cat{C}\big)$.

We conclude by induction and Lemma \ref{lem:local cartesian} that $q$ is locally $(d-1)$-cartesian. From the above description one readily verifies that restriction along $[1]_{\cat{A}'}\rt [1]_{\cat{A}}$ preserves lateral $k$-morphisms for all $1\leq k\leq d$.
\end{proof}

\subsection{The \texorpdfstring{$(d+2)$}{d+2}-category of \texorpdfstring{$d$}{d}-cocartesian fibrations}\label{sec:cat of cocart fibs}
Because of Remark \ref{rem:highest dimension of generality}, we will henceforth \emph{only consider $d$-cocartesian fibrations between $(d+1)$-categories}. In light of Lemma \ref{lem:d-cocart as d+1-cocart}, this is not really restrictive, since we can always choose $d$ large enough to cover $n$-cocartesian fibration between $m$-categories for some given $m$ and $n$.

\begin{definition}\label{def:cat of d-cocart fibs}
We will write $\hcat{Cocart}_d\subseteq \Fun_{d+1}\big([1], \hcat{Cat}_{d+1}\big)$ for the sub-$(d+2)$-category whose objects are the $d$-cocartesian fibrations $p\colon \cat{D}\rt \cat{C}$ and whose $1$-morphisms are the strong morphisms between them (and all higher morphisms between those). Let us write $\pi\colon \hcat{Cocart}_d\rt \hcat{Cat}_{d+1}$ for the codomain projection and $\hcat{Cocart}_d(\cat{C})$ for the fiber over a $(d+1)$-category $\cat{C}$.

Let us furthermore write $\pi\colon \cat{Cocart}_{d}\rt \Cat_{d+1}$ for the induced functor between $1$-cores. The fiber over a $(d+1)$-category $\cat{C}$ is denoted $\cat{Cocart}_{d}(\cat{C})$.
\end{definition}
\begin{remark}
\sloppy One can model $\pi\colon \hcat{Cocart}_d\rt \hcat{Cat}_{d+1}$ by a strict functor between $\CompSegS_{d+1}$-enriched categories, where $\hcat{Cocart}_d$ is modeled by an enriched subcategory of \mbox{$\Fun_{\enr}([1], \CompSegS_{d+1})$} whose objects are fibrant models for $d$-cocartesian fibrations and whose mapping objects are the subobjects consisting of strong morphisms.
\end{remark}
\begin{proposition}\label{prop:cocart forms cart fibration}
The projection $\pi\colon \hcat{Cocart}_d\rt \hcat{Cat}_{d+1}$ is a $(d+1)$-cartesian fibration.
\end{proposition}
\begin{lemma}\label{lem:marked functor cats}
Let $p\colon \cat{D}\rt \cat{C}$ be a $d$-cocartesian fibration and let $\cat{A}$ be a $(d+1)$-category equipped with subspaces $S_k\subseteq \cat{A}(1, \dots, 1, \vec{0}_{d+1-k})$ of `marked $k$-morphisms' for all $1\leq k\leq d$ (without any further condition). Let $\Fun^{\sharp}_{d+1}(\cat{A}, \cat{D})\subseteq \Fun_{d+1}(\cat{A}, \cat{D})$ denote the full sub-$(d+1)$-category spanned by those functors $\cat{A}\rt \cat{D}$ sending every marked $k$-morphism in $\cat{A}$ to a $p$-lateral $k$-morphism in $\cat{D}$. Then the natural map $\Fun^{\sharp}_{d+1}(\cat{A}, \cat{D})\rt \Fun_{d+1}(\cat{A}, \cat{C})$ is a $d$-cocartesian fibration.
\end{lemma}
\begin{proof}
By Lemma \ref{lem:stability properties cocart1}, $\Fun_{d+1}(\cat{A}, \cat{D})\rt \Fun_{d+1}(\cat{A}, \cat{C})$ is a $d$-cocartesian fibration. Since $\Fun^{\sharp}_{d+1}(\cat{A}, \cat{D})$ is a full sub-$(d+1)$-category of $\Fun_{d+1}(\cat{A}, \cat{D})$, the map $\Fun^{\sharp}_{d+1}(\cat{A}, \cat{D})\rt \Fun_{d+1}(\cat{A}, \cat{C})$ remains a homwise $(d-1)$-cartesian fibration. 
It then remains to verify the following assertion: if $\mu\colon f\rt g$ is a cocartesian $1$-morphism in $\Fun_{d+1}(\cat{A}, \cat{D})$ such that $f\in \Fun^\sharp_{d+1}(\cat{A}, \cat{D})$, then $g\in \Fun^\sharp_{d+1}(\cat{A}, \cat{D})$ as well.

Let us first verify that $g$ sends all marked $1$-morphisms in $\cat{A}$ to cocartesian morphisms in $\cat{D}$: indeed, for every such map $\alpha\colon a_0\rt a_1$, there is a commuting square in $\cat{D}$ of the form
$$\begin{tikzcd}
f(a_0)\arrow[r, "\mu"]\arrow[d, "f(\alpha)"{swap}] & g(a_0)\arrow[d, "g(\alpha)"]\\
f(a_1) \arrow[r, "\mu"] & g(a_1).
\end{tikzcd}$$
The horizontal arrows are cocartesian since $\mu$ was a cocartesian natural transformation (Lemma \ref{lem:stability properties cocart1}). Since $f(\alpha)$ is a cocartesian arrow, one sees that $g(\alpha)$ is a cocartesian arrow as well (cf.\ the argument in \cite[Proposition 2.4.1.7]{lur09} and \cite[Lemma 5.1.5]{rie21}).

To see that $g$ preserves  lateral $k$-morphisms for $k>1$, note that for each $a_0, a_1\in \cat{A}$ there is a commuting square of mapping $d$-categories
$$\begin{tikzcd}
\Map_{\cat{A}}(a_0, a_1)\arrow[r, "f"]\arrow[d, "g"{swap}] & \Map_{\cat{D}}(fa_0, fa_1)\arrow[d, "\mu(a_1)_*"]\\
\Map_{\cat{D}}(ga_0, ga_1)\arrow[r, "\mu(a_0)^*"] & \Map_{\cat{D}}(fa_0, ga_1).
\end{tikzcd}$$
Since $p\colon \cat{D}\rt \cat{C}$ is a $d$-cocartesian fibration, the functor $\mu(a_1)_*$ postcomposing with $\mu(a_1)$ is strong, i.e.\ it preserves lateral $k$-morphisms. Since $f$ sends marked $k$-morphisms to lateral $k$-morphisms, it follows that $\mu(a_1)_*\circ f$ does as well. On the other hand, since $\mu(a_0)$ is a cocartesian $1$-morphism, the map $\mu(a_0)^*$ is an equivalence of $d$-categories. This implies that $g$ sends marked $k$-morphisms to lateral $k$-morphisms as well.
\end{proof}
\begin{proof}[Proof of Proposition \ref{prop:cocart forms cart fibration}]
Let us start by verifying that $\pi\colon \hcat{Cocart}_d\rt \hcat{Cat}_{d+1}$ is a homwise $d$-cocartesian fibration. Let $p_0\colon \cat{D}_0\rt \cat{C}_0$ and $p_1\colon \cat{D}_1\rt \cat{C}_1$ be $d$-cocartesian fibrations. We have to check that the induced map on mapping $(d+1)$-categories
$$\begin{tikzcd}
\pi\colon \Map_{\hcat{Cocart}_d}(p_0, p_1)\arrow[r] & \Fun_{d+1}(\cat{C}_0, \cat{C}_1)
\end{tikzcd}$$
is a $d$-cocartesian fibration. To see this, note that one can identify
\begin{equation}\label{eq:mapping categories between cocart}\begin{tikzcd}
\Map_{\hcat{Cocart}_d}(p_0, p_1)=\Fun^{\natural}_{d+1}(\cat{D}_0, \cat{D}_1)\times_{\Fun_{d+1}(\cat{D}_0, \cat{C}_1)} \Fun_{d+1}(\cat{C}_0, \cat{C}_1)
\end{tikzcd}\end{equation}
where $\Fun^{\natural}_{d+1}(\cat{D}_0, \cat{D}_1)\subseteq \Fun_{d+1}(\cat{D}_0, \cat{D}_1)$ is the full sub-$(d+1)$-category spanned by those functors $\cat{D}_0\rt \cat{D}_1$ preserving lateral $k$-morphisms. The result then follows from Lemma \ref{lem:marked functor cats} and stability of $d$-cocartesian fibrations under base change (Lemma \ref{lem:stability properties cocart1}).

Next, given another $d$-cocartesian fibration $p_2\colon \cat{D}_2\rt\cat{C}_2$, we have to verify that the square induced by composition
$$\begin{tikzcd}
\Map_{\hcat{Cocart}_d}(p_1, p_2)\times\Map_{\hcat{Cocart}_d}(p_0, p_1)\arrow[d]\arrow[r, "\circ"] & \Map_{\hcat{Cocart}_d}(p_0, p_2)\arrow[d]\\
\Fun_{d+1}(\cat{C}_1, \cat{C}_2)\times \Fun_{d+1}(\cat{C}_0, \cat{C}_1)\arrow[r, "\circ"] & \Fun_{d+1}(\cat{C}_0, \cat{C}_2)
\end{tikzcd}$$
gives a strong morphism of $d$-cocartesian fibrations. To see this, let $f, g\in \Map(p_0, p_1)$ and $f', g'\in \Map(p_1, p_2)$. We will abuse notation and write $f, g\colon \cat{D}_0\rt \cat{D}_1$ and $f', g'\colon \cat{D}_1\rt \cat{D}_2$ for the underlying functors, which all preserve lateral $k$-morphisms. Let $\alpha\colon f \rt g$ and $\beta\colon f'\rt g'$ be lateral $k$-morphisms (which we depict as 1-morphisms instead of higher cells, for simplicity). Then the image of $(\alpha, \beta)$ under $\circ$ is given by the horizontal composition 
$$\begin{tikzcd}
f'f\arrow[r, "f'\alpha"] & f'g\arrow[r, "\beta g"] & g'g.
\end{tikzcd}$$
Then $f'\alpha$ is lateral since $f'$ preserves lateral $k$-morphisms and $\beta g$ is lateral since lateral natural transformations are detected pointwise. Consequently, their horizontal composite is a lateral $k$-morphism as well.

It remains to check that $\pi\colon \hcat{Cocart}_d\rt \hcat{Cat}_{d+1}$ is a pre-cartesian fibration. Given a map $f\colon \cat{C}_1\rt \cat{C}_2$ and a $d$-cocartesian fibration $p_2\colon \cat{D}_2\rt \cat{C}_2$, let $p_1\colon \cat{D}_1=\cat{D}_2\times_{\cat{C}_2}\cat{C}_1\rt \cat{C}_1$ be the base change. We have to show that the canonical (strong) map of cocartesian fibrations $\tilde{f}\colon \cat{D}_1\rt \cat{D}_2$ defines a cartesian lift of $f$, i.e.\ that for any $p_0\colon \cat{D}_0\rt \cat{C}_0$, the square
$$\begin{tikzcd}
\Map_{\hcat{Cocart}_d}(p_0, p_1)\arrow[r, "\tilde{f}_*"]\arrow[d] & \Map_{\hcat{Cocart}_d}(p_0, p_2)\arrow[d]\\
\Fun_{d+1}(\cat{C}_0, \cat{C}_1)\arrow[r, "f_*"] & \Fun_{d+1}(\cat{C}_0, \cat{C}_2)
\end{tikzcd}$$
is cartesian. Using \eqref{eq:mapping categories between cocart}, one sees that this comes down to verifying that
$$\begin{tikzcd}
\Fun^{\natural}_{d+1}\big(\cat{D}_0, \cat{D}_2\times_{\cat{C}_2}\cat{C}_1\big)\arrow[r]\arrow[d] &  \Fun^{\natural}_{d+1}\big(\cat{D}_0, \cat{D}_2\big)\arrow[d]\\
\Fun^{\natural}_{d+1}\big(\cat{D}_0, \cat{C}_1\big)\arrow[r] & \Fun_{d+1}\big(\cat{D}_0, \cat{C}_2\big)
\end{tikzcd}$$
is cartesian. This follows immediately from the fact that a functor $\cat{D}_0\rt \cat{D}_2\times_{\cat{C}_2}\cat{C}_1$ preserves lateral $k$-morphisms if and only if the corresponding map $\cat{D}_0\rt \cat{D}_2$  does.
\end{proof}

\section{Straightening and unstraightening}\label{sec:straightening}
Finally, we turn to straightening and unstraightening: for a small $(d+1)$-category $\cat{C}$, we will establish an equivalence between $d$-cocartesian fibrations over $\cat{C}$ and copresheaves $\cat{C}\rt \hcat{Cat}_d$. We will start by describing this as an equivalence of $1$-categories (Theorem \ref{thm:unstraightening}). We then use this to establish a more structured version of straightening-unstraightening that provides an equivalence of $(d+1)$-categories and is furthermore $(d+2)$-functorial in $\cat{C}$ (Theorem \ref{thm:unstraightening very functorial}).

\subsection{Straightening and unstraightening: 1-categorical version}
Our first goal will be to prove the following: 
\begin{theorem}\label{thm:unstraightening}
There is an equivalence of cartesian fibrations between $1$-categories
$$\begin{tikzcd}
\SegCoPSh_d\arrow[rr, "\sim"]\arrow[rd] & & \cat{Cocart}_d\arrow[ld]\\
& \Cat_{d+1}.
\end{tikzcd}$$
In particular, for each $(d+1)$-category $\cat{C}$ this induces equivalences of $1$-categories
$$\begin{tikzcd}
\core_1\Fun_{d+1}(\cat{C}, \hcat{Cat}_d)\arrow[r,  "\sim"] & \SegCoPSh_d(\cat{C})\arrow[r, "\sim"] & \cat{Cocart}_d(\cat{C})
\end{tikzcd}$$
between (Segal) copresheaves of $d$-categories on $\cat{C}$ and $d$-cocartesian fibrations over $\cat{C}$. 
\end{theorem}
\begin{remark}\label{rem:cart case}
Similarly, for every $(d+1)$-category $\cat{C}$, there is an equivalence of $1$-categories between cartesian fibrations and presheaves:
$$
\core_1\Fun_{d+1}\big(\cat{C}^{\op}, \hcat{Cat}_d\big)\simeq \core_1\Fun_{d+1}\big(\cat{C}^{{\scriptscriptstyle (1, \dots, d+1)-}\op}, \hcat{Cat}_d\big)\simeq \Cocart_d\big(\cat{C}^{{\scriptscriptstyle (1, \dots, d+1)-}\op}\big)\simeq \cat{Cart}_d(\cat{C}).
$$
Here the first equivalence sends a functor $F\colon \cat{C}^{\op}\rt \hcat{Cat}_d$ to the functor $c\mapsto F(c)^{(1,\dots, d)-\op}$ (which changes the variance with respect to higher morphisms). The second equivalence is Theorem \ref{thm:unstraightening} and the last equivalence takes $(1, \dots, d+1)$-opposites to obtain a cartesian fibration (Remark \ref{rem:opposites of cocart fib}), whose fibers are equivalent to the values of $F$.
\end{remark}
We will prove this theorem by a repeated application of Theorem \ref{thm:main theorem} in different simplicial directions. Concretely, the above equivalence will arise as a composition of $d$ equivalences, whose intermediate categories consist of certain fibrations of $(d+1)$-fold categories:
\begin{definition}\label{def:k-th stage fib}
Let $\cat{C}$ be a $(d+1)$-category and let $p\colon \dcat{D}\rt \cat{C}$ be a map of $(d+1)$-fold categories. By induction, we will say that $p$ is:
\begin{itemize}
\item a \emph{$0$-th stage} fibration if $\dcat{D}(1)\rt \cat{C}(1)\times_{\cat{C}(\{0\})} \dcat{D}(\{0\})$ is an equivalence of $d$-fold categories. A \emph{strong map} of $0$-th stage fibrations is simply a commuting square.

\item a \emph{$k$-th stage fibration}, for $1\leq k\leq d$, if it satisfies the following two conditions.
\begin{enumerate}[label=(a\arabic*)]
\item\label{it:loccartobjects} Each $\dcat{D}(n)^{(1, \dots, d)-\op}\rt \cat{C}(n)^{(1, \dots, d)-\op}$ is a $(k-1)$-st stage fibration and each $[n]\rt [n']$ induces a strong map between $(k-1)$-st stage fibrations.

\item\label{it:precocartobjects} The map of categories $\dcat{D}(-, \vec{0}_d)\rt \cat{C}(-, \vec{0}_d)$ is a cocartesian fibration and each map $\dcat{D}(-, \vec{0}_d)\rt \dcat{D}(-, \vec{n}_d)\times_{\cat{C}(-, \vec{n}_d)} \cat{C}(-, \vec{0}_d)$ preserves cocartesian arrows (over $\cat{C}(-, \vec{0}_d)$).
\end{enumerate}
A strong map between $k$-th stage fibrations is a commuting square such that:
\begin{enumerate}[label=(b\arabic*)]
\item\label{it:loccartmorphisms} evaluating at $[n]\in \Del$ in the first coordinate yields a strong map of $(k-1)$-st stage fibrations.
\item\label{it:precocartmorphisms} the map $\dcat{D}(-, \vec{0}_d)\rt \dcat{D}'(-, \vec{0}_d)$ preserves cocartesian arrows.
\end{enumerate}
\end{itemize}
We will write $\cat{Fib}_d^{(k)}\subseteq \Fun\big([1], \Cat^{\otimes d+1}\big)$ for the subcategory whose objects are $k$-th stage fibrations (in particular, the codomain is a $(d+1)$-category) and whose maps are commuting squares satisfying conditions \ref{it:loccartmorphisms} and \ref{it:precocartmorphisms}. The codomain projection defines a cartesian fibration
$$\begin{tikzcd}
\mm{codom}\colon \cat{Fib}_d^{(k)}\arrow[r] & \Cat_{d+1}.
\end{tikzcd}$$
\end{definition}
\begin{proposition}\label{prop:k-fibs to k+1-fibs}
For each $1\leq k\leq d$, there is an equivalence of categories
$$\begin{tikzcd}
\cat{Fib}_d^{(k-1)}\arrow[rr, "\sim"]\arrow[rd, "\mm{codom}"{swap}] & & \cat{Fib}_d^{(k)}\arrow[ld, "\mm{codom}"]\\
& \Cat_{d+1}.
\end{tikzcd}$$
\end{proposition}
\begin{proof}
Let us fix $k\geq 1$ and start with the following observation: if $p\colon \dcat{D}\rt \cat{C}$ is a $k$-th stage fibration, then each
\begin{equation}\label{diag:double cat part of k-fib}\begin{tikzcd}
\dcat{D}(\vec{n}_{k-1}, -, -, \vec{n}_{d-k})\arrow[r] & \cat{C}(\vec{n}_{k-1}, -, -, \vec{n}_{d-k})
\end{tikzcd}\end{equation}
is a (cart, left)-fibration if $k$ is even, and a (cocart, right)-fibration if $k$ is odd. Furthermore, each $\vec{n}_{k-1}\rt \vec{n}_{k-1}'$ and $\vec{n}_{d-k}\rt \vec{n}_{d-k}'$ induces a map preserving (co)cartesian morphisms. When $k=1$ this follows from condition \ref{it:precocartobjects} and the fact that $\cat{C}(-, \vec{0}_d)\simeq \cat{C}(-, 0, \vec{n}_{d-1})$; for higher $k$ it follows inductively, by repeatedly applying condition \ref{it:loccartobjects}.

From this point on, let us assume that $k$ is odd; when $k$ is even, the same argument applies up to taking opposites (in all dimensions). Note that for a $(k-1)$-st stage fibration $p\colon \dcat{D}\rt \cat{C}$, the map \eqref{diag:double cat part of k-fib} is a (left, cart)-fibration and each $\vec{n}_{k-1}\rt \vec{n}_{k-1}'$ and $\vec{n}_{d-k}\rt \vec{n}_{d-k}'$ induces a strong map of (left, cart)-fibrations. Indeed, the previous paragraph shows that $p$ is a left fibration in the $k$-th variable. Furthermore, setting the $k$-th entry equal to $0$ yields a map $\dcat{D}(\vec{n}_{k-1}, 0, -, \vec{n}_{d-k})\rt \cat{C}(\vec{n}_{k-1}, 0, -, \vec{n}_{d-k})$ from a $1$-category to a space. In particular, each such map is a cartesian fibration and since its cartesian arrows are the equivalences, each $\vec{n}_{k-1}\rt \vec{n}_{k-1}'$ and $\vec{n}_{d-k}\rt \vec{n}_{d-k}'$ induces a map preserving cartesian arrows.

Denoting $\cat{I}=\big(\Del^{\times k-1}\times \Del^{\times d-k}\big)^{\op}$, we therefore obtain a diagram
$$\begin{tikzcd}[column sep=0.4pc, row sep=1pc]
\cat{Fib}_d^{(k-1)}\arrow[rrr, hook]\arrow[rrd, dotted]\arrow[rrddd, "\mm{codom}"{swap}, bend right=10] & \hspace{10pt} & & \Fun\big(\cat{I}, \cat{Fib}^{\mm{left, cart}}\big)\arrow[rddd, bend right=10]\arrow[rd, "\Psi^\perp_{(k, k+1)}", "\sim"{swap}]\\
& & \cat{Fib}_d^{(k)}\arrow[rr, hook, crossing over]\arrow[dd, "\mm{codom}"] & & \Fun\big(\cat{I}, \cat{Fib}^{\mm{cocart, right}}\big)\arrow[dd, "\mm{codom}"]\\
\\
& & \Cat_d\arrow[rr, hook] &  & \Cat^{\otimes d}.
\end{tikzcd}$$
Here the functors from left to right are subcategory inclusions and the top right equivalence is the reflection functor from Theorem \ref{thm:main theorem}, applied to the $k$-th and $(k+1)$-st variable. We will prove by induction on $k$ that the equivalence $\Psi^\perp_{(k, k+1)}$ identifies the subcategories of $(k-1)$-st stage fibrations and $k$-th stage fibrations.

When $k=1$, note that the bottom and front square are pullbacks: a map $\dcat{D}\rt \cat{C}$ is a $0$-th stage fibration if and only if it is a natural (left, cart)-fibration in the first two variables, and a $1$-st stage fibration if and only if it is a natural (cocart, right)-fibration in the first two variables. In particular, $\Psi_{(1, 2)}^\perp$ identifies $0$-th stage fibrations and $1$-st stage fibrations.

For $k>1$, the inductive hypothesis implies that $p\colon \dcat{D}\rt \cat{C}$ satisfies condition \ref{it:loccartobjects} for $(k-1)$-st stage fibrations if and only if $\Psi_{(k, k+1)}^\perp(p)\colon \Psi^\perp(\dcat{D})\rt \cat{C}$ satisfies condition \ref{it:loccartobjects} for $k$-th stage fibrations. For condition \ref{it:precocartobjects}, note that $\dcat{D}\times_{\cat{C}}\core_1(\cat{C})\rt \core_1(\cat{C})$ is given in the $k$-th and $(k+1)$-st variables by a (left, cart)-fibration over a space. For such (left, cart)-fibrations over spaces, the functor $\Psi^\perp$ is equivalent to the functor that exchanges the $k$-th and $(k+1)$-st variable (Proposition \ref{prop:why its called reflection}). Using that $\Psi^\perp$ is compatible with base change, we therefore find
\begin{align*}
\Psi_{(k, k+1)}^\perp(\dcat{D})\big(-,\vec{n}_{k-2},m, n, \vec{n}_{d-k}\big)&\times_{\cat{C}(-,\vec{n}_{k-2},m, n, \vec{n}_{d-k})}\cat{C}(-, \vec{0}_d)\\
&\simeq \dcat{D}\big(-,\vec{n}_{k-2},n, m, \vec{n}_{d-k}\big)\times_{\cat{C}(-,\vec{n}_{k-2},n, m, \vec{n}_{d-k})}\cat{C}(-, \vec{0}_d).
\end{align*}
It follows that $p$ satisfies condition \ref{it:precocartobjects} if and only if $\Psi_{(k, k+1)}^\perp(p)$ does. 
A similar argument shows that $\Psi_{(k, k+1)}^\perp$ identifies strong morphisms between $(k-1)$-st stage fibrations and $k$-th stage fibrations, so that $\Psi_{(k, k+1)}^\perp$ indeed restricts to an equivalence $\cat{Fib}_d^{(k-1)}\rto{\sim} \cat{Fib}_d^{(k)}$.
\end{proof}
\begin{corollary}\label{cor:pullback fiberwise}
Consider a strong morphism of $k$-th stage fibrations
$$\begin{tikzcd}
\dcat{D}\arrow[r, "f"]\arrow[d, "p"{swap}] & \dcat{D}'\arrow[d, "p'"]\\
\cat{C}\arrow[r] & \cat{C}'.
\end{tikzcd}$$
This square is cartesian if and only if the map $f$ induces equivalences on fibers $\dcat{D}_c\rt \dcat{D}'_c$ for each object $c\in \cat{C}$.
\end{corollary}
\begin{proof}
This is evident for $k=0$. For higher $k$, it follows from Proposition \ref{prop:k-fibs to k+1-fibs}.
\end{proof}
\begin{lemma}\label{lem:d-th stage with d-cat fibers}
Let $p\colon \dcat{D}\rt \cat{C}$ be a $d$-th stage fibration from a $(d+1)$-fold category to a $(d+1)$-category. Then the following two conditions are equivalent:
\begin{enumerate}
\item For any object $c\in \cat{C}$, the fiber $\dcat{D}_c$ is a $d$-category in the first $d$-variables, and constant in the last variable.
\item $\dcat{D}$ is a $(d+1)$-category.
\end{enumerate}
\end{lemma}
\begin{proof}
We proceed by induction on $d$, the case $d=0$ being obvious. Note that a $d$-th stage fibration is either a left or a right fibration in the last variable, so that $\dcat{D}(\vec{n}_d, p)\simeq \dcat{D}(\vec{n}_d, 0)\times_{\cat{C}(\vec{n}_d, 0)} \cat{C}(\vec{n}_{d}, p)$. This immediately shows that (2) implies (1). For the converse, note that (1) implies that the fibers of $\dcat{D}(0)\rt \cat{C}(0)$ are spaces, so that $\dcat{D}(0)$ is a space. It then remains to verify that $\dcat{D}(1)$ is a $d$-category. Since $\dcat{D}(1)\rt \cat{C}(1)$ is the opposite of a $(d-1)$-st stage fibration, the inductive hypothesis says that it suffices to verify that its fibers are $(d-1)$-categories. In particular, we can reduce to the case where $\cat{C}\simeq \core_1(\cat{C})$ is a $1$-category.

In this case, take a morphism $\gamma\colon c_0\rt c_1$ in $\cat{C}$ and consider the map on fibers $\dcat{D}(1)_{\gamma}\rt \dcat{D}(\{0\})_{c_0}$. Since the target is a space, it suffices to verify that the fiber over each point $d_0\in \dcat{D}_{c_0}$ is a $(d-1)$-category. Using condition \ref{it:precocartobjects}, we can take a cocartesian lift $\alpha\colon d_0\rt d_1$ of $\gamma$ such that precomposition with $\alpha$ determines an equivalence of $(d-1)$-fold categories
$$\begin{tikzcd}
\alpha^*\colon \dcat{D}(1)_{\mm{id}_{c_1}}\times_{\dcat{D}(\{0\})} \{d_1\}\arrow[r, "\sim"] & \dcat{D}(1)_{\gamma}\times_{\dcat{D}(\{0\})} \{d_0\}.
\end{tikzcd}$$
The domain is a $(d-1)$-category since the fiber $\dcat{D}_{c_1}$ is a $d$-category by assumption.
\end{proof}
\begin{proposition}\label{prop:cocartesian full inside d-fibration}
There is a fully faithful inclusion $\cat{Cocart}_{d}\hooklongrightarrow \cat{Fib}_d^{(d)}$, whose essential image consists of the $d$-th stage fibrations satisfying the equivalent conditions of Lemma \ref{lem:d-th stage with d-cat fibers}. 
\end{proposition}
\begin{proof}
Both $\cat{Cocart}_d$ and the full subcategory of $\cat{Fib}_d^{(d)}$ defined by Lemma \ref{lem:d-th stage with d-cat fibers} are subcategories of $\Fun\big([1], \Cat^{\otimes d+1}\big)$. It therefore suffices 
to show that these two subcategories have the same objects and morphisms. We will do this by induction on $d$, the case $d=0$ being evident.

\subsubsection*{Objects}
Let $p\colon \cat{D}\rt \cat{C}$ be a map between $(d+1)$-categories. By inductive hypothesis and Lemma \ref{lem:local cartesian}, $p$ satisfies condition \ref{it:loccartobjects} if and only if $p$ is a homwise $(d-1)$-cartesian fibration. Assuming that $p$ satisfies these equivalent conditions, it remains to show that $p$ satisfies condition \ref{it:precocartobjects} if and only if it is a pre-cocartesian fibration. To see this, consider the base change 
$$\begin{tikzcd}
p'\colon \cat{D}'=\cat{D}\times_{\cat{C}}\core_1(\cat{C})\arrow[r] &  \core_1(\cat{C}).
\end{tikzcd}$$ 
For any $1$-morphism $\alpha\colon d_0\rt d_1$ and object $d_2$ in $\cat{D}'$ (or equivalently, in $\cat{D}$), we have
$$\begin{tikzcd}[column sep=0.05pc, row sep=0.8pc]
\Map_{\cat{D}'}(d_1, d_2)\arrow[rr, "\alpha^*"]\arrow[rd]\arrow[dd, "p'"{swap}] & & \Map_{\cat{D}'}(d_0, d_2)\arrow[dd, dotted]\arrow[rd]\\
& \Map_{\cat{D}}(d_1, d_2)\arrow[rr]\arrow[dd] & & \Map_{\cat{D}}(d_0, d_2)\arrow[dd, "p"]\\
\Map_{\core_1(\cat{C})}(pd_1, pd_2)\arrow[rr, dotted]\arrow[rd] & &  \Map_{\core_1(\cat{C})}(pd_0, pd_2)\arrow[rd, dotted] \\
& \Map_{\cat{C}}(pd_1, pd_2)\arrow[rr, "p(\alpha)^*"{swap}] &  & \Map_{\cat{C}}(pd_0, pd_2).
\end{tikzcd}$$
The left and right face of the cube are cartesian. By our assumption on $p$, the vertical arrows are $(d-1)$-cartesian fibrations and the front and back face determine strong maps of $(d-1)$-cartesian fibrations. By inductive hypothesis, this means equivalently that the vertical arrows are $(d-1)$-st stage fibrations between $d$-categories and that the front and back face determine strong morphisms between $(d-1)$-st stage fibrations. Since the back and front faces are equivalent at the level of objects, Corollary \ref{cor:pullback fiberwise} then implies that the front face is cartesian if and only if the back face is cartesian. By Corollary \ref{cor:classical def of pre-cocart}, this means that $\alpha$ is $p'$-cocartesian if and only if it is $p$-cocartesian. It follows that $p$ is a pre-cocartesian fibration if and only if $p'$ is a pre-cocartesian fibration, which in turn is equivalent to condition \ref{it:precocartobjects}, by Remark \ref{rem:precocart reform}.

\subsubsection*{Morphisms}
A commuting square determines a strong morphism between $d$-cocartesian fibrations if and only if the induced square on $d$-categories of morphisms induces a strong morphism between $(d-1)$-cartesian fibrations and the map $\core_1(\cat{D})\rt \core_1(\cat{D}')$ preserves cocartesian $1$-morphisms. Comparing this with condition \ref{it:loccartmorphisms} and \ref{it:precocartmorphisms} shows inductively that such strong morphisms coincide with strong morphisms of $d$-th stage fibrations.
\end{proof}

\begin{proof}[Proof (of Theorem \ref{thm:unstraightening})]
Composing the equivalences from Proposition \ref{prop:k-fibs to k+1-fibs} yields
\begin{equation}\label{eq:big unstraightening}\begin{tikzcd}
\cat{Fib}^{(0)}\arrow[r, "\sim"] & \cat{Fib}^{(d)}
\end{tikzcd}\end{equation}
that commutes with the codomain projection. Note that on both sides, the fiber over $\ast\in \Cat_{d+1}$ can be identified with the category $\Cat^{\otimes d}$ of $d$-fold categories: in $\cat{Fib}^{(0)}$, a $d$-fold category is viewed as a $(d+1)$-fold category which is constant in the first variable, and in $\cat{Fib}^{(d)}$ it is constant in the last variable. Since the above equivalence is given by repeatedly applying the functor $\Psi^\perp$, Proposition \ref{prop:why its called reflection} shows that the induced equivalence on fibers over $\ast\in \Cat_{d+1}$ simply sends a $d$-fold category to itself.

Now observe that there is a fully faithful inclusion $\SegCoPSh_{d}\hooklongrightarrow \cat{Fib}^{(0)}$, whose essential image consists of all those $0$-th stage fibrations $p\colon \dcat{D}\rt \cat{C}$ with the property that for any object $c\in \cat{C}$, the fiber $\dcat{D}_c$ is a $d$-category (rather than a $d$-fold category). The essential image of $\SegCoPSh_d$ under \eqref{eq:big unstraightening} then consists of $d$-th stage fibrations whose fibers are $d$-categories. By Proposition \ref{prop:cocartesian full inside d-fibration}, this means that \eqref{eq:big unstraightening} restricts to an equivalence $\SegCoPSh_{d}\rto{\sim} \cat{Cocart}_{d}$, as desired. The second part of Theorem \ref{thm:unstraightening} now follows directly from Corollary \ref{cor:left fibs are copresheaves}. 
\end{proof}

\subsection{Straightening and unstraightening: uniqueness and further properties}\label{sec:rigidity}
The equivalence of Theorem \ref{thm:unstraightening} arises from an iteration of explicit combinatorial maneuvers. The purpose of this intermezzo is to argue that at the level of objects, the precise construction of this equivalence is essentially irrelevant:
\begin{proposition}\label{prop:unique unstraightening}
There is a contractible space of equivalences of right fibrations
$$\begin{tikzcd}
\cat{coPSh}_d^{\mm{Seg}, \cart}\arrow[rd]\arrow[rr, dotted, "\sim"] & & \Cocart_d^{\cart}\arrow[ld]\\
& \cat{Cat}_{d+1}.
\end{tikzcd}$$
\end{proposition}
Combining this with Proposition \ref{prop:space representability} (in the large setting) and rephrasing everything in terms of presheaves rather than right fibrations (using straightening), we obtain:
\begin{corollary}\label{cor:naturality space level}
For any small $(d+1)$-category $\cat{C}$, there is a \emph{unique $1$-natural} equivalence 
$$\begin{tikzcd}
\mm{Un}\colon \core_0\Fun_{d+1}\big(\cat{C}, \hcat{Cat}_d\big)\arrow[r, yshift=1ex] & \core_0\Cocart_d\big(\cat{C}\big)\colon \mm{St}\arrow[l, yshift=-1ex, "\sim"{swap}]
\end{tikzcd}$$
between the \emph{spaces} of copresheaves $\cat{C}\rt \hcat{Cat}_d$ and of $d$-cocartesian fibrations $\cat{D}\rt \cat{C}$. We will refer to these functors as \emph{unstraightening} and \emph{straightening}.
\end{corollary}
Remark \ref{rem:cart case} also gives a unique $1$-natural equivalence $\core_0\Fun_{d+1}(\cat{C}^{\op}, \hcat{Cat}_d)\simeq \core_0\cat{Cart}_d(\cat{C})$.
\begin{remark}
In particular, the straightening of a $d$-cocartesian fibration $p\colon \cat{D}\rt \cat{C}$ between $(d+1)$-categories is naturally equivalent to the straightening of $p$, viewed as a special kind of $(d+1)$-cocartesian fibration between $(d+2)$-categories.
\end{remark}
We will deduce Proposition \ref{prop:unique unstraightening} from the following (simpler) version of a result of Barwick and Schommer-Pries \cite{bar21}:
\begin{proposition}\label{prop:rigidity of d-cat}
Let $\hcat{CAT}_{d+1}$ be the (very large) $(d+2)$-category of large $(d+1)$-categories and let $\Aut_{\hcat{CAT}_{d+1}}(\hcat{Cat}_d)\subseteq \Fun_{d+1}(\hcat{Cat}_d, \hcat{Cat}_d)$ be the full sub-$(d+1)$-category on those $(d+1)$-functors that are equivalences. Then $\Aut_{\hcat{CAT}_{d+1}}(\hcat{Cat}_d)\simeq \ast$.
\end{proposition}
\begin{proof}
Let $\phi\colon \hcat{Cat}_d\rt \hcat{Cat}_d$ be an equivalence of $(d+1)$-categories. Since $\phi$ preserves the terminal object, we can pick a (unique) element $u\in \phi(\ast)$. By the enriched Yoneda lemma \cite[6.2.7]{hin20} (or Example \ref{ex:Yoneda lemma}), $u$ induces a natural transformation of $(d+1)$-functors $u\colon \Map_{\hcat{Cat}_d}(\ast, -)\rt \phi(-)$. Note that the domain can be identified with the identity on $\hcat{Cat}_d$, so that $u$ is in particular a natural transformation between two automorphisms of the $1$-category $\Cat_d$. By \cite[Theorem 10.1]{bar21}, $u$ is a natural equivalence.

We conclude that $\Aut_{\hcat{CAT}_{d+1}}(\hcat{Cat}_d)$ has an essentially unique object, given by $\mm{id}_{\hcat{Cat}_d}\simeq \Map_{\hcat{Cat}_d}(\ast, -)$. Using the enriched Yoneda lemma once more, the endomorphism $d$-category of this object can be identified with $\Map_{\hcat{Cat}_d}(\ast, \ast)\simeq \ast$, which concludes the proof.
\end{proof}
\begin{proof}[Proof of Proposition \ref{prop:unique unstraightening}]
By Theorem \ref{thm:unstraightening}, the space of such equivalences is a nonempty torsor over the automorphism group $G$ of the right fibration $\cat{coPSh}^{\mm{Seg}, \cart}_d\rt \cat{Cat}_{d+1}$. Note that this right fibration fits into a pullback square
$$\begin{tikzcd}
\cat{coPSh}^{\mm{Seg}, \cart}_d\arrow[r, hook]\arrow[d, "\pi"{swap}] & \cat{coPSH}^{\mm{Seg}, \cart}_d\arrow[d, "\pi'"]\\
\cat{Cat}_{d+1}\arrow[r, hook] & \cat{CAT}_{d+1}
\end{tikzcd}$$
where $\cat{coPSH}^{\mm{Seg}, \cart}_d$ is the category of Segal copresheaves $X\rt \cat{C}$ over large $(d+1)$-categories, whose fibers are essentially small.

Applying Proposition \ref{prop:space representability} in the large setting (with $\kappa=\kappa_{\mm{sm}}$ the supremum of all small cardinals), the right fibration $\pi'$ corresponds under (some $1$-natural version of) unstraightening to the presheaf $F\colon \cat{CAT}_{d+1}^{\op}\rt \sS^{\mm{big}}$ representable by the $(d+1)$-category $\hcat{Cat}_d$. We then obtain equivalences of spaces of automorphisms
$$\begin{tikzcd}
\mm{Aut}_{\cat{CAT}_{d+1}}(\hcat{Cat}_d)\arrow[r, "\sim"] & \mm{Aut}(F)\arrow[r, "\sim"] & \mm{Aut}(F\big|\Cat_{d+1})\simeq G.
\end{tikzcd}$$
Here the first equivalence is the Yoneda embedding \eqref{diag:yoneda} and the last equivalence is unstraightening. The middle equivalence uses that $\CAT_{d+1}$ is a (large) ind-completion of $\Cat_{d+1}$; in particular, restriction along $\Cat_{d+1}\rt \CAT_{d+1}$ is fully faithful on presheaves sending large $\kappa_\mm{sm}$-filtered colimits in $\CAT_{d+1}$ to large limits in $\sS^{\mm{big}}$ (e.g.\ on representable presheaves). Finally, Proposition \ref{prop:rigidity of d-cat} implies that the left space is contractible.
\end{proof}
Theorem \ref{thm:rectification} and Theorem \ref{thm:unstraightening} imply that the (un)straightening equivalence of Corollary \ref{cor:naturality space level} admits some extension to an equivalence of $1$-categories, rather than spaces. Using this, let us record some properties of the (un)straightening at the level of objects.
\begin{example}\label{ex:straightening fibers}
The proof of Theorem \ref{thm:unstraightening} shows that over $\cat{C}=\ast$, the (un)straightening equivalence of Corollary \ref{cor:naturality space level} is homotopic to the identity on $\Cat_d$. Since (un)straightening intertwines restriction of copresheaves with base change of $d$-cocartesian fibrations, it follows that the straightening of a cocartesian fibration $p\colon \cat{D}\rt \cat{C}$ is given pointwise by $\mm{St}(p)_c\simeq \cat{D}_c$.
\end{example}
\begin{example}\label{ex:straightening morphisms}
Let $p\colon \cat{D}\rt \cat{C}$ be a $d$-cocartesian fibration and let $\gamma\colon c_0\rt c_1$ be a $1$-morphism in $\cat{C}$. The straightening of $p$ determines a functor $\gamma_*\colon \cat{D}_{c_0}\rt \cat{D}_{c_1}$ that can be understood explicitly as follows. For any diagram $\sigma_0\colon \cat{K}\rt \cat{D}_{c_0}$, there exists a unique extension of $\sigma_0$ to a map $\sigma\colon [1]\times \cat{K}\rt \cat{D}$ that covers $\gamma\colon [1]\rt \cat{C}$ and preserves cocartesian arrows. The restriction $\sigma_1\colon \{1\}\times \cat{K}\rt \cat{D}_1$ to the fiber over $1$ then presents the composite diagram $\gamma_*\circ \sigma_0$. In particular, applying this when $\sigma_0=\mm{id}_{\cat{D}_{c_0}}$ gives a description of $\gamma_*$.

To see this, it suffices to treat the case where $\cat{C}=[1]$. By naturality with respect to the projection $[1]\rt \ast$, the unstraightening of a constant copresheaf $\Delta_{\cat{K}}\colon [1]\rt \hcat{Cat}_d$ with value $\cat{K}$ is given by the projection $\pi_1\colon [1]\times \cat{K}\rt [1]$. Theorem \ref{thm:unstraightening} then shows that we have a commuting diagram of mapping spaces
$$\begin{tikzcd}[column sep=0.5pc, row sep=0.8pc]
 & \Map_{\Fun([1], \Cat_d)}(\Delta_{\cat{K}}, \mm{St}(p))\arrow[dd, "\sim"]\arrow[rd, "\mm{ev}_1"]\arrow[ld, "\mm{ev}_0"{swap}, "\sim"] & \\
\Map_{\Cat_d}(\cat{K}, \cat{D}_0) & & \Map_{\Cat_d}(\cat{K}, \cat{D}_1)\\
& \Map_{\cat{Cocart}_d([1])}(\pi_1, p)\arrow[lu, "0^*", "\sim"{swap}]\arrow[ru, "1^*"{swap}].
\end{tikzcd}$$
Now note that the top zigzag is homotopic to postcomposition with $\gamma_*$, while the bottom zigzag corresponds precisely to the procedure described above.
\end{example}
\begin{example}\label{ex:straightening 2-morphism}
Let $C_2$ be the $2$-cell and let $[2]^\mm{lax}$ denote the $2$-category
$$\begin{tikzcd}[column sep=1.5pc, row sep=1pc]
0\arrow[r]\arrow[rd, bend right, ""{name=t}] & 1\arrow[d]\\ & 2.\arrow[Rightarrow, from=1-2, to=t]
\end{tikzcd}$$
The projection $p\colon [2]^\mm{lax}\rt C_2$ sending $0\mapsto 0$ and $1, 2\mapsto 1$ is a $1$-cocartesian fibration. By Examples \ref{ex:straightening fibers} and \ref{ex:straightening morphisms}, the straightening $\mm{St}(p)\colon C_2\rt \hcat{Cat}_1$ is given by the diagram
$$\begin{tikzcd}[column sep=3pc]
\ast\arrow[r, bend left, "{0}"{above}, ""{swap, name=s}]\arrow[r, bend right, "{1}"{below}, ""{name=t}] & {[1]}\arrow[Rightarrow, from=s, to=t]
\end{tikzcd}$$
where the natural transformation is the only possible one. In other words, the straightening $\mm{St}(p)=\allowbreak\Map_{C_2}(0, -)$ is representable by $0$.
Likewise, let $\cat{K}$ be a $d$-category and consider the diagram $F_{\cat{K}}\colon C_2\rt \hcat{Cat}_d$ encoding the evident natural transformation from $(0, \mm{id})$ to $(1, \mm{id})\colon \cat{K}\rt [1]\times \cat{K}$. Then $\mm{Un}(F_{\cat{K}})\simeq [2]^\mm{lax}\times \cat{K}$. 

Now let $q\colon \cat{D}\rt C_2$ be a general $d$-cocartesian fibration, so that $\mm{St}(q)$ describes a natural transformation $\mu\colon f_0\rt f_1$ between two functors $f_0, f_1\colon \cat{D}_0\rt \cat{D}_1$. As in Example \ref{ex:straightening morphisms}, one can obtain $\mu$ as follows: there is a unique map $[2]^\mm{lax}\times \cat{D}_0\rt \cat{D}$ over $C_2$ restricting to the identity on $\{0\}\times \cat{D}_0$, and $\mu$ is given by its restriction $\{1\leq 2\}\times \cat{D}_0\rt \cat{D}_1$.
\end{example}
\begin{example}
Given a $2$-category $\cat{C}$, let us denote by $\mm{Tw}_{\mm{lax}}(\cat{C})\rt \cat{C}^{\op}\times \cat{C}$ the $1$-cocartesian fibration arising as the unstraightening of $\Map_{\cat{C}}\colon \cat{C}^{\op}\times\cat{C}\rt \hcat{Cat}_1$. We will refer to $\mm{Tw}_{\mm{lax}}(\cat{C})$ as the \emph{lax twisted arrow $2$-category} of $\cat{C}$.

Explicitly, $\mm{Tw}_{\mm{lax}}(\cat{C})$ is the image under the functor $\Psi^\perp$ of the Segal copresheaf $\mm{Tw}(\cat{C})\rt \cat{C}^{\op}\times \cat{C}$ (Example \ref{ex:mapping space functor}). Unraveling the definitions, one arrives at the following explicit description of $\mm{Tw}_{\mm{lax}}(\cat{C})$ as a bisimplicial space. For each $[m], [n]\in \Del$, consider the bisimplicial space (in fact, set)
$$
J[m, n]=\colim_{\substack{\alpha\star \beta\colon [q]\star [p]\to [m]\\ \alpha'\colon [q]\to [n]}} \big([p]^{\op}\star [p]\big)\boxtimes [q].
$$
This comes with natural inclusions $[m]^{\op}\boxtimes [n]\rt J[m, n]\lt [m]\boxtimes [n]$, arising as parts of the bisimplex associated to $(\alpha, \alpha', \beta)=(\mm{cst}(0), \mm{id}_{[m]}, \mm{id}_{[n]})$. One then has that
$$
\mm{Tw}_{\mm{lax}}(\cat{C})(m, n)\subseteq \Map_{\Fun(\Del^{\times 2, \op}, \sS)}\big(J[m, n], \cat{C}\big)
$$
is the subspace of maps $J[m, n]\rt \cat{C}$ whose restriction to any bisimplex of the form $\big([0]^{\op}\star [0]\big)\boxtimes [q]\rt J[m, n]$ (i.e.\ a summand for which $p=0$) is degenerate in the $q$-direction.

Note that in the definition of $J[m, n]$, it suffices to take the colimit over the cofinal subcategory of $(\alpha, \alpha', \beta)$ such that $(\alpha, \alpha')\colon [q]\rt [m]\times [n]$ is injective and $\beta\colon [p]\rt \{\alpha(q)\leq \dots\leq m\big\}$ is the inclusion of an upwards closed subset. Using this for low values of $m$ and $n$, one sees that the $2$-category $\mm{Tw}_{\mm{lax}}(\cat{C})$ has objects given by morphisms in $\cat{C}$, and $1$-and $2$-morphisms can be depicted by diagrams in $\cat{C}$ of the form
$$\begin{tikzcd}
c_0\arrow[r, ""{below, name=s1}] & c_1\arrow[d]\\
d_0\arrow[u] \arrow[r, ""{above, name=t1}] & d_1 \arrow[Rightarrow, from=s1, to=t1, shorten=1ex]
\end{tikzcd}
\qquad\qquad\qquad \qquad
\begin{tikzcd}[column sep=0pc]
c_0\arrow[rrrr, ""{below, name=s1}] & & \hspace{25pt}& &  c_1\arrow[rd, ""{swap, name=t2}, bend left=35]\arrow[rd, ""{name=s2}, bend right=25, color={\vertcolor}]\arrow[Rightarrow, from=s2, to=t2, shorten=-0.4ex, start anchor={[yshift=1.5pt]}, end anchor={[yshift=-0.5pt]}, color=blue]\\
& d_0\arrow[lu, ""{swap, name=s3}, bend left=25, color={\vertcolor}]\arrow[lu, ""{name=t3}, bend right=35]\arrow[Rightarrow, from=s3, to=t3, shorten=-0.4ex,  start anchor={[yshift=0.5pt]}, end anchor={[yshift=-1.5pt]}, color=blue] \arrow[rrrr, ""{above, name=t1}] & & & &  d_1 
\arrow[Rightarrow, from=s1, to=t1, shorten=0.5ex, bend left=40, start anchor={[xshift=0ex, yshift=-0.3ex]}, end anchor={[xshift=0ex, yshift=1ex]}]\arrow[Rightarrow, from=s1, to=t1, shorten=1ex, bend right=50, color={\vertcolor}, start anchor={[xshift=-0.5ex]}, end anchor={[xshift=0ex, yshift=1ex]}].
\end{tikzcd}
$$
Here the left diagram defines an arrow from $c_0\rt c_1$ to $d_0\rt d_1$ and the right diagram defines a (blue) $2$-morphism from the red to the black morphism. The projection to $\cat{C}^{\op}\times \cat{C}$ restricts to the left and right column. 

By restricting to $\cat{C}^{\op}\times \{c\}$ or $\{c\}\times \cat{C}$, we obtain an explicit description of the unstraightening of the representable presheaf $\Map_{\cat{C}}(-, c)$ and copresheaf $\Map_{\cat{C}}(c, -)$ in terms of lax versions of the over- and under-category of $c$.
\end{example}
\begin{lemma}\label{lem:straightening and cores}
Let $p\colon \cat{D}\rt \cat{C}$ be a $d$-cocartesian fibration between $(d+1)$-categories. Then the following assertions hold:
\begin{enumerate}[label=(\arabic*)]
\item\label{it:1-core of cocart} The induced map $p'\colon \core_1\cat{D}\rt \core_1\cat{C}$ is a $1$-cocartesian fibration, whose straightening $\mm{St}(p')$ can be identified with the $1$-functor
$$\begin{tikzcd}[column sep=2.5pc]
\mm{St}(p')\colon \core_1\cat{C}\arrow[r, "\mm{St}(p)"] & \Cat_d\arrow[r, "\mm{core}_1"] & \Cat_1.
\end{tikzcd}$$
\item\label{it:underlying 0-cocart} The restriction to cocartesian arrows $p''\colon (\core_1\cat{D})^{\mm{cocart}}\rt \core_1\cat{C}$ is a $0$-cocartesian fibration, whose straightening $\mm{St}(p'')$ can be identified with the $1$-functor
$$\begin{tikzcd}[column sep=2.5pc]
\mm{St}(p'')\colon \core_1\cat{C}\arrow[r, "\mm{St}(p)"] & \Cat_d\arrow[r, "\mm{core}_0"] & \sS.
\end{tikzcd}$$
\end{enumerate}
\end{lemma}
\begin{proof}
By replacing $p$ by its base change $\cat{D}\times_{\cat{C}}\core_1\cat{C}\rt \core_1\cat{C}$, we can reduce to the case where $\cat{C}$ is a $1$-category. In this case, Lemma \ref{lem:d-cocart as d+1-cocart} implies that a $d$-cocartesian fibration $q\colon \cat{E}\rt \cat{C}$ is a $1$-cocartesian fibration if and only if $\cat{E}$ is a $1$-category. Consequently, $p'$ is the terminal $1$-cocartesian fibration over $\cat{C}$ equipped with a natural map to $p$. Under straightening, this corresponds to the terminal $\Cat_1$-valued copresheaf on $\cat{C}$ equipped with a natural transformation to $\mm{St}(p)$. This is precisely the composite functor given in \ref{it:1-core of cocart}.

Likewise, Lemma \ref{lem:d-cocart as d+1-cocart} implies $p''$ is the terminal $0$-cocartesian fibration with a map to $p'$. Under straightening, this corresponds to the terminal space-valued presheaf with a natural map to $\mm{St}(p')$, which is precisely the functor given in \ref{it:underlying 0-cocart}. 
\end{proof}

\subsection{Straightening and unstraightening: higher functoriality}
Theorem \ref{thm:unstraightening} shows that for every $(d+1)$-category $\cat{C}$, there is an equivalence between the $1$-category of $d$-cocartesian fibrations over $\cat{C}$ and the $1$-category underlying $\Fun_{d+1}(\cat{C}, \hcat{Cat}_d)$. 
We will now refine this equivalence to a natural equivalence of $(d+2)$-functors with values in $(d+1)$-categories, using Theorem \ref{thm:unstraightening} one categorical dimension higher. 
\begin{definition}\label{def:cocart fib functor}
Consider the $(d+1)$-cartesian fibration $\pi\colon \hcat{Cocart}_d\rt \hcat{Cat}_{d+1}$ from Proposition \ref{prop:cocart forms cart fibration}. We will write
$$\begin{tikzcd}
\hcat{Cocart}_d(-)\colon \hcat{Cat}_{d+1}^{\op}\arrow[r] & \hcat{CAT}_{d+1}
\end{tikzcd}$$
for the $(d+2)$-functor corresponding to $\pi$ under the (unique) $1$-natural straightening equivalence of Corollary \ref{cor:naturality space level} (in the large setting).
\end{definition}
We now come to our main result:
\begin{theorem}[Straightening and unstraightening]\label{thm:unstraightening very functorial}
The $(d+1)$-category of natural equivalences of $(d+2)$-functors $\hcat{Cat}_{d+1}^{\op}\rt \hcat{CAT}_{d+1}$
$$\begin{tikzcd}
\mm{St}\colon \hcat{Cocart}_d(-)\arrow[r, yshift=1ex] & \Fun_{d+1}\big(-, \hcat{Cat}_d\big)\colon \mm{Un}\arrow[l, yshift=-1ex, "\sim"{swap}]
\end{tikzcd}$$
is equivalent to a single point.
\end{theorem}
Since $\hcat{Cat}_d$ is itself a large $(d+1)$-category, it will be convenient to first study a version of the above result where we bound the cardinality.
\begin{notation}\label{not:kappa-small cocart}
Let $\kappa$ be a regular uncountable cardinal. Consider the $(d+2)$-functor $\pi\colon \hcat{Cocart}_d(\kappa)\rt \hcat{Cat}_{d+1}$ whose domain is the full subcategory of $\hcat{Cocart}_d$ on the $d$-cocartesian fibrations with $\kappa$-small fibers. This is a $(d+1)$-cartesian fibration, and we will write $\hcat{Cocart}_d^\kappa\colon \hcat{Cat}_{d+1}^{op}\rt \hcat{Cat}_{d+1}$ for its image under the (natural) straightening of Corollary \ref{cor:naturality space level}. Note that the values of this functor are indeed small categories: Corollary \ref{cor:naturality space level} identifies their spaces of objects with the spaces of enriched functors $\cat{C}\rt \hcat{Cat}_d(\kappa)$.
\end{notation}
Let us start by recording some properties of the $(d+2)$-functor $\hcat{Cocart}^\kappa_d$.
\begin{lemma}\label{lem:base change is cotensored}
Let $f\colon \cat{C}\rt \cat{C}'$ be a functor of $(d+1)$-categories. Then the $(d+1)$-functor 
$$\begin{tikzcd}
f^*=\hcat{Cocart}^\kappa_d(f)\colon \hcat{Cocart}^\kappa_d(\cat{C}')\arrow[r] & \hcat{Cocart}^\kappa_d(\cat{C})
\end{tikzcd}$$
is tensored over $\hcat{Cat}_{d}(\kappa)$ (in particular, so are its domain and codomain).
\end{lemma}
Heuristically, this simply comes down to the fact that the tensoring of $\hcat{Cocart}^\kappa_d(\cat{C})$ is given by $K\otimes (\cat{D}\to \cat{C}) = (K\times \cat{D}\rt \cat{C})$ and that such products are stable under base change. However, since we defined the $(d+1)$-functor $f^*$ somewhat abstractly, in terms the straightening from Corollary \ref{cor:naturality space level}, we will provide a more detailed argument using point-set models.
\begin{proof}
The $(d+2)$-functor $\pi\colon \hcat{Cocart}_d(\kappa)\rt \hcat{Cat}_{d+1}$ can be modeled by a strict functor of $\CompSegS_{d+1}$-enriched categories $\ul{\pi}\colon \modcat{Cocart}_d(\kappa)\rt \CompSegS_{d+1}$. Here $\modcat{Cocart}_d(\kappa)$ is the enriched subcategory of $\Fun\big([1], \CompSegS_d\big)$, whose objects are fibrations in $\CompSegS_{d+1}$ that present $d$-cocartesian fibrations with $\kappa$-small fibers and with 
$$
\Map_{\modcat{Cocart}_d(\kappa)}(p, q)\subseteq \Map_{\Fun([1], \CompSegS_d)}(p, q)
$$
the maximal subobject in $\CompSegS_{d+1}$ that is a model for the full sub-$d$-category of strong morphisms. Example \ref{ex:straightening fibers} identifies the values of $\hcat{Cocart}^\kappa_d(-)$ with the fibers of $\pi$. Since $\ul{\pi}$ is a fibration in the model category of $\CompSegS_{d+1}$-enriched categories, the fibers of $\pi$ can be modeled by the strict fibers of $\ul{\pi}$. Furthermore, as we already know that the fibers $\hcat{Cocart}^\kappa_d(\cat{C})$ are $(d+1)$-categories (instead of $(d+2)$-categories), we can model each $\hcat{Cocart}^\kappa_d(-)$ by the $\CompSegS_d$-enriched category underlying $\modcat{Cocart}^\kappa_d(\cat{C})=\ul{\pi}^{-1}(\cat{C})$.

As a $\CompSegS_d$-enriched category (rather than a $\CompSegS_{d+1}$-enriched category), this fiber is strictly tensored over $\CompSegS_d(\kappa)$: the tensoring of $p\colon \cat{D}\rt \cat{C}$ with $K\in \CompSegS_d(\kappa)$ is simply given by $\cat{D}\times K\rt \cat{C}$. Note that for any map $f\colon \cat{C}\rt \cat{C}'$ in $\CompSegS_{d+1}$, taking (strict) base change along $f$ defines an enriched functor $f^*\colon \modcat{Cocart}^\kappa_d(\cat{C}')\rt \modcat{Cocart}^\kappa_d(\cat{C})$ which strictly preserves the tensoring over $\CompSegS_d(\kappa)$. It therefore remains to verify that this $f^*$ indeed presents the value of $\hcat{Cocart}^\kappa_d(-)$ on $f$.

To see this, note that associated to $f$ is a strict diagram of $\CompSegS_{d+1}$-enriched categories
$$\begin{tikzcd}
\modcat{Cocart}^\kappa_d(\cat{C}')\times [1]\arrow[d]\arrow[r, "\phi"] & \modcat{Cocart}_d(\kappa)\arrow[d]\\
{[1]}\arrow[r, "f"{swap}] & \CompSegS_{d+1}.
\end{tikzcd}$$
Here $\phi$ is given on $\modcat{Cocart}^\kappa_d(\cat{C}')\times \{1\}$ by the fiber inclusion, and on $\modcat{Cocart}^\kappa_d(\cat{C}')\times \{0\}$ by the strict base change functor $f^*$ (which comes with an evident natural transformation to the fiber inclusion). One easily sees that this enriched functor $\phi$ models a map of $(d+1)$-cartesian fibrations (i.e.\ it preserves cartesian $1$-morphisms). Example \ref{ex:straightening morphisms} then shows that the strict functor $f^*$ indeed presents (up to homotopy) the value of the straightening $\hcat{Cocart}^\kappa_d(-)$ on the arrow $f$, so that the latter is tensored as well.
\end{proof}
\begin{lemma}\label{lem:cocart fib preserves limits}
The underlying functor between $1$-categories $\hcat{Cocart}^\kappa_d\colon \cat{Cat}_{d+1}^{\op}\rt \cat{Cat}_{d+1}$ preserves all small limits.
\end{lemma}
\begin{proof}
By Lemma \ref{lem:base change is cotensored}, the functor of $1$-categories $\hcat{Cocart}^\kappa_d\colon \cat{Cat}_{d+1}^{\op}\rt \cat{Cat}_{d+1}$ takes values in $(d+1)$-categories and functors which are tensored over $\hcat{Cat}_d(\kappa)$. Using Lemma \ref{lem:tensored equivalence}, Remark \ref{rem:less than full tensoring} and the fact that taking cores preserves limits, it then suffices to verify that the composite
$$\begin{tikzcd}[column sep=3.2pc]
F\colon \cat{Cat}_{d+1}^{\op}\arrow[r, "\hcat{Cocart}^\kappa_d"] & \cat{Cat}_{d+1}\arrow[r, "\core_1"] & \cat{Cat}_1
\end{tikzcd}$$
preserves limits. To see this it suffices to verify that for $i=0, 1$, 
\begin{equation}\label{eq:map into copresheaf functor}\begin{tikzcd}
\Map_{\cat{Cat}_1}\big([i], F(-)\big)\colon \cat{Cat}_{d+1}^{\op}\arrow[r] & \sS
\end{tikzcd}\end{equation}
preserves limits. By Lemma \ref{lem:straightening and cores}, the functor $F$ arises as the unstraightening of the cartesian fibration of $1$-categories $\pi\colon \cat{Cocart}(\kappa)_d\rt \CAT_{d+1}$. Theorem \ref{thm:unstraightening} implies that this can be identified with $\SegCoPSh(\kappa)\rt \cat{Cat}_{d+1}$, where the domain consists of Segal copresheaves whose fibers are $\kappa$-small. 

\sloppy The two functors \eqref{eq:map into copresheaf functor} therefore arise as the straightening of the right fibrations \mbox{$\SegCoPSh_d(\kappa)^{\cart}\rt \cat{Cat}_{d+1}$} and $\cat{X}^\mm{cart}\rt \cat{Cat}_{d+1}$ from Proposition \ref{prop:space representability} and Variant \ref{var:maps of presheaves} respectively. It was shown there that the straightenings of these two right fibrations are representable, by $\hcat{Cat}_d(\kappa)$ and $\Fun_{d+1}\big([1], \hcat{Cat}_d(\kappa)\big)$ respectively. In particular, they preserve all limits, as required.
\end{proof}
\begin{theorem}\label{thm:cocart representable}
The $(d+2)$-functor $\hcat{Cocart}_d^\kappa\colon \hcat{Cat}_{d+1}^{\op}\rt \hcat{Cat}_{d+1}$ from Notation \ref{not:kappa-small cocart} is representable by $\hcat{Cat}_d(\kappa)$.
\end{theorem}
\begin{proof}
\sloppy By Proposition \ref{prop:representable from cotensored}, it suffices to verify that the functor of $(d+2)$-categories \mbox{$\hcat{Cocart}^{\kappa}_d\colon \hcat{Cat}_{d+1}^{\op}\rt \hcat{Cat}_{d+1}$} is cotensored over $\hcat{Cat}_{d+1}$ and that the functor of $1$-categories $\core_0\hcat{Cocart}^\kappa_d\colon \cat{Cat}_{d+1}^{\op}\rt \sS$ is representable. We have already seen in the proof of Lemma \ref{lem:cocart fib preserves limits} that this latter functor is representable by the $(d+1)$-category $\hcat{Cat}_d(\kappa)$.

It remains to verify that for any $K\in \cat{Cat}_{d+1}$ and $\cat{C}\in \cat{Cat}_{d+1}$, the comparison map 
$$\begin{tikzcd}
\mu\colon \hcat{Cocart}^\kappa_d(K\times \cat{C})\arrow[r] & \Fun_{d+1}\big(K, \hcat{Cocart}^\kappa_d(\cat{C})\big)
\end{tikzcd}$$
from Definition \ref{def:tensored functor} is an equivalence of $(d+1)$-categories. This will follow immediately from Lemma \ref{lem:tensored equivalence} once we show the following two claims:
\begin{enumerate}[label=(\alph*)]
\item\label{it:mu underlying equiv} $\mu$ induces an equivalence on the underlying $1$-categories.

\item\label{it:mu tensored} Both the domain and codomain of $\mu$ are tensored over $\hcat{Cat}_d(\kappa)$ and $\mu$ is tensored over $\hcat{Cat}_d(\kappa)$.
\end{enumerate}
Let us start with assertion \ref{it:mu underlying equiv}. As a first reduction, note that $\mu$ depends in particular $1$-functorially on $K$ and $\cat{C}$. Since both its domain and codomain send colimits in $K$ and $\cat{C}$ to limits (Lemma \ref{lem:cocart fib preserves limits} and Corollary \ref{cor:internal hom}), it suffices to verify that $\mu$ is an equivalence on underlying $1$-categories for generating objects $K$ and $\cat{C}$. We can therefore reduce to the case where $K$ and $\cat{C}$ are contained in $\Theta_{d+1}$. In particular, we may assume that $K$ is $\kappa$-small and that for any object $k\in K$, $K(k, k)\simeq \ast$, and likewise for $\cat{C}$.

By Theorem \ref{thm:rectification}, Theorem \ref{thm:unstraightening} and Lemma \ref{lem:straightening and cores}, the $1$-cartesian fibration between $1$-categories $\pi\colon \core_1\hcat{Cocart}_d(\kappa)\rt \core_1\hcat{Cat}_{d+1}$ arises as the localization of the homotopy cartesian fibration $\EnrFun_d(\kappa)\rt \Cat(\CompSegS_{d})$, where $\EnrFun_d(\kappa)$ consists of enriched categories $\scat{C}$ together with an enriched functor $\scat{C}\rt \CompSegS_d$ whose values are homotopically $\kappa$-small, as in the proof of Proposition \ref{prop:space representability}. In particular, this implies that there are equivalences 
\begin{align*}
\core_1\hcat{Cocart}_d^\kappa(K\times \cat{C})&\simeq \core_1\Fun_{d+1}(K\times \cat{C}, \hcat{Cat}_{d}(\kappa))\\
\core_1\Fun_{d+1}\big(K, \hcat{Cocart}_d^\kappa(\cat{C})\big)&\simeq \core_1\Fun_{d+1}(K\times \cat{C}, \hcat{Cat}_{d}(\kappa))
\end{align*}
such that for any two maps $K\rt K'$ and $\cat{C}\rt \cat{C}'$, the functoriality on the left hand side is homotopic to restriction of functors on the right hand side (we will not need any information about composition or homotopy coherence).

Using these identifications, the map $\mu$ therefore determines an endofunctor that fits into a commuting diagram of $1$-categories
$$\begin{tikzcd}
\core_1\Fun_{d+1}(K\times \cat{C}, \hcat{Cat}_{d}(\kappa))\arrow[d, "r"{swap}]\arrow[r, "\mu"] & \core_1\Fun_{d+1}(K\times \cat{C}, \hcat{Cat}_{d}(\kappa))\arrow[d, "r"]\\
\prod_{k\in K, c\in \cat{C}} \Cat_{d}(\kappa)\arrow[r, "\prod \mu"] & \prod_{k\in K, c\in \cat{C}} \Cat_{d}(\kappa).
\end{tikzcd}$$
Here $r$ denotes the functor restricting to the sets of (isomorphism classes of) objects of $K$ and $\cat{C}$. Note that $\mu$ is (by construction) homotopic to the identity when applied to monoidal unit $K=\ast$. Up to homotopy, we may therefore assume that the bottom map in the square is the identity.

The restriction functor $r$ preserves $\kappa$-small limits and colimits and detects equivalences. By our assumption that $K$ and $\cat{C}$ are $\kappa$-small, it has a left adjoint $r_!$ sending each $(k, c, \cat{A})$ to the representable functor $K(k, -)\times \cat{C}(c, -)\times \cat{A}$; this follows directly from the point-set presentation of $r$ in terms of enriched functor categories. The above square then shows that $\mu\colon \core_1\Fun_{d+1}(K\times \cat{C}, \hcat{Cat}_{d}(\kappa))\rt \core_1\Fun_{d+1}(K\times \cat{C}, \hcat{Cat}_{d}(\kappa))$ preserves $\kappa$-small limits and colimits and detects equivalences as well, and admits a left adjoint $\mu_!$. It suffices to show that for any $G\colon K\times \cat{C}\rt \hcat{Cat}_{d}(\kappa)$, the unit map $\eta\colon G\rt \mu \mu_!(G)$ is an equivalence. Since $\mu$ and $\mu_!$ both preserve $\kappa$-small colimits, it suffices to verify this when $G=r_!(k, c, \cat{A})$ is left Kan extended from an object. By adjunction, $\eta$ is then determined by a map
$$
\phi_{k, c} \colon \cat{A}\rt K(k, k)\times \cat{C}(c, c)\times \cat{A}
$$
depending ($1$-)naturally on $\cat{A}$. Since $K$ and $\cat{C}$ were assumed to have trivial endomorphisms, we therefore obtain an endomorphism $\phi_{k, c}\colon \mm{id}\rt \mm{id}$ of the identity functor on $\cat{Cat}_{d}(\kappa)$. There is a contractible space of such natural endomorphisms by \cite[Theorem 10.1]{bar21}. 

The map $\phi_{k, c}$ is therefore naturally (in $\cat{A}$) homotopic to $\{\mm{id}_k\}\times \{\mm{id}_c\}\times\mm{id}_{\cat{A}}$. This means precisely that $\eta$ is homotopic to the identity, so in particular an equivalence. We conclude that $\mu_!$ and $\mu$ form an adjoint equivalence, so that $\mu$ does indeed induce an equivalence on the underlying $1$-categories, as asserted.

We now turn to assertion \ref{it:mu tensored}. Note that the construction of $\mu$ (Definition \ref{def:tensored functor}) implies that for any object $k\in K$, the composite with evaluation at $k$
$$\begin{tikzcd}
\hcat{Cocart}^\kappa_d(K\times \cat{C})\arrow[r, "\mu"] & \Fun_{d+1}\big(K, \hcat{Cocart}^\kappa_d(\cat{C})\big)\arrow[r, "\mm{ev}_k"] & \hcat{Cocart}^\kappa_d(\cat{C})
\end{tikzcd}$$
is equivalent to the value of $\hcat{Cocart}^\kappa_d(-)$ on  the inclusion $\{k\}\times \cat{C}\hookrightarrow K\times \cat{C}$. By Lemma \ref{lem:base change is cotensored}, this implies that the domain and codomain of each $\mm{ev}_k\circ \mu$ are tensored over $\hcat{Cat}_d(\kappa)$ and that $\mm{ev}_k\circ \mu$ preserves the tensoring. Claim \ref{it:mu tensored} will therefore follow once we know that $\Fun_{d+1}\big(K, \hcat{Cocart}^\kappa_d(\cat{C})\big)$ is tensored over $\hcat{Cat}_d(\kappa)$ and that evaluation at each $k\in K$ preserves the tensoring.

To see this, one can either use the formalism from \cite{hin20} or use the following trick. We will first prove assertion \ref{it:mu tensored} in the case where $\cat{C}=\ast$, so that the target of $\mu$ reduces to $\Fun_{d+1}(K, \hcat{Cat}_d(\kappa))$. By Example \ref{ex:tensoring copresheaves}, this is indeed tensored over $\hcat{Cat}_d(\kappa)$ and restriction along $\{k\}\rt K$ preserves the tensoring. Combined with \ref{it:mu underlying equiv}, it follows that $\mu$ yields an equivalence $\hcat{Cocart}^\kappa_d(K)\simeq \Fun_{d+1}(K, \hcat{Cat}_d^\kappa)$. Using this equivalence (with $K$ replaced by $\cat{C}$), the target of the map $\mu$ is then equivalent to
$$
\Fun_{d+1}\big(K, \hcat{Cocart}^\kappa_d(\cat{C})\big)\simeq \Fun_{d+1}\big(K, \Fun_{d+1}(\cat{C}, \hcat{Cat}_d(\kappa))\big)\simeq \Fun_{d+1}\big(K\times \cat{C}, \hcat{Cat}_d(\kappa)\big).
$$
Applying Example \ref{ex:tensoring copresheaves} once more then shows that this $(d+1)$-category is indeed tensored over $\hcat{Cat}_d(\kappa)$ and that evaluation at $k\in K$ preserves the tensoring. We conclude that the natural transformation $\mu$ satisfies both \ref{it:mu underlying equiv} and \ref{it:mu tensored}, so that it is a natural equivalence as desired.
\end{proof}
\begin{proof}[Proof of Theorem \ref{thm:unstraightening very functorial}]
Let us apply Theorem \ref{thm:cocart representable} in the large setting, with $\kappa=\kappa_\mm{sm}$ the supremum of all small cardinals. It follows that the $(d+2)$-functor on large $(d+1)$-categories
$$\begin{tikzcd}
\hcat{Cocart}_d^{\kappa_\mm{sm}}\colon \hcat{CAT}_{d+1}^{op}\arrow[r] & \hcat{CAT}_{d+1}
\end{tikzcd}$$
is representable by $\hcat{Cat}_d$. Note that the restriction of $\hcat{Cocart}_d^{\kappa_\mm{sm}}$ to $\hcat{Cat}_{d+1}\subseteq \hcat{CAT}_{d+1}$ is precisely the functor $\hcat{Cocart}_d$ from Theorem \ref{thm:unstraightening very functorial}. This proves the existence of a natural equivalence $\hcat{Cocart}_d(-)\simeq \Fun_{d+1}(-, \hcat{Cat}_d)$ between $(d+2)$-functors $\hcat{Cat}_{d+1}^{\op}\rt \hcat{CAT}_{d+1}$.

For uniqueness, let us write $F\colon \hcat{CAT}_{d+1}^{\op}\rt \hcat{CAT}$ for the $(d+2)$-functor represented by $\hcat{Cat}_d$. Let us then write
\begin{align*}
\cat{Iso}&\subseteq \Map_{\Fun_{d+2}(\hcat{Cat}^{\op}_{d+1}, \hcat{CAT}_{d+1})}\big(\hcat{Cocart}_d, F\big|\hcat{Cat}_{d+1}\big),\\
\cat{Aut}(F\big|\hcat{Cat}_{d+1})&\subseteq \Map_{\Fun_{d+2}(\hcat{Cat}^{\op}_{d+1}, \hcat{CAT}_{d+1})}\big(F\big|\hcat{Cat}_{d+1}, F\big|\hcat{Cat}_{d+1}\big)
\end{align*}
for the full sub-$(d+1)$-categories spanned by the natural equivalences between $(d+2)$-functors. We have to show that $\cat{Iso}\simeq \ast$. We have just verified that $\cat{Iso}$ is nonempty, and precomposing with any element $\phi\in \cat{Iso}$ determines an equivalence $\cat{Iso}\simeq \cat{Aut}(F\big|\hcat{Cat}_{d+1})$. It therefore remains to verify that the $(d+1)$-category of automorphisms of $F\big|\hcat{Cat}_{d+1}$ is equivalent to a point.
As in Proposition \ref{prop:unique unstraightening}, there are natural maps of $(d+1)$-categories
$$\begin{tikzcd}
\mm{Aut}_{\hcat{CAT}_{d+1}}(\hcat{Cat}_d)\arrow[r] &  \mm{Aut}(F)\arrow[r] & \mm{Aut}(F\big|\hcat{Cat}_{d+1}).
\end{tikzcd}$$
The first map is an equivalence by the enriched Yoneda lemma (Section \ref{sec:corep}). The second map arises by restricting the natural automorphisms of $F$ along $i\colon \hcat{Cat}^{\op}_{d+1}\hooklongrightarrow \hcat{CAT}^{\op}_{d+1}$. To see that this is an equivalence, note that together with right Kan extension, restriction along $i$ defines a coreflective localization of (very large) $(d+2)$-categories
$$\begin{tikzcd}
i^*\colon \Fun_{d+2}\big(\hcat{CAT}^{\op}_{d+1}, \hcat{CAT}_{d+1}\big)\arrow[r, yshift=1ex]\arrow[r, hookleftarrow, yshift=-1ex] & \Fun_{d+2}\big(\hcat{Cat}^{\op}_{d+1}, \hcat{CAT}_{d+1}\big)\colon i_*.
\end{tikzcd}$$
The essential image of $i_*$ consists of enriched functors $\hcat{CAT}^{\op}_{d+1}\rt \hcat{CAT}_{d+1}$ whose underlying $1$-functor preserves large, $\kappa_{\mm{sm}}$-cofiltered limits; this can be verified by presenting everything by enriched (model) categories of strictly enriched functors. In particular, the representable functor $F$ is contained in the essential image of $i_*$, so that restriction along $i$ induces an equivalence on its automorphism $(d+1)$-categories. 
The result now follows since $\mm{Aut}_{\hcat{CAT}_{d+1}}(\hcat{Cat}_d)\simeq \ast$ by Proposition \ref{prop:rigidity of d-cat}.
\end{proof}

\appendix
\section{Technicalities on categories of fibrant objects}\label{sec:fib obj}
The purpose of this appendix is to describe a version of the Grothendieck construction for categories of fibrant objects, in the following sense:
\begin{definition}
By a \emph{category of fibrant objects} $\cat{M}$ we will mean a $(1, 1)$-category with two classes of maps, called \emph{weak equivalences} and \emph{fibrations}, both containing the isomorphisms and closed under composition, such that the following conditions hold:
\begin{enumerate}
\item The weak equivalences have the 2-out-of-6 property.
\item $\cat{M}$ admits pullbacks along fibrations and the base change of an (acyclic) fibration is again an (acyclic) fibration. 
\item $\cat{M}$ admits a terminal object and each $x\rt \ast$ is a fibration.
\item Each map in $\cat{M}$ admits a factorization into a weak equivalence, followed by a fibration.
\end{enumerate}
A map between categories of fibrant objects $f\colon \cat{M}\rt \cat{N}$ is a functor preserving fibrations, weak equivalences, the terminal object and pullbacks along fibrations. Let us write $\cat{FibCat}$ for the $(2, 1)$-category of categories of fibrant objects.
\end{definition}
Our condition that weak equivalences have the 2-out-of-6 property is slightly nonstandard; it is equivalent to the assertion that a morphism in $\cat{M}$ is a weak equivalence if and only if its image in the localization $\cat{M}[W^{-1}]$ is an equivalence (see \cite[Theorem 7.2.7]{rad09} or use the description of $\cat{M}[W^{-1}]$ from Remark \ref{rem:localization fibcat}).
\begin{lemma}\label{lem:fibcat grothendieck}
Let $\cat{B}\in \cat{FibCat}$. Consider a (pseudo-)functor $\cat{B}^{\op}\rt \cat{FibCat};\, b\longmapsto \cat{E}(b)$ and denote by $p\colon \cat{E}\rt \cat{B}$ its (classical) Grothendieck construction. Suppose that the following two assertions hold:
\begin{enumerate}[label=(\alph*)]
\item\label{it:homotopy invariance} For every weak equivalence $f\colon b_0\rt b_1$, the functor $f^*\colon \cat{E}(b_1)\rt \cat{E}(b_0)$ induces an equivalence on localizations.

\item\label{it:factorization} Let $f\colon b_0\rt b_1$ be a morphism and $g\colon x\rt f^*y$ a morphism in $\cat{E}(b_0)$. Then there exists a factorization $f=pi$ with $p\colon b_0'\rt b_1$ a fibration and $i\colon b_0\rt b_0'$ a weak equivalence, together with a factorization $g=i^*(q)j$ with $q\colon z\twoheadrightarrow p^*y$ a fibration in $\cat{E}(b'_0)$ and $j\colon x\rt i^*z$ a weak equivalence in $\cat{E}(b_0)$.
\end{enumerate}
Define a map $x_0\rt x_1$ in $\cat{E}$ to be a weak equivalence (fibration) if it factors (uniquely) as $x_0\rt x\rt x_1$ where the second map is a $p$-cartesian lift of a weak equivalence (fibration) and the first map is a weak equivalence (fibration) in $\cat{E}(px_0)$. This makes $p\colon \cat{E}\rt \cat{B}$ a map between categories of fibrant objects.
\end{lemma}
\begin{proof}
One easily verifies that pullbacks along fibrations exist and that the base change of an (acyclic) fibration remains an (acyclic) fibration. The terminal object in $\cat{E}$ is given by $\ast\in \cat{E}(\ast)$ and each $x\rt \ast$ is a fibration. Condition \ref{it:homotopy invariance} implies that for any weak equivalence $f\colon b_0\rt b_1$, the functor $f^*\colon \cat{E}(b_1)\rt \cat{E}(b_0)$ detects weak equivalences (which are detected in the localization). Using this, one readily verifies that the weak equivalences have the 2-out-of-6 property. Finally, the factorization condition is precisely \ref{it:factorization}.
\end{proof}
\begin{remark}\label{rem:factorization model cat case}
Suppose that $\cat{B}^{\op}\rt \cat{FibCat}$ arises from a diagram of model categories $\cat{N}(b)$ and Quillen adjunctions $f_!\colon \cat{N}(b_0)\leftrightarrows \cat{N}(b_1)\colon f^*$, by letting $\cat{E}(b)\subseteq \cat{N}(b)$ be the full subcategory of fibrant objects. In this case, condition \ref{it:factorization} follows from the following assertion: for every map $f\colon b_0\rt b_1$ in $\cat{B}$, there exists a factorization $f=pi$ with $p\colon b'_0\rt b_1$ a fibration and $i\colon b_0\rt b_0'$ a weak equivalence such that $i_!\colon \cat{N}(b_0)\rt \cat{N}(b'_0)$ preserves all weak equivalences.

Indeed, for a map $g\colon x\rt f^*y$ in $\cat{E}(b_0)$, we can consider the adjoint map $i_!x\rt p^*y$ and factor it into $i_!x\rto{\sim} z$ followed by a fibration $q\colon z\twoheadrightarrow p^*y$. One then obtains the desired factorization $g=i^*(q)j$, where $j\colon x\rt i^*i_!x\rt i^*z$ is a model for the derived unit map and hence an equivalence by condition \ref{it:homotopy invariance}.
\end{remark}
\begin{example}
Consider a model category $\cat{M}$ and a diagram of model categories and left Quillen functors $\cat{M}\rt \cat{ModCat^L}$ sending $b\longmapsto \cat{N}(b)$. Let $p\colon \cat{N}\rt \cat{M}$ be the Grothendieck construction of this diagram and let us say that a map $\alpha\colon x\rt y$ in $\cat{N}$ is:
\begin{itemize}
\item a (trivial) fibration if $p(\alpha)\colon px\to py$ is a (trivial) fibration and the map $x\to p(\alpha)^*y$ is a (trivial) fibration in $\cat{N}(px)$.
\item a (trivial) cofibration if $p(\alpha)\colon px\to py$ is a (trivial) cofibration and the map $p(\alpha)_!x\to y$ is a (trivial) fibration in $\cat{N}(py)$.
\end{itemize}
Harpaz and Prasma have shown that this determines a model structure on $\cat{N}$ if the diagram $b\longmapsto \cat{N}(b)$ satisfies three conditions \cite{har15}: (a) every weak equivalence in $\cat{M}$ induces a Quillen equivalence on fibers, (b) for every trivial cofibration in $\cat{M}$, the associated left Quillen functor preserves weak equivalences and (c) for every trivial fibration in $\cat{M}$, the associated right Quillen functor preserves weak equivalences.

The diagram of model categories indexed by $\cat{M}$ induces a diagram of categories of fibrant objects $\cat{M}^{\mm{fib}, \op}\rt \cat{FibCat}$ sending each fibrant object $b$ in $\cat{M}$ to the full subcategory of fibrant objects in $\cat{N}(b)$. In light of Remark \ref{rem:factorization model cat case}, conditions (a) and (b) from \cite{har15} imply conditions \ref{it:homotopy invariance} and \ref{it:factorization} from Lemma \ref{lem:fibcat grothendieck}. The category of fibrant objects resulting from Lemma \ref{lem:fibcat grothendieck} is then simply the category of fibrant objects associated to the model category $\cat{N}$ from \cite{har15}.
\end{example}
\begin{definition}\label{def:homotopically cart fib}
A map $p\colon \cat{E}\rt \cat{B}$ of categories of fibrant objects is said to be a \emph{homotopy cartesian fibration} if it arises as in Lemma \ref{lem:fibcat grothendieck}. In this situation, we will say that $\alpha\colon x_0\rt x_1$ in $\cat{E}$ is \emph{homotopy $p$-cartesian} if $x_0\rt p(\alpha)^*x_1$ is a weak equivalence in the fiber, and we will write $\cat{E}^{\mm{hocart}}\subseteq \cat{E}$ for the wide subcategory of homotopy cartesian arrows.
\end{definition}

\begin{remark}\label{rem:Ehocart}
If $w$ is a weak equivalence in $\cat{E}$, then a composition $wf$ is homotopy cartesian if and only if $f$ is homotopy cartesian, and likewise for a composition $fw$. In particular, $\cat{E}^\mm{hocart}$ is a relative category (with the same weak equivalences as in $\cat{E}$) whose weak equivalences are generated under the $2$-out-of-$3$ property by the acyclic fibrations in $\cat{E}$.
\end{remark}
\begin{proposition}\label{prop:localizing homotopy cartesian fibrations}
Let $p\colon \cat{E}\rt \cat{B}$ be a homotopy cartesian fibration between categories of fibrant objects. Then the following assertions hold.
\begin{enumerate}
\item The pullback square of relative categories
$$\begin{tikzcd}
\cat{E}(b)\arrow[r]\arrow[d] & \cat{E}\arrow[d, "p"]\\
\{b\}\arrow[r] & \cat{B}
\end{tikzcd}$$
induces a cartesian square upon localization at the weak equivalences.

\item The functor between localizations $\cat{E}^{\mm{hocart}}[W^{-1}]\rt \cat{E}[W^{-1}]$ is a subcategory inclusion, whose image is the wide subcategory of cartesian morphisms for $p\colon \cat{E}[W^{-1}]\rt \cat{B}[W^{-1}]$. 

\item The localization $p\colon \cat{E}[W^{-1}]\rt \cat{B}[W^{-1}]$ is a cartesian fibration.

\item Suppose that $f\colon b_0\rt b_1$ is a morphism in $\cat{B}$ such that $f^*\colon \cat{E}(b_1)\rt \cat{E}(b_0)$ admits a left adjoint $f_!$ that preserves acyclic fibrations. For every $x_0\in \cat{E}(b_0)$, the natural arrow $x_0\rt f_!x_0$ in $\cat{E}$ determines a $p$-cocartesian arrow in $\cat{E}[W^{-1}]$.
\end{enumerate}
\end{proposition}
This proposition is essentially a special case of a result of Hinich \cite{hin16}. We have included a proof that does not rely on Lurie's (un)straightening machinery (in contrast to loc.\ cit.).
\begin{remark}\label{rem:localization fibcat}
For $\cat{M}\in \cat{FibCat}$, its ($\infty$-categorical) localization can be modeled as follows \cite{hor16, nui16}. Let us write $\Span_{\cat{M}}$ for the $(2, 2)$-category whose:
\begin{enumerate}[start=0]
\item objects are the objects of $\cat{M}$.
\item mapping categories $\Span_{\cat{M}}(x_0, x_1)$ are the full subcategories of spans $x_0\lt x_{01}\rt x_1$ where the left morphism $\alpha\colon x_{01}\rt x_0$ is an acyclic fibration and $\beta\colon x_{01}\rt x_1$ is any morphism (and composition by composition of spans). We will abbreviate such a span by $(\alpha, x_{01}, \beta)\colon x_0\dashrightarrow x_1$. 
\end{enumerate}
Taking classifying spaces then determines a $1$-categorical algebra (i.e.\ a Segal space) whose completion models the localization \cite{hor16}. In fact, this description works for any relative category $(\cat{C}, W)$ where $W$ is generated under the $2$-out-of-$3$ property by a class of `acyclic fibrations' which is stable under base change.
\end{remark}
\begin{proof}
Let us start by considering the map of $(2, 2)$-categories $p\colon \Span_{\cat{E}}\rt \Span_{\cat{B}}$. For each $x_0, x_1\in \cat{E}$, the induced functor on mapping $(1, 1)$-categories 
\begin{equation}\label{eq:span functor maps}
\begin{tikzcd}
\Span_{\cat{E}}(x_0, x_1)\arrow[r] & \Span_{\cat{B}}(px_0, px_1)
\end{tikzcd}\end{equation} is a cartesian fibration, whose fiber over a span $(\alpha, b, \beta)\colon px_0\dashrightarrow px_1$ can be identified with $\Span_{\cat{E}(b)}(\alpha^*x_0, \beta^*x_0)$. Using condition \ref{it:homotopy invariance} and Remark \ref{rem:localization fibcat}, it follows that for each morphism $\gamma\colon b'\rto{\sim} b$ in $\Span_{\cat{B}}(px_0, px_1)$, the induced map on fibers 
$$
\gamma^*\colon \Span_{\cat{E}(b)}(\alpha^*x_0, \beta^*x_0)\rt \Span_{\cat{E}(b')}(\gamma^*\alpha^*x_0, \gamma^*\beta^*x_0)
$$
induces an equivalence on classifying spaces. It follows that \eqref{eq:span functor maps} satisfies the conditions of Quillen's Theorem B; consequently, the classifying space functor $|-|\colon \Cat_{(1, 1)}\rt \sS$ preserves all pullbacks along it.

In particular, this implies that the pullback square of $(2, 2)$-categories
$$\begin{tikzcd}
\Span_{\cat{E}(b)}\arrow[r]\arrow[d] & \Span_{\cat{E}}\arrow[d]\\
\{b\}\arrow[r] & \Span_{\cat{B}}
\end{tikzcd}$$
induces a (homotopy) pullback square of Segal spaces upon taking classifying spaces of all mapping categories. To prove assertion (1), we have to show that this also induces a homotopy pullback square on completions. This follows as soon as $\Span_{\cat{E}}\rt \Span_{\cat{B}}$ induces an isofibration on homotopy $(1, 1)$-categories (cf.\ \cite[Corollary 4.4]{har20}). 

To see this, note that by the 2-out-of-6 property, any isomorphism in $\ho(\Span_{\cat{B}})$ arises from a span $(\alpha, b, \beta)\colon b_0\dashrightarrow b_1$ in $\cat{B}$ with $\alpha$ a trivial fibration and $\beta$ a weak equivalence. It therefore suffices to prove the following stronger assertion: 
\begin{itemize}[label={($\ast$)}]
\item For any span $(\alpha, b, \beta)\colon b_0\dashrightarrow b_1$ in $\cat{B}$ with $\alpha$ a trivial fibration and any $x_1\in \cat{E}_{b_1}$, there exists a lift to a span $(\tilde{\alpha}, x_{12}, \tilde{\beta})\colon x_0\dashrightarrow x_1$ with $\tilde{\alpha}$ a trivial fibration and $\tilde{\beta}$ homotopy cartesian.
\end{itemize}
To see this, first take $\beta^*x_1\in \cat{E}_b$ and recall that $\alpha^*\colon \cat{E}_{b_0}\rt \cat{E}_b$ induces an equivalence on localizations. This implies that there exists some $x_0\in \cat{E}_{b_0}$ together with a zig-zag $\alpha^*(x_0)\stackrel{\sim}{\twoheadleftarrow} x\rto{\sim} \beta^*x_1$. The composite $x_0\stackrel{\sim}{\twoheadleftarrow} \alpha^*(x_0) \stackrel{\sim}{\twoheadleftarrow} x\rto{\sim} \beta^*(x_1)\rt x_1$ then provides the desired lift.

This proves (1). For (2) and (3), note that by Remark \ref{rem:Ehocart} and Remark \ref{rem:localization fibcat}, the localization of $\cat{E}^{\mm{hocart}}$ arises from the $(2, 2)$-category $\Span_{\cat{E}^\mm{hocart}}$. Since being a homotopy $p$-cartesian morphism is invariant under weak equivalences, $\Span_{\cat{E}^\mm{hocart}}(x_0, x_1)\rt \Span_{\cat{E}}(x_0, x_1)$ is an inclusion of connected components, so that $\cat{E}^\mm{hocart}[W^{-1}]\rt \cat{E}[W^{-1}]$ is indeed the inclusion of a wide subcategory. Assertions (2) and (3) then follow from ($\ast$) as soon as every arrow in $\cat{E}^\mm{hocart}[W^{-1}]$ is a cartesian arrow.

Let $(\tilde{\alpha}, x_{12}, \tilde{\beta})\colon x_1\dashrightarrow x_2$ be a morphism in $\Span_{\cat{E}^\mm{hocart}}$, so that $\tilde{\beta}$ is homotopy $p$-cartesian, and let $(\alpha, b_{12}, \beta)\colon b_1\dashrightarrow b_2$ denote its image in $\cat{B}$. To see that this map is an equivalence, we have to show that for any $x_0\in \cat{E}$ with image $b_0=px_0$ in $\cat{B}$, the square of categories
\begin{equation}\label{eq:span square}\begin{tikzcd}[column sep=5pc]
\Span_{\cat{E}}(x_0, x_1)\arrow[d]\arrow[r, "{(\tilde{\alpha}, x_{12}, \tilde{\beta})\circ (-)}"] & \Span_{\cat{E}}(x_0, x_2)\arrow[d]\\
 \Span_{\cat{B}}(b_0, b_1)\arrow[r, "{(\alpha, b_{12}, \beta)\circ (-)}"] & \Span_{\cat{B}}(b_0, b_2)
\end{tikzcd}\end{equation}
induces a homotopy cartesian square of classifying spaces. Since the vertical maps are cartesian fibrations satisfying the conditions of Quillen's Theorem B, it suffices to show that for any span $(\gamma, b_{01}, \delta)\colon b_0\dashrightarrow b_1$ in $\Span_{\cat{B}}(b_0, b_1)$, the induced map between fibers induces an equivalence on classifying spaces. To see this, let us denote the composite span $(\alpha, b_{12}, \beta)\circ (\gamma, b_{01}, \delta)$ by
$$\begin{tikzcd}[sep=small]
& & b_{02}\arrow[rd, "\delta'"]\arrow[ld, "\alpha'"{swap}, two heads, "\sim"]\\
& b_{01}\arrow[ld, "\gamma"{swap}, "\sim", two heads]\arrow[rd, "\delta"] & & b_{12}\arrow[ld, "\alpha"{swap}, "\sim", two heads]\arrow[rd, "\beta"]\\
b_0 & & b_1 & & b_2
\end{tikzcd}$$
where $b_{02}=b_{01}\times_{b_1} b_{12}$. Recall from the beginning of the proof that the fibers of the vertical maps in \eqref{eq:span square} are themselves given by categories of spans. Unraveling the definition, the map between the vertical fibers over $(\gamma, b_{01}, \delta)$ can then be identified with the composite
$$\begin{tikzcd}[column sep=0.6pc, row sep=1pc, cramped]
\Span_{\cat{E}(b_{01})}(\gamma^*x_0, \delta^*x_1)\arrow[rd, "(\alpha')^*"{swap}]\arrow[rr] & & \Span_{\cat{E}(b_{02})}\Big((\gamma\alpha')^*x_0, (\beta\delta')^*x_2\Big)\\
& \Span_{\cat{E}(b_{02})}\Big((\gamma\alpha')^*x_0, (\alpha\delta')^*x_2\Big)\arrow[ru, "f\circ (-)"{swap}]
\end{tikzcd}$$
The first map is an equivalence since $(\alpha')^*$ induces an equivalence on localizations. The second map takes the postcomposition with the span $(\delta')^*\alpha^*x_1\lt (\delta')^*x_{12}\rt (\delta')^*\beta^*x_2$, where both maps are weak equivalences since $\tilde{\alpha}$ and $\tilde{\beta}$ were homotopy cartesian. This map induces an equivalence on classifying spaces as well.

Finally, for (4), let us start by recalling that for any cartesian fibration $p\colon \cat{D}\rt \cat{C}$ of $1$-categories, an arrow in $\cat{D}$ is $p$-cocartesian if and only if it is \emph{locally} $p$-cocartesian, i.e.\ cocartesian in the base change of $p$ to a $1$-simplex \cite[Corollary 5.2.2.4]{lur09}. In the present situation, the observations at the beginning of the proof imply that the base change of $\cat{E}[W^{-1}]\rt \cat{B}[W^{-1}]$ to the $1$-simplex $f\colon b_0\rt b_1$ can be computed at the level of $(2, 2)$-categories. Unraveling the definitions as in the proofs of parts (2) and (3), we reduce to showing that for any $x_2\in \cat{E}(b_1)$,
\begin{equation}\label{eq:unit compose span}\begin{tikzcd}[column sep=2.5pc]
\Span_{\cat{E}(b_1)}(f_!x_0, x_2)\arrow[r, "f^*"] & \Span_{\cat{E}(b_1)}(f^*f_!x_0, f^*x_2)\arrow[r, "(-)\circ \eta"] & \Span_{\cat{E}(b_1)}(x_0, f^*x_2)
\end{tikzcd}\end{equation}
induces an equivalence on classifying spaces. Here the second map precomposes spans with the unit $\eta\colon x_0\rt f^*f_!x_0$. Explicitly, \eqref{eq:unit compose span} sends $f_!x_0\lt x_{12}\rt x_2$ to the span $x_0\lt x_0\times_{f^*f_!x_0} f^*x_{12}\rt f^*x_2$. This induces an equivalence on classifying spaces since it admits a left adjoint, sending $x_0\lt x_{02}\rt f^*x_2$ to $f_!x_0\lt f_!x_{02}\rt f_!f^*x_2\rt x_2$.
\end{proof}

\bibliographystyle{abbrv}
\bibliography{bibliography_thesis} 

\end{document}